\documentclass[10pt]{article}

\title{Renormalised energy between boundary vortices in thin-film micromagnetics with Dzyaloshinskii-Moriya interaction}
\author{Radu Ignat\thanks{Institut de Math\'ematiques de Toulouse \& Institut Universitaire de France, UMR 5219, Universit\'e de Toulouse, CNRS, UPS
IMT, F-31062 Toulouse Cedex 9, France. Email: Radu.Ignat@math.univ-toulouse.fr
} \quad  \quad  Fran\c cois L'Official\thanks{Institut de Math\' ematiques de Toulouse, UMR 5219, Universit\' e de Toulouse, CNRS, UPS, IMT, F-31062 Toulouse Cedex 9, France. Email: Francois.Lofficial@math.univ-toulouse.fr; francois.lof@gmail.com}}

\usepackage[mathscr]{euscript}
\usepackage{pstricks} 
\usepackage{graphicx,pst-plot,pst-math} 
\usepackage{amsfonts,dsfont}
\usepackage{amsmath,latexsym}
\usepackage{amssymb,verbatim}
\usepackage{amsthm}

\theoremstyle{definition}
\newtheorem{definition}{Definition}[section]

\theoremstyle{plain}
\newtheorem{theorem}[definition]{Theorem}
\newtheorem{lemma}[definition]{Lemma}
\newtheorem{corollary}[definition]{Corollary}
\newtheorem{proposition}[definition]{Proposition}

\theoremstyle{remark}
\newtheorem{remark}[definition]{Remark}

\newtheorem{notation}[definition]{Notation}

\usepackage{csquotes}
\usepackage{dsfont}
\usepackage[title]{appendix}

\usepackage{colortbl}

\usepackage{tikz}
\usepackage{pgfplots}

\usepackage{mathrsfs}
\usetikzlibrary{arrows}

\newcommand{\N}{\mathbb{N}}
\newcommand{\Z}{\mathbb{Z}}

\newcommand{\R}{\mathbb{R}}
\newcommand{\C}{\mathbb{C}}

\newcommand{\phid}{\phi}
\newcommand{\phir}{\varphi}
\renewcommand{\epsilon}{\varepsilon}
\renewcommand{\d}{\mathrm{d}}
\newcommand{\dr}{\partial}

\newcommand{\eps}{\varepsilon}

\newcommand{\lv}{\left\vert}
\newcommand{\rv}{\right\vert}
\newcommand{\lV}{\left\Vert}
\newcommand{\rV}{\right\Vert}
\newcommand{\lb}{\left\lbrace}
\newcommand{\rb}{\right\rbrace}

\newcommand{\m}{V}
\newcommand{\Dirac}{\boldsymbol\delta}

   {\vspace{1cm}}
\makeatother

\begin{document}

\maketitle
\begin{abstract}
We consider a three-dimensional micromagnetic model with Dzyaloshinskii-Moriya interaction in a thin-film regime for boundary vortices. In this regime, we prove a dimension reduction result: the nonlocal three-dimensional model reduces to a local two-dimensional Ginzburg-Landau type model in terms of the averaged magnetization in the thickness of the film. This reduced model captures the interaction between boundary vortices (so-called renormalised energy), that we determine by a~$\Gamma$-convergence result at the second order and then we analyse its minimisers. They nucleate two boundary vortices whose position depends on the Dzyaloshinskii-Moriya interaction.
\end{abstract}

\tableofcontents

\section{Introduction}
\label{SECTION_INTRODUCTION}

We consider a ferromagnetic sample of cylindrical shape of thickness~$t$:
\begin{equation*}
\Omega_t^\ell
= \Omega^\ell \times (0,t) \subset \mathbb{R}^3,
\end{equation*}
where the horizontal section~$\Omega^\ell \subset \mathbb{R}^2$ is a bounded, simply connected~$C^{1,1}$ domain of typical length~$\ell$ (for example,~$\Omega^\ell$ can be assumed to be a disk of diameter~$\ell$).
The magnetization~$m$ is a three-dimensional unit-length vector field
\begin{equation*}
m \colon \Omega_t^\ell \rightarrow \mathbb{S}^2,
\end{equation*}
where~$\mathbb{S}^2$ is the unit sphere in~$\mathbb{R}^3$, the constraint~$\left\vert m \right\vert = 1$ yielding the non-convexity of the problem. 
We consider the micromagnetic energy with Dzyaloshinskii-Moriya interaction (in the absence of anisotropy and applied magnetic field):
\begin{align}
\label{DEF_general_micromagnetic_energy}
E(m)
&
= A^2 \int_{\Omega_t^\ell} \left\vert \nabla m \right\vert^2 \mathrm{d} x
+ \int_{\Omega_t^\ell} D : \nabla m \wedge m \ \mathrm{d} x
+ \int_{\mathbb{R}^3} \left\vert \nabla u \right\vert^2 \mathrm{d} x.
\end{align}
The first term in~$E(m)$ is the exchange energy, generated by short-range interactions between magnetic spins, where the exchange length~$A>0$ is an intrinsic parameter of the ferromagnetic material, typically of the order of nanometers.
The second term is the Dzyaloshinskii-Moriya interaction~(DMI), see~\cite{Dzyalo}, taking into account the antisymmetric properties of the material. The~DMI density is given here by
\begin{align}
\label{DEF_DMI_density}
D : \nabla m \wedge m
= \sum_{j=1}^3 D_j \cdot \partial_j m \wedge m,
\end{align}
where~$D=(D_1,D_2,D_3) \in \mathbb{R}^{3 \times 3}$ is the~DMI tensor,~$\cdot$ denotes the inner product in~$\mathbb{R}^3$, and~$\wedge$ denotes the cross product in~$\mathbb{R}^3$.
The third term in~$E(m)$ is called magnetostatic or stray-field energy, it is a nonlocal term generated by long-range spins interactions carried by the stray-field potential~$u \in H^1(\mathbb{R}^3,\mathbb{R})$ satisfying
\begin{align*}
\Delta u = \nabla \cdot (m \mathds{1}_{\Omega_t^\ell}) \ \text{ in the distributional sense in } \mathbb{R}^3,
\end{align*}
where~$m$ is extended by~0 outside the sample~$\Omega_t^\ell$.
For more details, we refer to~\cite{Aharoni}, \cite{Dzyalo} or~\cite{HS98}.

\subsection{Nondimensionalisation}
\label{INTROSECTION_nondimensionalisation}

The multiscale aspect of the micromagnetic model is carried by the~DMI tensor~$D$, the exchange length~$A$, the horizontal length~$\ell$ and the thickness~$t$ of the sample.
We introduce the dimensionless parameters
\begin{equation*}
h = \frac{t}{\ell} \ \ \text{ and } \ \ \eta = \frac{A}{\ell};
\end{equation*}
the thin-film regimes correspond to the limit~$h \rightarrow 0$ (sometimes denoted by~$h \ll 1$).

In order to study the micromagnetic energy in a thin-film regime appropriate to boundary vortices, we nondimensionalize in length and set~$\widehat{D}=\frac{1}{\ell}D$,
\begin{equation*}
\Omega_h = \frac{\Omega_t^\ell}{\ell} = \Omega \times (0,h) \subset \mathbb{R}^3,
\end{equation*}
where~$\Omega = \frac{\Omega^\ell}{\ell} \subset \mathbb{R}^2$ is a bounded, simply connected~$C^{1,1}$ domain of typical length~$1$.
To \linebreak each~$x=(x_1,x_2,x_3) \in \Omega_t^\ell$, we associate~$\widehat{x} = \frac{x}{\ell} \in \Omega_h$, and then~$m_h \colon \Omega_h \rightarrow \mathbb{S}^2$ and~$u_h \colon \mathbb{R}^3 \rightarrow \mathbb{R}$ given by
\begin{equation*}
m_h(\widehat{x})=m(x),
\ \ \
u_h(\widehat{x})=\frac{1}{\ell}u(x),
\end{equation*}
that satisfy
\begin{align}
\label{EQ_Maxwell_u_h}
\Delta u_h = \nabla \cdot (m_h\mathds{1}_{\Omega_h}) \ \text{ in the distributional sense in } \mathbb{R}^3.
\end{align}
The micromagnetic energy \eqref{DEF_general_micromagnetic_energy} can then be written in terms of~$m_h$:
\begin{align*}
E(m)
=\widehat{E}(m_h)
& = \ell^3 \left[
\eta^2 \int_{\Omega_h} \vert \nabla m_h \vert^2 \mathrm{d} \widehat{x}
+ \int_{\Omega_h} \widehat{D} : \nabla m_h \wedge m_h \ \mathrm{d} \widehat{x}
+ \int_{\mathbb{R}^3} \vert \nabla u_h \vert^2 \mathrm{d} \widehat{x}
\right].
\end{align*}
For simplicity of the notations, we write~$x$ instead of~$\widehat{x}$ in the following.

\subsection{Thin-film regime for boundary vortices}
\label{SECTION_regime}

We consider the thin-film regime studied by Ignat-Kurzke~\cite{IK22} with appropriate scaling of the~DMI tensor~$\widehat{D}=\frac{D}{\ell}$. More precisely, we consider here the regime
\begin{equation}
\begin{split}
& h \ll 1,
\ \ \eta \ll 1,
\ \ \frac{1}{\lv \log h \rv} \ll \epsilon \ll 1, \ \ \frac{\widehat{D}_{13}}{\eta^2} \rightarrow 2\delta_1,
\ \ \frac{\widehat{D}_{23}}{\eta^2} \rightarrow 2\delta_2,
\\
& \frac{1}{\eta^2} \left( \sum_{j,k=1}^2 \vert \widehat{D}_{jk} \vert + \sum_{k=1}^3 \vert \widehat{D}_{3k} \vert \right) \ll \sqrt{\vert \log \epsilon \vert},
\end{split}
\label{DEF_regime3D}
\end{equation}
where $\delta_1,\delta_2 \in \mathbb{R}$, $\widehat{D}=(\widehat{D}_{jk})_{1 \leqslant j,k \leqslant 3} \in \mathbb{R}^{3 \times 3}$ and
\begin{align*}
\epsilon = \frac{\eta^2}{h \lv \log h \rv}
\end{align*}
that corresponds to the core size of the boundary vortices.
The parameters~$\eta=\eta(h)$,~$\epsilon=\epsilon(h)$ and~$\widehat{D}=\widehat{D}(h)$ are assumed to be functions in~$h$; our first aim is to prove a~$\Gamma$-convergence result at the first order in the limit~$h \rightarrow 0$ that quantifies the number of boundary vortices in the domain.

The regime~\eqref{DEF_regime3D} implies\footnote{We use that inequalities~$a \ll b \ll 1$ imply~$a \lv \log a \rv \ll b \lv \log b \rv \ll 1$, and then choose~$a=\frac{1}{\lv \log h \rv}$ and~$b=\epsilon$.}:
\begin{align*}
\frac{\log \lv \log h \rv}{\lv \log h \rv} \ll \epsilon \lv \log \epsilon \rv \ll 1.
\end{align*}
In fact, our second aim is to prove a~$\Gamma$-convergence result at the second order as~$h \rightarrow 0$, capturing the interaction energy between boundary vortices; for that, we restrict to a narrower regime than~\eqref{DEF_regime3D}:
\begin{equation}
\begin{split}
& h \ll 1,
\ \ \eta \ll 1,
\ \ \frac{\log \lv \log h \rv}{\lv \log h \rv} \ll \epsilon \ll 1,
\ \ \lv \frac{\widehat{D}_{13}}{\eta^2} - 2\delta_1 \rv + \lv \frac{\widehat{D}_{23}}{\eta^2} - 2\delta_2 \rv \ll \frac{1}{\sqrt{\vert \log \epsilon \vert}},
\\
& \frac{1}{\eta^2} \left( \sum_{j,k=1}^2 \vert \widehat{D}_{jk} \vert + \sum_{k=1}^3 \vert \widehat{D}_{3k} \vert \right) \ll \frac{1}{\sqrt{\vert \log \epsilon \vert}}.
\end{split}
\label{DEF_regime3D_log}
\end{equation}

We assume that~$\Omega \subset \R^2$ is a bounded, simply connected~$C^{1,1}$  domain. We consider the three-dimensional rescaled energy
\begin{align*}
E_h(m_h)
& = \frac{\widehat{E}(m_h)}{\ell^3h\eta^2\lv \log \epsilon \rv},
\end{align*}
for maps~$m_h \colon \Omega_h =\Omega \times (0,h) \rightarrow \mathbb{S}^2$ and~$u_h \colon \mathbb{R}^3 \rightarrow \mathbb{R}$ satisfying~\eqref{EQ_Maxwell_u_h}.
More precisely, the energy~$E_h(m_h)$ is given by
\begin{align}
\label{DEF_Eh}
E_h(m_h)
& = \frac{1}{\lv \log \epsilon \rv}
\left( \frac{1}{h} \int_{\Omega_h} \lv \nabla m_h \rv^2 \d x
+ \frac{1}{h \eta^2} \int_{\Omega_h} \widehat{D} : \nabla m_h \wedge m_h \ \d x
+ \frac{1}{h \eta^2} \int_{\R^3} \lv \nabla u_h \rv^2 \d x \right).
\end{align}

\subsection{Dimension reduction}
\label{INTROSECTION_dimension_reduction}

The thin-film regime is characterized by the assumption~$h=t/\ell \rightarrow 0$, i.e., the variations of the magnetization~$m_h$ in the thickness direction~$x_3$ are strongly penalised. Therefore~$m_h$ is expected to behave as its~$x_3$-average~$\overline{m}_h \colon \Omega \rightarrow \R^3$, that is, 
\begin{align}
\label{DEF_mh_mean}
\overline{m}_h(x') = \frac{1}{h} \int_0^h m_h(x',x_3) \ \mathrm{d} x_3 \ \ \ \text{ for every } x' \in \Omega,
\end{align}
where~$\lv \overline{m}_h \rv \leqslant 1$ in~$\Omega$.
With this in mind, we assume for a bit that~$m_h$ does not depend on~$x_3$, and that
\begin{align*}
m_h \text{ varies on length scales } \gg h.
\end{align*}
We are interested in the scaling of the stray-field energy in this regime; within our assumption (i.e.,~$m_h \equiv \overline{m}_h$), the Maxwell equation~\eqref{EQ_Maxwell_u_h} implies
\begin{align*}
\Delta u_h
= \nabla \cdot (m_h\mathds{1}_{\Omega_h})
= \left( \nabla' \cdot m'_h \right) \mathds{1}_{\Omega_h} - (m_h \cdot \nu) \mathds{1}_{\partial\Omega_h}
\end{align*}
in the distributional sense in~$\mathbb{R}^3$, where~$\nu$ is the outer unit normal vector on~$\partial\Omega_h$, and \textbf{prime~$'$ corresponds to~2D quantities} such as the in-plane magnetization~$m_h'=(m_{h,1},m_{h,2})$ and the in-plane divergence~$\nabla' \cdot m'_h=\partial_1m_{h,1}+\partial_2m_{h,2}$. In other words,~$u_h \in H^1(\R^3)$ is the solution of the transmission problem
\begin{align*}
\left\lbrace
\begin{array}{rcll}
\Delta u_h & = & \nabla' \cdot m'_h & \text{ in } \Omega_h,
\\ \Delta u_h & = & 0 & \text{ in } \mathbb{R}^3 \setminus \Omega_h,
\\ \left[ \frac{\partial u_h}{\partial \nu} \right] & = & m_h \cdot \nu & \text{ on } \partial \Omega_h,
\end{array}
\right.
\end{align*}
where~$[a]=a^+-a^-$ stands for the jump of~$a$ with respect to the outer unit normal vector~$\nu$ on~$\partial \Omega_h$.
From~\cite{GarciaCervera} (see also~\cite{Aharoni66},~\cite{DKMO04},~\cite[Section~2.1.2]{IgnatHDR}), we can express the stray-field energy by considering the Fourier transform in the horizontal variables:
\begin{align*}
\int_{\mathbb{R}^3} \left\vert \nabla u_h \right\vert^2 \mathrm{d} x
& =
h \int_{\mathbb{R}^2} \frac{\left\vert \xi' \cdot \mathcal{F}(m'_h\mathds{1}_{\Omega})(\xi') \right\vert^2}{\left\vert \xi' \right\vert^2} \left( 1 - g_h \left( \left\vert \xi' \right\vert \right) \right) \mathrm{d} \xi'
\\
& \quad
+ h \int_{\mathbb{R}^2} \left\vert \mathcal{F} (m_{h,3} \mathds{1}_{\Omega})(\xi') \right\vert^2 g_h \left( \left\vert \xi' \right\vert \right) \mathrm{d} \xi',
\end{align*}
where~$\mathcal{F}$ stands for the Fourier transform in~$\mathbb{R}^2$, i.e., for every~$f \colon \mathbb{R}^2 \rightarrow \mathbb{C}$ and for every~$\xi' \in \mathbb{R}^2$,
\begin{align*}
\mathcal{F}(f)(\xi')
= \int_{\mathbb{R}^2} f(x')e^{- 2i\pi x' \cdot \xi'} \mathrm{d} x',
\ \ \text{and} \ \
g_h(\vert \xi' \vert) = \frac{1-e^{-2\pi h \vert \xi' \vert}}{2\pi h \vert \xi' \vert}.
\end{align*}
In the asymptotics~$h \rightarrow 0$, we have~$g_h \left( \left\vert \xi' \right\vert \right) \rightarrow 1$ and~$1-g_h \left( \left\vert \xi' \right\vert \right) \approx \pi h \lv \xi' \rv$, hence (see~\cite{Carbou01},~\cite{DKMO04},~\cite{KS05})
\begin{equation*}
\begin{split}
\int_{\mathbb{R}^3} \left\vert \nabla u_h \right\vert^2 \mathrm{d} x
& \approx h^2 \left\Vert (\nabla' \cdot m'_h)\mathds{1}_\Omega \right\Vert_{\dot{H}^{-1/2}(\mathbb{R}^2)}^2
\\
& \quad + \dfrac{1}{2\pi} h^2 \left\vert \log h \right\vert \int_{\partial \Omega} \left( m'_h \cdot \nu' \right)^2 \mathrm{d} \mathcal{H}^1
+ h \int_{\Omega} m_{h,3}^2 \ \mathrm{d} x',
\end{split}
\end{equation*}
where~$\nu'$ is the outer unit normal vector on~$\partial\Omega$. Hence, the stray-field energy is asymptotically decomposed in three terms in the thin-film regime.
The first term is nonlocal and penalizes the volume charges, as an homogeneous~$\dot{H}^{-1/2}$ seminorm, and favors N\'eel walls. The second term takes into account the lateral charges on the cylindrical sample and favors boundary vortices. The third term penalizes the surface charges on the top and bottom of the cylinder, and leads to interior vortices. For more details on the different types of singularities that may occur in thin-film regimes, we refer to~\cite{DKMO04} or~\cite{IgnatHDR}.

Our dimension reduction result in Theorem~\ref{THM_reduction3D2D} shows that the stray-field energy of a general magnetization~$m_h$ depending on~$x_3$ reduces in our regime to the last two local terms in the~$x_3$-average~$\overline{m}_h$ (while the nonlocal term~$h^2 \lV (\nabla' \cdot \overline{m}'_h)\mathds{1}_\Omega \rV_{\dot{H}^{-1/2}(\R^2)}^2$ becomes negligible).
More precisely, for a given vector~$\delta=(\delta_1,\delta_2) \in \R^2$, we consider the two-dimensional reduced energy functional for a two-dimensional map~$v \colon \Omega \subset \R^2 \rightarrow \R^2$ (standing for the in-plane average~$\overline{m}'_h$):
\begin{align}
\label{DEF_Eepsetadelta}
E_{\epsilon,\eta}^\delta(v)
& = \int_\Omega \left\vert \nabla' v \right\vert^2 \d x
+ 2 \int_\Omega \delta \cdot \nabla' v \wedge v \ \d x
\nonumber
\\
& \qquad
+ \frac{1}{\eta^2} \int_\Omega \left( 1- \vert v \vert^2 \right)^2 \d x
+ \frac{1}{2\pi\epsilon} \int_{\partial \Omega} (v \cdot \nu')^2 \d \mathcal{H}^1,
\end{align}
where~$\epsilon,\eta>0$,~$\nabla'=(\dr_1,\dr_2)$ and~$\nu'$ is the outer unit normal vector on~$\dr\Omega$. 
Note that by identifying the complex plane~$\C \simeq \R^2$, we have~$\lv \nabla'v \rv^2+2\delta \cdot \nabla'v \wedge v=\lv (\nabla'-i\delta)v \rv^2-\lv \delta \rv^2\lv v \rv^2$. Therefore, the functional~$E_{\epsilon,\eta}^\delta$ is similar to the Ginzburg-Landau model with magnetic potential, combining a boundary penalisation (favoring boundary vortices) with an interior penalisation (favoring interior vortices). The case~$\delta=0$ was analysed in Ignat-Kurzke~\cite{IK21} (see also Moser~\cite{Moser03}). 

\begin{theorem}
\label{THM_reduction3D2D}
Let~$\Omega_h=\Omega \times (0,h)$ with~$\Omega \subset \R^2$ a bounded, simply connected~$C^{1,1}$ domain. In the regime~\eqref{DEF_regime3D}, consider a family of magnetizations~$\lb m_h \colon \Omega_h \rightarrow \mathbb{S}^2 \rb_{h \downarrow 0}$ that satisfies
\begin{align*}
\limsup\limits_{h \rightarrow 0} E_h(m_h) < +\infty
\end{align*}
and let~$\overline{m}_h=(\overline{m}'_h,\overline{m}_{h,3}) \colon \Omega \rightarrow \R^3$ be the average of~$m_h$ in~\eqref{DEF_mh_mean}.
Then\footnote{We write~$a=o_h(b)$ if~$a/b \rightarrow 0$ as~$h \rightarrow 0$.}
\begin{align*}
E_h(m_h) \geqslant \frac{1}{\lv \log \epsilon \rv} E_{\epsilon,\eta}^\delta(\overline{m}'_h) - o_h(1) \ \ \text{ as } h \rightarrow 0.
\end{align*}
Moreover, in the more restrictive regime~\eqref{DEF_regime3D_log}, we have
\begin{align*}
E_h(m_h) \geqslant \frac{1}{\lv \log \epsilon \rv} E_{\epsilon,\eta}^\delta(\overline{m}'_h) - o_h \left( \frac{1}{\lv \log \epsilon \rv} \right) \ \ \text{ as } h \rightarrow 0.
\end{align*}
Furthermore, if~$m_h$ is independent of~$x_3$ (i.e.,~$m_h=m_h(x')$,~$x'=(x_1,x_2)$) and~$m_{h,3}=0$ (i.e.,~$m_h=(m'_h,0)$, $|m'_h|=1$), then in the regime~\eqref{DEF_regime3D}, we have
\begin{align*}
E_h(m_h) = \frac{1}{\lv \log \epsilon \rv} E_{\epsilon,\eta}^\delta(\overline{m}'_h) - o_h(1) \ \ \text{ as } h \rightarrow 0,
\end{align*}
while in the regime~\eqref{DEF_regime3D_log}, we have
\begin{align*}
E_h(m_h) = \frac{1}{\lv \log \epsilon \rv} E_{\epsilon,\eta}^\delta(\overline{m}'_h) - o_h \left( \frac{1}{\lv \log \epsilon \rv} \right) \ \ \text{ as } h \rightarrow 0.
\end{align*}
\end{theorem}

Therefore, in order to determine the asymptotic behaviour of~$E_h$ as~$h \rightarrow 0$, we need to analyse the behaviour of the two-dimensional reduced functional~$E_{\epsilon,\eta}^\delta$ in the regime~\eqref{DEF_regime3D}, respectively in the regime~\eqref{DEF_regime3D_log}.

\subsection{Gamma-convergence of the two-dimensional reduced functional~$E_{\varepsilon,\eta}^\delta$}
\label{INTROSECTION_GC_2D_energy}

In this section, we prove~$\Gamma$-convergence results (at the first and second order) for the reduced functional~$E_{\epsilon,\eta}^\delta$ defined in~\eqref{DEF_Eepsetadelta} in the regimes~\eqref{DEF_regime3D} and~\eqref{DEF_regime3D_log}, respectively. These results are reminiscent from Ignat-Kurzke~\cite{IK21}, who studied the energy functional~$E_{\epsilon,\eta}^0$ (i.e. the case~$\delta=0$). 
At the first order, the~$\Gamma$-limit energy is expected to quantify the number of boundary vortices detected by the global Jacobian introduced in~\cite{IK21},~\cite{IK22}.
More precisely, for a two-dimensional \linebreak map~$v \in H^1(\Omega,\R^2)$ defined in a bounded~$C^{1,1}$ domain~$\Omega \subset \R^2$, the \textbf{global Jacobian} of~$v$ is given by the linear operator~$\mathcal{J}(v) \colon W^{1,\infty}(\Omega) \rightarrow \R$ defined as
\begin{align*}
\langle \mathcal{J}(v),\zeta \rangle
= - \int_\Omega v \wedge \nabla' v \cdot \nabla'^\perp \zeta \ \d x',
\end{align*}
for every Lipschitz function~$\zeta \colon \Omega \rightarrow \R$, where~$v \wedge \nabla' v = (v \wedge \dr_1 v , v \wedge \dr_2 v)$, $\nabla'^\perp = (-\dr_2,\dr_1)$ and~$\langle \cdot,\cdot \rangle$ stands for the algebraic dual pairing between~$(W^{1,\infty}(\Omega))^\ast$ and~$W^{1,\infty}(\Omega)$. In particular, the global Jacobian has zero average, i.e.~$\langle \mathcal{J}(v),1 \rangle = 0$.
Moreover,
\begin{align*}
\mathcal{J}(v)
=2 \ \mathrm{jac}(v)+\mathcal{J}_\mathrm{bd}(v)
\end{align*}
where~$\mathrm{jac}(v) = \det (\nabla' v) = \dr_1 v \wedge \dr_2 v \in L^1(\Omega)$ is the interior Jacobian of~$v$, and~$\mathcal{J}_\mathrm{bd}(v)$ is the boundary Jacobian of~$v$. For example, if~$v \in C^1(\overline{\Omega},\R^2)$, then~$\langle \mathcal{J}_\mathrm{bd}(v),\zeta \rangle = -\int_{\partial \Omega} v \wedge \partial_\tau v \zeta \ \d \mathcal{H}^1$ where~$\tau=(\nu')^\perp$ is the tangent unit vector at~$\partial\Omega$ (for more details, see~\cite[Section~2]{IK21}).

For a fixed vector~$\delta \in \R^2$, we start by analysing the energy functional~$E_{\epsilon,\eta}^\delta$ in the asymptotic regime:
\begin{align}
\label{DEF_regime2D}
& \eta \ll 1,
\ \ 
\epsilon \ll 1,
\ \ 
\lv \log \epsilon \rv \ll \lv \log \eta \rv,
\end{align}
which is implied\footnote{\label{foot} Indeed, the regime~\eqref{DEF_regime3D} implies~$h \ll \eta^2 \ll h \lv \log h \rv \ll 1$ by using the definition of~$\epsilon$, hence~$\lv \log h \rv \sim \lv \log \eta \rv$. It follows from~\eqref{DEF_regime3D} that~$\frac{1}{\epsilon} \ll \lv \log \eta \rv$, hence~\eqref{DEF_regime2D} is satisfied because~$\lv \log \epsilon \rv \ll \frac{1}{\epsilon}$.} by the regime~\eqref{DEF_regime3D}.
For a family~$E_{\epsilon,\eta}^\delta(v_\epsilon) \leqslant C \lv \log \epsilon \rv$, we prove that the global Jacobian~$\mathcal{J}(v_\epsilon)$ concentrates on Dirac masses at the boundary vortices (up to a diffuse measure carried by the curvature of~$\partial\Omega$). Moreover, the lower bound of~$E_{\epsilon,\eta}^\delta(v_\epsilon)$ quantifies the number of these boundary vortices. We also show~$L^p(\partial\Omega)$ compactness of the traces~$v_\epsilon \vert_{\partial\Omega}$ converging to a map in~$BV(\partial\Omega,\mathbb{S}^1)$.

\begin{theorem}[Compactness at the boundary and lower bound at the first order]
\label{THM_1.2_IK21+DMI}
Let~$\delta \in \R^2$,~$\Omega \subset \mathbb{R}^2$ be a bounded, simply connected~$C^{1,1}$ domain and~$\kappa$ be the curvature of~$\partial \Omega$. Assume~$\epsilon \rightarrow 0$ and~$\eta=\eta(\epsilon) \rightarrow 0$ in the regime~\eqref{DEF_regime2D}. 
Let~$(v_\epsilon)_\epsilon$ be a family in~$H^1(\Omega,\mathbb{R}^2)$ such that
\begin{align}
\label{EQ_limsup_order1_2D}
\limsup\limits_{\epsilon \rightarrow 0} \frac{1}{\left\vert \log \epsilon \right\vert} E_{\epsilon,\eta}^\delta(v_\epsilon) < +\infty.
\end{align}
Then the following statements hold.

\begin{itemize}
\item[(i)] \textbf{Compactness of global Jacobians and~$L^p(\partial \Omega)$-compactness of~$v_\epsilon \vert_{\partial \Omega}$.}
\\
For a subsequence~$\epsilon \rightarrow 0$,~$(\mathcal{J}(v_\epsilon))_\epsilon$ converges to a measure~$J$ on the closure~$\overline{\Omega}$, in the sense that
\begin{align}
\label{EQ_convergence_jacobian_2D}
\lim\limits_{\epsilon \rightarrow 0} \left( \sup\limits_{\left\vert \nabla' \zeta \right\vert \leqslant 1 \text{ in } \Omega} \left\vert \langle \mathcal{J}(v_\epsilon)-J,\zeta \rangle \right\vert \right) =0.
\end{align}
Moreover,~$J$ is supported on~$\partial \Omega$ and has the form\footnote{Note that the condition~$\langle \mathcal{J},1 \rangle=0$ and the Gauss-Bonnet formula~$\int_{\partial \Omega} \kappa \ \d \mathcal{H}^1=2\pi$ yield the constraint:~$\sum_{j=1}^N d_j=2$.}
\begin{align}
\label{EQ_limit_jacobian_2D}
J=-\kappa \mathcal{H}^1 \llcorner \partial \Omega + \pi \sum\limits_{j=1}^N d_j\Dirac_{a_j}
\end{align}
for~$N \geqslant 1$ distinct boundary vortices~$a_j \in \partial \Omega$ carrying the multiplicities~$d_j \in \mathbb{Z} \setminus \left\lbrace 0 \right\rbrace$, for~$j \in \left\lbrace 1,...,N \right\rbrace$, such that~$\sum_{j=1}^N d_j = 2$. 
\\
Furthermore, for a subsequence, the family of traces~$(v_\epsilon \vert_{\partial \Omega})_\epsilon$ converges to~$e^{i\phir_0} \in BV(\partial \Omega, \mathbb{S}^1)$ in~$L^p(\partial \Omega,\R^2)$, for every~$p \in [1,+\infty)$, where~$\phir_0$ is a~BV lifting of the tangent field~$\pm \tau$ on~$\partial \Omega$ determined (up to a constant in~$\pi\Z$) by
\begin{equation*}
\partial_\tau\phir_0 = \kappa \mathcal{H}^1 \llcorner \dr\Omega - \pi \sum\limits_{j=1}^N d_j \Dirac_{a_j} \ \ \text{ as measure on } \partial \Omega.
\end{equation*}
\item[(ii)] \textbf{Energy lower bound at the first order.}
\\
If~$(\mathcal{J}(v_\epsilon))_\epsilon$ satisfies the convergence assumption in~\textit{(i)} as~$\epsilon \rightarrow 0$, then the energy lower bound at the first order is the total mass of the measure~$J + \kappa \mathcal{H}^1 \llcorner \partial \Omega$ on~$\partial \Omega$:
\begin{equation*}
\liminf\limits_{\epsilon \rightarrow 0} \frac{1}{\left\vert \log \epsilon \right\vert} E_{\epsilon,\eta}^\delta(v_\epsilon) \geqslant \pi \sum\limits_{j=1}^N \left\vert d_j \right\vert = \left\vert J + \kappa \mathcal{H}^1 \llcorner \partial \Omega \right\vert (\partial \Omega).
\end{equation*}
\end{itemize}
\end{theorem}

In order to get a lower bound at the second order for~$E_{\epsilon,\eta}^\delta$, we introduce the following renormalised energy, in the spirit of Brezis-Bethuel-H\'elein~\cite{BBH}, that captures the interaction energy between boundary vortices. In comparison with Ignat-Kurzke~\cite{IK21},~\cite{IK22}, our renormalised energy includes the contribution of the~DMI energy.

\begin{definition}
\label{DEF_phi0}
Let~$\delta \in \R^2$,~$\Omega \subset \mathbb{R}^2$ be a bounded, simply connected~$C^{1,1}$ domain and~$\kappa$ be the curvature of~$\partial \Omega$.
Consider~$\phir_0 \colon \partial \Omega \rightarrow \mathbb{R}$ to be a~BV function such that~$e^{i\phir_0} \cdot \nu'=0$ \linebreak in~$\partial \Omega \setminus \left\lbrace a_1,...,a_N \right\rbrace$, and
\begin{align*}
\partial_\tau\phir_0 = \kappa \mathcal{H}^1 \llcorner \dr\Omega - \pi \sum\limits_{j=1}^N d_j \Dirac_{a_j} \ \ \text{ as measure on } \partial \Omega,
\end{align*}
for~$N \geqslant 1$ distinct points~$a_j \in \partial \Omega$ carrying the multiplicities~$d_j \in \lb \pm 1 \rb$ for~$j \in \left\lbrace 1,...,N \right\rbrace$ with~$\sum_{j=1}^Nd_j=2$.
If~$\phir_\ast$ is the harmonic extension of~$\phir_0$ to~$\Omega$, then the renormalised energy of~$\left\lbrace(a_j,d_j)\right\rbrace_{j \in \left\lbrace 1,...,N \right\rbrace}$ is defined as
\begin{align}
\label{DEF_renormalised_energy}
W_\Omega^\delta(\left\lbrace(a_j,d_j)\right\rbrace)=\liminf\limits_{r \rightarrow 0} \left( \int_{\Omega \setminus \bigcup_{j=1}^N B_r(a_j)} \left( \left\vert \nabla' \phir_\ast \right\vert^2 - 2 \delta \cdot \nabla' \phir_\ast \right) \d x -N\pi\log \frac{1}{r} \right),
\end{align}
where~$B_r(a_j)$ is the disk of center~$a_j$ and radius~$r>0$.
\end{definition}

We will see that the~$\liminf$ in~\eqref{DEF_renormalised_energy} is in fact a limit (see Proposition~\ref{prop-expr-renenergy-psi-R}). In the following two theorems, we prove the asymptotic expansion at the second order for the energy~$E_{\epsilon,\eta}^\delta$ by the~$\Gamma$-convergence method. In particular, under a prescribed level of energy (see~\eqref{EQ_limsup_order2_2D}), we prove that the boundary vortices have multiplicities~$\pm 1$ and that the second order energy is given by the renormalised energy~\eqref{DEF_renormalised_energy} up to a constant depending on the number of boundary vortices.

\begin{theorem}[Compactness in the interior and lower bound at the second order]
\label{THM_1.4_IK21+DMI}
Let~$\delta \in \R^2$,~$\Omega \subset \mathbb{R}^2$ be a bounded, simply connected~$C^{1,1}$ domain and~$\kappa$ be the curvature of~$\partial \Omega$. Assume~$\epsilon \rightarrow 0$ and~$\eta=\eta(\epsilon) \rightarrow 0$ in the regime~\eqref{DEF_regime2D}.
Let~$(v_\epsilon)_\epsilon$ be a family in~$H^1(\Omega,\mathbb{R}^2)$ satisfying~\eqref{EQ_limsup_order1_2D} and the convergence at~\textit{(i)} in Theorem~\ref{THM_1.2_IK21+DMI} to the measure~$J$ given at~\eqref{EQ_limit_jacobian_2D} as~$\epsilon \rightarrow 0$. In addition, we assume the following more precise bound than~\eqref{EQ_limsup_order1_2D}:
\begin{align}
\label{EQ_limsup_order2_2D}
\limsup\limits_{\epsilon \rightarrow 0} \left( E_{\epsilon,\eta}^\delta(v_\epsilon) - \pi \left\vert \log \epsilon \right\vert \sum\limits_{j=1}^N \left\vert d_j \right\vert \right) < +\infty.
\end{align}
Then the following statements hold.
\begin{itemize}
\item[(i)] \textbf{Single multiplicity and second order lower bound.}
\\
The multiplicities satisfy~$d_j \in \left\lbrace \pm 1 \right\rbrace$ for every~$j \in \left\lbrace 1,...,N \right\rbrace$, so we have\footnote{Recall that~$\sum_{j=1}^N d_j=2$.}~$\sum_{j=1}^N \left\vert d_j \right\vert =~N$; moreover, the following second order energy lower bound holds:
\begin{align*}
\liminf\limits_{\epsilon \rightarrow 0} \left( E_{\epsilon,\eta}^\delta(v_\epsilon) -N\pi\lv \log \epsilon \right\vert \right) \geqslant W_\Omega^\delta(\left\lbrace (a_j,d_j) \right\rbrace) + N\gamma_0
\end{align*}
where~$\gamma_0=\pi\log\frac{e}{4\pi}$ and~$W_\Omega^\delta(\lb(a_j,d_j)\rb)$ is the renormalised energy defined in~\eqref{DEF_renormalised_energy}.

\item[(ii)] \textbf{$W^{1,q}(\Omega)$-weak compactness of maps~$v_\epsilon$.}
\\
For any~$q \in [1,2)$, $(v_\epsilon)_\epsilon$ is uniformly bounded in~$W^{1,q}(\Omega,\mathbb{R}^2)$.
Moreover, for a subsequence,~$(v_\epsilon)_\epsilon$ converges weakly in~$W^{1,q}(\Omega,\mathbb{R}^2)$, for every~$q \in [1,2)$, and strongly in~$L^p(\Omega,\mathbb{R}^2)$, for every~$p \in [1,+\infty)$, to~$e^{i\widehat{\phir}_0}$, where~$\widehat{\phir}_0 \in W^{1,q}(\Omega)$ is an extension (not necessarily harmonic) to~$\Omega$ of the lifting~$\phir_0 \in BV(\partial \Omega,\pi \mathbb{Z})$ determined in Theorem \ref{THM_1.2_IK21+DMI}\textit{(i)}.
\end{itemize}
\end{theorem}

Now we state the upper bound part in the~$\Gamma$-expansion where the recovery sequence~$v_\epsilon$ can be chosen with values into~$\mathbb{S}^1$:

\begin{theorem}[Upper bound]
\label{THM_1.5_IK21+DMI}
Let~$\delta \in \R^2$,~$\Omega \subset \mathbb{R}^2$ be a bounded, simply connected~$C^{1,1}$ domain and~$\kappa$ be the curvature of~$\partial \Omega$. Let~$\left\lbrace a_j \in \partial\Omega \right\rbrace_{1 \leqslant j \leqslant N}$ be $N \geqslant 1$ distinct points and~$d_j \in \mathbb{Z} \setminus \left\lbrace 0 \right\rbrace$ be the corresponding multiplicities, for~$j \in \left\lbrace 1,...,N \right\rbrace$, that satisfy~$\sum_{j=1}^N d_j=2$. Assume~$\epsilon \rightarrow 0$ and~$\eta=\eta(\epsilon) \rightarrow 0$ in the regime~\eqref{DEF_regime2D}.
Then we can construct a family~$(v_\epsilon)_\epsilon$ in~$H^1(\Omega,\mathbb{S}^1)$ such that~$(\mathcal{J}(v_\epsilon))_\epsilon$ converges as in~\eqref{EQ_convergence_jacobian_2D} to the measure
\begin{equation*}
J=-\kappa \mathcal{H}^1 \llcorner \partial \Omega + \pi \sum_{j=1}^N d_j \Dirac_{a_j}.
\end{equation*}
Furthermore,~$(v_\epsilon)_\epsilon$ converges strongly to~$e^{i\phir_\ast}$ in~$L^p(\Omega,\R^2)$ and in~$L^p(\partial \Omega,\R^2)$, for \linebreak every~$p \in [1,+\infty)$, where~$\phir_\ast$ is the harmonic extension to~$\Omega$ of a boundary lifting~$\phir_0$ of the tangent field~$\pm\tau$ on~$\partial\Omega$ that satisfies~$\partial_\tau \phir_0 = -J$ as measure on~$\partial \Omega$, and the energy of~$v_\epsilon$ satisfies
\begin{equation*}
\lim\limits_{\epsilon \rightarrow 0} \frac{1}{\left\vert \log \epsilon \right\vert} E_{\epsilon,\eta}^\delta(v_\epsilon) = \pi \sum\limits_{j=1}^N \left\vert d_j \right\vert.
\end{equation*}
Furthermore, if~$d_j \in \left\lbrace \pm 1 \right\rbrace$ for every~$j \in \left\lbrace 1,...,N \right\rbrace$, then~$v_\epsilon$ can be chosen such that
\begin{equation*}
\lim\limits_{\epsilon \rightarrow 0} \left( E_{\epsilon,\eta}^\delta(v_\epsilon) - N\pi \left\vert \log \epsilon \right\vert \right) = W_\Omega^\delta(\left\lbrace (a_j,d_j) \right\rbrace) + N\gamma_0
\end{equation*}
where~$\gamma_0=\pi\log\frac{e}{4\pi}$ and~$W_\Omega^\delta(\lb(a_j,d_j)\rb)$ is the renormalised energy defined in~\eqref{DEF_renormalised_energy}.
\end{theorem}

\subsection{Minimisation of the renormalised energy}
\label{INTROSECTION_minimisation_renormalised_energy}

In this section, we compute explicitely the renormalised energy~$W_\Omega^\delta$ in terms of the location of boundary vortices~($a_j$), their multiplicities~($d_j$) and the reduced~DMI vector~$\delta \in \R^2$. Then we will analyse the minimisers of~$W_\Omega^\delta$ showing the role of~$\delta$ in the optimal location of boundary vortices. These results are reminiscent from~\cite{IK22} in the absence of the~DMI vector~$\delta$.

\begin{theorem}
\label{THM_6_IK22+DMI}
Let~$\delta \in \R^2$.
\begin{itemize}
\item[\textit{(i)}] We denote by~$B_1$ the unit disk in~$\R^2$. Let~$\lb a_j \in \partial B_1 \rb_{1 \leqslant j \leqslant N}$ be~$N \geqslant 2$ distinct points \linebreak and~$d_j \in \lb \pm 1 \rb$ be the corresponding multiplicities, for~$j \in \lb 1,...,N \rb$, that satisfy~$\sum_{j=1}^N d_j =~2$. Then the renormalised energy of~$\lb(a_j,d_j)\rb$ in~$B_1$ satisfies
\begin{align}
\label{EQ_WB1delta}
W_{B_1}^\delta(\lb (a_j,d_j) \rb)
& = - 2\pi \sum_{1\leqslant j<k \leqslant N} d_jd_k \log \lv a_j-a_k \rv
+2\pi \sum_{j=1}^N d_j \delta \cdot a_j^\perp.
\end{align}
\item[\textit{(ii)}] Let~$\Omega \subset \R^2$ be a bounded, simply connected~$C^{1,1}$ domain with~$\kappa$ the curvature of~$\dr \Omega$ and~$\nu$ the outer unit normal vector on~$\partial\Omega$. Let~$\Phi \colon \overline{B_1} \rightarrow \overline{\Omega}$ be a~$C^1$ conformal diffeomorphism with inverse~$\Psi = \Phi^{-1}$. Let~$\lb a_j \in \partial \Omega \rb_{1 \leqslant j \leqslant N}$ be $N \geqslant 2$ distinct points and~$d_j \in \lb \pm 1 \rb$ be the corresponding multiplicities, for~$j \in \lb 1,...,N \rb$, that satisfy~$\sum_{j=1}^N d_j =2$. Then the renormalised energy of~$\lb(a_j,d_j)\rb$ in~$\Omega$ satisfies
\begin{equation}
\label{EQ_WOmegadelta}
\begin{split}
& W_\Omega^\delta(\lb (a_j,d_j) \rb)
\\
& \qquad = - 2\pi \sum_{1\leqslant j<k \leqslant N} d_jd_k \log \lv \Psi(a_j)-\Psi(a_k) \rv + \pi \sum_{j=1}^N (d_j-1)\log \lv \partial_z\Psi(a_j) \rv
\\
& \qquad \quad + \int_{\dr \Omega}(\kappa + 2 \delta^\perp \cdot \nu')(w) \left( \sum_{j=1}^N d_j \log \lv \Psi(w)-\Psi(a_j) \rv - \log \lv \partial_z\Psi(w) \rv \right) \d \mathcal{H}^1(w)
\end{split}
\end{equation}
where~$\partial_z\Psi$ stands for the complex derivative of~$\Psi \colon \overline{\Omega} \subset \mathbb{C} \rightarrow \overline{B_1} \subset \mathbb{C}$.
\end{itemize}
\end{theorem}

By Theorems~\ref{THM_1.2_IK21+DMI},~\ref{THM_1.4_IK21+DMI} and~\ref{THM_1.5_IK21+DMI} together with the constraint~$\sum_{j=1}^N d_j=2$, any minimiser of~$E_{\epsilon,\eta}^\delta$ nucleates two boundary vortices of multiplicity~1 in the limit~$\epsilon \rightarrow 0$ in the regime~\eqref{DEF_regime2D}. For such configurations, we prove the existence of optimal minimisers of the renormalised energy:

\begin{corollary}
\label{COR_7_IK22+DMI}
Let~$\delta \in \R^2$ and~$\Omega \subset \R^2$ be a bounded, simply connected~$C^{1,1}$ domain. There exists a pair~$(a_1^\ast,a_2^\ast) \in \dr \Omega \times \dr \Omega$ of distinct points such that
\begin{align*}
W_\Omega^\delta(\lb (a_1^\ast,1),(a_2^\ast,1) \rb)
= \min \lb W_\Omega^\delta(\lb (a_1,1),(a_2,1) \rb) : a_1 \neq a_2 \in \dr \Omega \rb.
\end{align*}
\end{corollary}

We will determine now the location of the optimal pair~$(a_1^\ast,a_2^\ast)$ in the unit disk~$\Omega=B_1$. Recall that for the model with~$\delta=0$ studied by Ignat and Kurzke~\cite{IK21},~\cite{IK22}, the pair of boundary vortices that minimises the renormalised energy corresponds to two diametrically opposed points on~$\dr B_1$ and is unique (up to a rotation). For our model with~$\delta \in \R^2$, the location of the optimal pair~$(a_1^\ast,a_2^\ast)$ is influenced by the Dzyaloshinskii-Moriya interaction. Indeed, the renormalised energy for two vortices of multiplicity~1 at~$a_1,a_2 \in \dr B_1$ is
\begin{equation*}
W_{B_1}^\delta(\lb(a_1,1),(a_2,1)\rb)
= -2\pi \log \lv a_1-a_2 \rv +2\pi \delta \cdot (a_1^\perp+a_2^\perp).
\end{equation*}
For~$\delta \neq 0$, the minimisation of the renormalised energy is different due to the competition between the two terms in~$W_{B_1}^\delta$. In the next theorem, we show that the minimal configuration for the points~$a_1,a_2 \in \dr B_1$ is unique (up to switching~$a_1$ and~$a_2$) and the vortices are not diametrically opposed if~$\delta \neq 0$, but symmetric with respect to~$\delta^\perp$: the distance between~$a_1$ and~$a_2$ becomes smaller as~$\vert\delta\vert$ gets larger.

\begin{theorem}
\label{THM_minimisers_unit_disk}
Let~$B_1$ be the unit disk in~$\R^2$ and~$\delta = \lv \delta \rv e^{i\theta} \in \R^2 \setminus \lb (0,0) \rb$ for some~$\theta \in \R$. \linebreak Then the pair of distinct points~$(a_1^\ast,a_2^\ast) \in \dr B_1 \times \dr B_1$ that minimises the renormalised \linebreak energy~$W_{B_1}^\delta(\lb(\cdot,1),(\cdot,1)\rb)$ in Corollary~\ref{COR_7_IK22+DMI} is unique (up to switching~$a_1^\ast$ and~$a_2^\ast$) and given by
\begin{align*}
a_1^\ast=e^{i(\theta+\theta_\delta)} \ \ \text{ and } \ \ a_2^\ast=e^{i(\theta+\pi-\theta_\delta)}
\end{align*}
where~$\theta_\delta=\arcsin \left( \sqrt{1+\frac{1}{16\lv \delta \rv^2}} - \frac{1}{4\lv \delta \rv} \right)$ if~$\delta \neq 0$. In particular,~$a_1^\ast$ and~$a_2^\ast$ are symmetric with respect to~$\delta^\perp$.
\end{theorem}

\begin{figure}[!h]
\definecolor{qqwuqq}{rgb}{0,0.39215686274509803,0}
\definecolor{qqzzqq}{rgb}{0,0.6,0}
\definecolor{qqttcc}{rgb}{0,0.2,0.8}
\definecolor{ffqqqq}{rgb}{1,0,0}
\begin{center}
\begin{minipage}{7cm}
\begin{center}
\begin{tikzpicture}[line cap=round,line join=round,x=2.5cm,y=2.5cm]

\draw[->,color=black,line width=.3mm] (-1.2,0) -- (1.2,0);
\draw[->,color=black,line width=.3mm] (0,-1.2) -- (0,1.2);
\clip(-1.2,-1.2) rectangle (1.2,1.2);

\draw [line width=2pt,color=qqttcc] (0,0) circle (2.5cm);

\draw [line width=1.5pt,color=qqzzqq,domain=-2.062212512572243:1.8872658142793661] plot(\x,{(-0--0.7653668647301796*\x)/1.8477590650225735});
\draw [line width=1.5pt,color=qqzzqq,domain=-2.062212512572243:1.8872658142793661] plot(\x,{(-0-0.9238795325112867*\x)/0.3826834323650898});

\draw [->,line width=2.5pt,color=qqwuqq] (0,0) -- (0.9238795325112867,0.3826834323650898);
\draw [->,line width=2.5pt,color=qqwuqq] (0,0) -- (-0.3826834323650898,0.9238795325112867);

\draw [fill=ffqqqq] (0.8328871559753063,0.5534428474660827) circle (3pt);
\draw [fill=ffqqqq] (-0.9802833463957761,-0.19759696551085734) circle (3pt);

\draw [color=qqwuqq](0.6,0.25) node[above] {$\frac{\delta}{\vert\delta\vert}$};
\draw [color=qqwuqq](-0.35,0.4) node[above] {$\frac{\delta^\perp}{\vert\delta\vert}$};
\draw [color=ffqqqq](0.86,0.5534428474660827) node[right] {$a_1^\ast$};
\draw [color=ffqqqq](-1,-0.19759696551085734) node[left] {$a_2^\ast$};
\end{tikzpicture}
\end{center}
\end{minipage}
\begin{minipage}{7cm}
\begin{center}
\begin{tikzpicture}[line cap=round,line join=round,x=2.5cm,y=2.5cm]

\draw[->,color=black,line width=.3mm] (-1.2,0) -- (1.2,0);
\draw[->,color=black,line width=.3mm] (0,-1.2) -- (0,1.2);
\clip(-1.2,-1.2) rectangle (1.2,1.2);

\draw [line width=2pt,color=qqttcc] (0,0) circle (2.5cm);

\draw [line width=1.5pt,color=qqzzqq,domain=-2.062212512572243:1.8872658142793661] plot(\x,{(-0--0.7653668647301796*\x)/1.8477590650225735});
\draw [line width=1.5pt,color=qqzzqq,domain=-2.062212512572243:1.8872658142793661] plot(\x,{(-0-0.9238795325112867*\x)/0.3826834323650898});

\draw [->,line width=2.5pt,color=qqwuqq] (0,0) -- (0.9238795325112867,0.3826834323650898);
\draw [->,line width=2.5pt,color=qqwuqq] (0,0) -- (-0.3826834323650898,0.9238795325112867);

\draw [fill=ffqqqq] (-0.1692161564254841,0.9855789630489209) circle (3pt);
\draw [fill=ffqqqq] (-0.577255676471914,0.8165634598614813) circle (3pt);

\draw [color=qqwuqq](0.6,0.25) node[above] {$\frac{\delta}{\vert\delta\vert}$};
\draw [color=qqwuqq](-0.35,0.4) node[above] {$\frac{\delta^\perp}{\vert\delta\vert}$};
\draw [color=ffqqqq](-0.1692161564254841,1) node[above] {$a_1^\ast$};
\draw [color=ffqqqq](-0.577255676471914,0.8165634598614813) node[above left] {$a_2^\ast$};
\end{tikzpicture}
\end{center}
\end{minipage}
\end{center}
\caption{Minimising pair~$(a_1^\ast,a_2^\ast)$ for~$\delta=\frac{1}{10} e^{i\pi/8}$ (left) and~$\delta=10e^{i\pi/8}$ (right). The distance between the boundary vortices~$a_1^\ast$ and~$a_2^\ast$ becomes smaller as~$\vert\delta\vert$ gets larger.}
\label{figure1}
\end{figure}
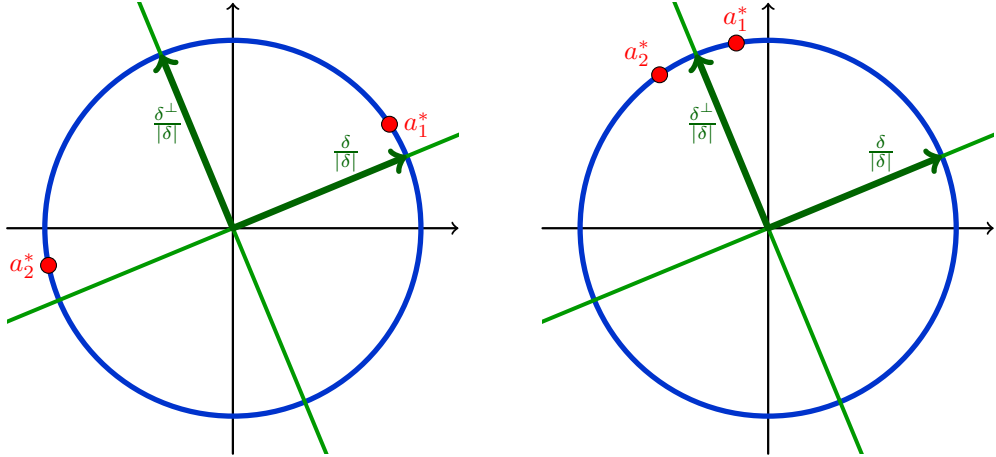

The following result proves the asymptotic behaviour of minimisers\footnote{The existence of minimisers of~$E_{\epsilon,\eta}^\delta$ is proved in Lemma~\ref{lem-minimizer-EepsetaD}.} of~$E_{\epsilon,\eta}^\delta$ as~$\epsilon \rightarrow 0$.

\begin{corollary}
\label{COR_1.6_IK21+DMI}
Let~$\delta \in \R^2$,~$\Omega \subset \R^2$ be a bounded, simply connected~$C^{1,1}$ domain and~$\kappa$ be the curvature of~$\dr \Omega$. Assume~$\epsilon \rightarrow 0$ and~$\eta=\eta(\epsilon) \rightarrow 0$ in the regime~\eqref{DEF_regime2D}.
For every family~$(v_\epsilon)_\epsilon$ of minimisers of~$E_{\epsilon,\eta}^\delta$ on~$H^1(\Omega,\R^2)$, there exists a subsequence~$\epsilon \rightarrow 0$ such that~$(\mathcal{J}(v_\epsilon))_\epsilon$ converges as in~\eqref{EQ_convergence_jacobian_2D} to the measure
\begin{align*}
J=-\kappa \mathcal{H}^1 \llcorner \dr \Omega + \pi (\Dirac_{a_1} + \Dirac_{a_2}),
\end{align*}
where~$a_1$ and~$a_2$ are two distinct points in~$\dr \Omega$ that minimise the renormalised energy (for the multiplicities~$d_1=d_2=1$), i.e.
\begin{align*}
W_\Omega^\delta(\lb (a_1,1),(a_2,1) \rb)
= \min \lb W_\Omega^\delta(\lb (\tilde{a}_1,1),(\tilde{a}_2,1) \rb) : \tilde{a}_1 \neq \tilde{a}_2 \in \dr \Omega \rb.
\end{align*}
Moreover, for a subsequence~$\epsilon \rightarrow 0$,~$(v_\epsilon)_\epsilon$ converges weakly to~$e^{i\phir_\ast}$ in~$W^{1,q}(\Omega,\R^2)$, for \linebreak every~$q \in [1,2)$, where~$\phir_\ast$ is the harmonic extension to~$\Omega$ of a boundary lifting~$\phir_0 \in BV(\dr\Omega,\pi\Z)$ that satisfies~$\partial_\tau\phir_0=-{J}$ on~$\partial\Omega$ and~$e^{i\phir_0} \cdot \nu' = 0$ in~$\dr \Omega \setminus \lb a_1,a_2 \rb$.
\\
Furthermore, there holds the following second order expansion of the minimal energy:
\begin{align*}
\min E_{\epsilon,\eta}^\delta = 2\pi \lv \log \epsilon \rv + W_\Omega^\delta(\lb (a_1,1),(a_2,1) \rb) + 2\gamma_0 + o_\epsilon(1)
\ \ \text{ as } \epsilon \rightarrow 0,
\end{align*}
where~$\gamma_0=\pi\log\frac{e}{4\pi}$.
\end{corollary}

In the case of the unit disk~$\Omega=B_1$, combining Theorem~\ref{THM_6_IK22+DMI} and Corollary~\ref{COR_1.6_IK21+DMI}, we deduce in the regime~\eqref{DEF_regime2D} for a minimiser~$v_\epsilon$ of~$E_{\epsilon,\eta}^\delta$ that
\begin{equation*}
\lim\limits_{\epsilon \rightarrow 0} \left( E_{\epsilon,\eta}^\delta(v_\epsilon) -2\pi\lv \log \epsilon \rv \right)
= -2\pi \log \lv a_1^\ast-a_2^\ast \rv
+2\pi\delta \cdot \left( (a_1^\ast)^\perp+(a_2^\ast)^\perp \right)
+2\pi\log\frac{e}{4\pi}
\end{equation*}
where~$(a_1^\ast,a_2^\ast)$ is given in Theorem~\ref{THM_minimisers_unit_disk}.

\subsection{Gamma-convergence of the three-dimensional energy~$E_h$}
\label{INTROSECTION_GC_3D_energy}

We come back to the nonlocal three-dimensional model for maps~$m_h \colon \Omega_h \subset \R^3 \rightarrow \mathbb{S}^2$ and study the three-dimensional energy~$E_h$ given at~\eqref{DEF_Eh}.
Thanks to the dimension reduction in Theorem~\ref{THM_reduction3D2D} and the results obtained for the two-dimensional reduced model studied in Section~\ref{INTROSECTION_GC_2D_energy}, we prove the~$\Gamma$-expansion at the second order of~$E_h$ as~$h \rightarrow 0$ summarized in the following three theorems:

\begin{theorem}
\label{THM_9_order1_IK22+DMI}
Let~$\Omega_h=\Omega \times (0,h)$ with~$\Omega \subset \R^2$ a bounded, simply connected~$C^{1,1}$ domain, and~$\kappa$ be the curvature of~$\partial \Omega$. In the regime~\eqref{DEF_regime3D}, consider a family of magnetizations~$\lb m_h \colon \Omega_h \rightarrow \mathbb{S}^2 \rb_{h \downarrow 0}$ that satisfies
\begin{align}
\label{EQ_limsup_order1_3D}
\limsup\limits_{h \rightarrow 0} E_h(m_h) < +\infty
\end{align}
and let~$\overline{m}_h=(\overline{m}'_h,\overline{m}_{h,3}) \colon \Omega \rightarrow \R^3$ be the average of~$m_h$ in~\eqref{DEF_mh_mean}.
\begin{itemize}
\item[\textit{(i)}] \textbf{Compactness of the global Jacobians and~$L^p(\dr\Omega)$-compactness of the traces~$\overline{m}_h \vert_{\dr\Omega}$.}
\\ For a subsequence~$h \rightarrow 0$,~$\left( \mathcal{J}(\overline{m}'_h) \right)_h$ converges to a measure~$J$ on the closure~$\overline{\Omega}$, in the sense that
\begin{equation}
\label{EQ_convergence_jacobian_3D}
\lim\limits_{h \rightarrow 0}
\left(
\sup\limits_{\lv \nabla' \zeta \rv \leqslant 1 \text{ in } \Omega} \lv \langle \mathcal{J}(\overline{m}'_h)-J,\zeta \rangle \rv
\right)
= 0.
\end{equation}
Moreover,~$J$ is supported on~$\dr\Omega$ and has the form~\eqref{EQ_limit_jacobian_2D} for~$N \geqslant 1$ distinct boundary vortices~$a_j \in \dr\Omega$ carrying the multiplicities~$d_j \in \Z \setminus \lb 0 \rb$, for~$j \in \lb 1,...,N \rb$, with~$\sum_{j=1}^N d_j =2$.
\\
Furthermore, for a subsequence,~$\left( \overline{m}_h \vert_{\dr\Omega} \right)_h$ converges to~$(e^{i\phir_0},0) \in BV(\dr\Omega,\mathbb{S}^1 \times \lb 0 \rb)$ \linebreak in~$L^p(\dr\Omega,\R^3)$, for every~$p \in [1,+\infty)$, where~$\phir_0 \in BV(\dr\Omega,\pi\Z)$ is a lifting of the tangent field~$\pm\tau$ on~$\dr\Omega$ determined (up to a constant in~$\pi\Z$) by~$\partial_\tau\phir_0=-\mathcal{J}$ as measure on~$\partial\Omega$.

\item[\textit{(ii)}] \textbf{Energy lower bound at the first order.}
\\ If~$\left( \mathcal{J}(\overline{m}'_h) \right)_h$ satisfies the convergence assumption in~\textit{(i)} as~$h \rightarrow 0$, then the energy lower bound at the first order is the total mass of the measure~$J+\kappa \mathcal{H}^1 \llcorner \dr \Omega$ on~$\dr \Omega$:
\begin{align*}
\liminf\limits_{h \rightarrow 0} E_h(m_h) \geqslant \pi \sum_{j=1}^N \lv d_j \rv = \lv J+\kappa \mathcal{H}^1 \llcorner \dr \Omega \rv(\dr\Omega).
\end{align*}
\end{itemize}
\end{theorem}

As in Theorem~\ref{THM_1.4_IK21+DMI}, within a more precise bound similar to~\eqref{EQ_limsup_order2_2D}, we prove the following lower bound at the second order for~$E_h$ in the narrower regime~\eqref{DEF_regime3D_log}.

\begin{theorem}
\label{THM_9_order2_IK22+DMI}
Let~$\Omega_h=\Omega \times (0,h)$ with~$\Omega \subset \R^2$ a bounded, simply connected~$C^{1,1}$ domain, and~$\kappa$ be the curvature of~$\partial \Omega$. In the regime~\eqref{DEF_regime3D_log}, consider a family of magnetizations~$\lb m_h \colon \Omega_h \rightarrow \mathbb{S}^2 \rb_{h \downarrow 0}$ that satisfies~\eqref{EQ_limsup_order1_3D} and the convergence at~\textit{(i)} in Theorem~\ref{THM_9_order1_IK22+DMI} to the measure~$J$ given at~\eqref{EQ_limit_jacobian_2D} as~$h \rightarrow 0$. In addition, we assume the following more precise bound than~\eqref{EQ_limsup_order1_3D}:
\begin{align}
\label{EQ_limsup_order2_3D}
\limsup\limits_{h \rightarrow 0} \lv \log \epsilon \rv \left( E_h(m_h) - \pi \sum_{j=1}^N \lv d_j \rv \right) < +\infty.
\end{align}
\begin{itemize}
\item[\textit{(i)}] \textbf{Single multiplicity and second order lower bound.}
\\ The multiplicities satisfy\footnote{Recall that~$\sum_{j=1}^N d_j=2$.}~$d_j \in \lb \pm 1 \rb$ for every~$j \in \lb 1,...,N \rb$, so we have~$\sum_{j=1}^N \lv d_j \rv = N$, and there holds the following second order energy lower bound:
\begin{align*}
\liminf\limits_{h \rightarrow 0} \lv \log \epsilon \rv \left( E_h(m_h)-N\pi \right) \geqslant W_\Omega^\delta(\lb (a_j,d_j) \rb)+ N\gamma_0
\end{align*}
where~$\gamma_0=\pi\log\frac{e}{4\pi}$ and~$W_\Omega^\delta(\lb(a_j,d_j)\rb)$ is the renormalised energy defined in~\eqref{DEF_renormalised_energy}.

\item[\textit{(ii)}] \textbf{$L^p(\Omega)$-compactness of the rescaled magnetizations.}
\\
For a subsequence, the family of rescaled magnetizations~$\lb \widetilde{m}_h \colon \Omega_1 \rightarrow \mathbb{S}^2 \rb_{h \downarrow 0}$, defined \linebreak as~$\widetilde{m}_h(x',x_3)=m_h(x',hx_3)$, converges strongly in~$L^p(\Omega_1,\R^3)$, for every~$p \in [1,+\infty)$, to a map~$\widetilde{m}=(\widetilde{m}',0) \in W^{1,q}(\Omega_1,\R^3)$, for every~$q \in [1,2)$, such that~$\lv \widetilde{m} \rv = \lv \widetilde{m}' \rv = 1$ and~$\dr_3 \widetilde{m} = 0$, i.e.~$\widetilde{m}=\widetilde{m}(x') \in W^{1,q}(\Omega,\mathbb{S}^1 \times \lb 0 \rb)$. Furthermore, the global Jacobian~$\mathcal{J}(\widetilde{m}')$ coincides with the measure~$J$ on~$\overline{\Omega}$ given at~\eqref{EQ_limit_jacobian_2D}. 
\end{itemize}
\end{theorem}

For the upper bound part in the~$\Gamma$-expansion, the recovery sequence can be chosen to be invariant in the thickness direction~$x_3$ and to be in-plane, i.e., they take values into~$\mathbb{S}^1 \times \lb 0 \rb$.

\begin{theorem}
\label{THM_10_IK22+DMI}
Let~$\Omega_h=\Omega \times (0,h)$ with~$\Omega \subset \R^2$ a bounded, simply connected~$C^{1,1}$ domain, and~$\kappa$ be the curvature of~$\partial \Omega$. Let~$\lb a_j \in \partial \Omega \rb_{1 \leqslant j \leqslant N}$ be~$N \geqslant 1$ distinct points and~$d_j \in \Z \setminus \lb 0 \rb$ be the corresponding multiplicities, for~$j \in \lb 1,...,N \rb$, that satisfy~$\sum_{j=1}^N d_j = 2$.
Then in the regime~\eqref{DEF_regime3D}, we can construct a family~$\lb m_h=(m'_h,0) \rb_{h>0}$ of~$H^1(\Omega_h,\mathbb{S}^1 \times \lb 0 \rb)$ with the following properties.
\begin{itemize}
\item[\textit{(i)}] For every~$h>0$,~$m_h$ is independent of~$x_3$.
\item[\textit{(ii)}] $(\mathcal{J}(m_h'))_h$ converges to~$J=-\kappa\mathcal{H}^1 \llcorner \dr\Omega + \pi \sum_{j=1}^N d_j\Dirac_{a_j}$ as in~\eqref{EQ_convergence_jacobian_3D}.
\item[\textit{(iii)}] We have~$\lim\limits_{h \rightarrow 0} E_h(m_h) = \pi \sum_{j=1}^N \lv d_j \rv$.
\end{itemize}
Furthermore in the regime~\eqref{DEF_regime3D_log}, if~$d_j \in \lb \pm 1 \rb$ for every~$j \in \lb 1,...,N \rb$, then~$\lb m_h \rb_{h \downarrow 0}$ can be chosen such that
\begin{align*}
\lim\limits_{h \rightarrow 0} \lv \log \epsilon \rv \left( E_h(m_h) - N\pi \right) = W_\Omega^\delta(\lb (a_j,d_j) \rb) + N\gamma_0
\end{align*}
where~$\gamma_0=\pi\log\frac{e}{4\pi}$ and~$W_\Omega^\delta(\lb(a_j,d_j)\rb)$ is the renormalised energy defined in~\eqref{DEF_renormalised_energy}.
\end{theorem}

We conclude with the asymptotic behaviour of minimisers of~$E_h$:\footnote{The existence of minimisers of~$E_h$ is proved in~Corollary~\ref{cor:mini}.}

\begin{corollary}
\label{COR_11_IK22+DMI}
Let~$\Omega_h=\Omega \times (0,h)$ with~$\Omega \subset \R^2$ a bounded, simply connected~$C^{1,1}$ domain, and~$\kappa$ be the curvature of~$\partial \Omega$. In the regime~\eqref{DEF_regime3D}, for every family~$\lb m_h \rb_h$ of minimisers of~$E_h$ on~$H^1(\Omega_h,\mathbb{S}^2)$, there exists a subsequence~$h \rightarrow 0$ such that the global Jacobians~$\mathcal{J}(\overline{m}'_h)$ of the in-plane averages~$\overline{m}'_h$ converge as in~\eqref{EQ_convergence_jacobian_3D} to the measure
\begin{align*}
J=-\kappa \mathcal{H}^1 \llcorner \dr\Omega +\pi(\Dirac_{a_1}+\Dirac_{a_2})
\end{align*}
where $a_1$ and $a_2$ are two points on $\dr\Omega$. Moreover, we have
\begin{align*}
\lim\limits_{h \rightarrow 0} E_h(m_h) = 2\pi.
\end{align*}
Furthermore, in the regime~\eqref{DEF_regime3D_log}, we have that $a_1 \neq a_2$, the pair~$(a_1,a_2)$ minimises the renormalised energy (for the multiplicities~$d_1=d_2=1$) over the set $\lb(\tilde{a}_1,\tilde{a}_2) \in \dr\Omega \times \dr\Omega : \tilde{a}_1 \neq \tilde{a}_2\rb$, i.e.
\begin{align*}
W_\Omega^\delta(\lb(a_1,1),(a_2,1)\rb)
= \min \lb W_\Omega^\delta(\lb(\tilde{a}_1,1),(\tilde{a}_2,1)\rb) : \tilde{a}_1 \neq \tilde{a}_2 \in \dr\Omega \rb,
\end{align*}
and
\begin{align*}
\lim\limits_{h \rightarrow 0} \lv \log \epsilon \rv ( E_h(m_h)-2\pi ) = W_\Omega^\delta(\lb(a_1,1),(a_2,1)\rb) + 2\gamma_0
\end{align*}
where~$\gamma_0=\pi\log\frac{e}{4\pi}$.
\end{corollary}

In the case of the unit disk~$\Omega=B_1$, combining Theorem~\ref{THM_6_IK22+DMI} and Corollary~\ref{COR_11_IK22+DMI}, we deduce that the minimal energy in the regime~\eqref{DEF_regime3D_log} is given by
\begin{align*}
\min_{H^1(\Omega_h,\mathbb{S}^2)} E_h
= 2\pi\lv \log \epsilon \rv -2\pi \log \lv a_1^\ast-a_2^\ast \rv
+2\pi\delta \cdot \left( (a_1^\ast)^\perp+(a_2^\ast)^\perp \right)
+2\pi\log\frac{e}{4\pi}+o_h(1)
\end{align*}
where~$(a_1^\ast,a_2^\ast)$ is the optimal pair given in Theorem~\ref{THM_minimisers_unit_disk}.

\subsection{Related models}
\label{INTROSECTION_related_models}

Our thin-film regime~\eqref{DEF_regime3D}, that is the same as in~\cite{IK22} (where $D=0$), favors boundary vortices while taking into account the Dzyaloshinskii-Moriya interaction $D\neq 0$. However, assuming that~$h \ll 1$, there are other convergence rates of~$\eta(h) \rightarrow 0$ as~$h \rightarrow 0$ (different than~\eqref{DEF_regime3D}), that lead to different thin-film regimes and thus to different effects on the magnetizations in the limit~$h \rightarrow 0$. More precisely, three types of singular pattern of the magnetization can occur: N\'eel walls, interior vortices and boundary vortices. The choice of the asymptotic regime corresponds in general to the energy ordering of these three patterns (for more details, see~\cite{DKMO04}).

The micromagnetic energy has been studied in many thin-film regimes, often without taking into account the Dzyaloshinskii-Moriya interaction.
In the regime of small films, where~$\eta>0$ is fixed, Gioia-James~\cite{GioiaJames} proved that the~$\Gamma$-limit of the micromagnetic energy is minimised by any constant and in-plane magnetization. Recently, Davoli et. al.~\cite{DDFPR20} studied the~$\Gamma$-convergence of the micromagnetic energy with~DMI, and showed that in the limiting energy, the~DMI energy contributes to increase the shape anisotropy of the thin film.

For relatively small films, i.e.,~$\eta^2 \gg h \lv \log h \rv$, Kohn-Slastikov~\cite{KS05} (see also Carbou~\cite{Carbou01}) derived a~$\Gamma$-limit that reduces to a boundary penalisation term (coming from the stray-field energy, see Section~\ref{INTROSECTION_dimension_reduction}) for magnetizations that are constant and in-plane.
For slightly larger films, \linebreak where~$\eta^2=\alpha h \lv \log h \rv$ with~$0<\alpha<+\infty$, also studied by Kohn-Slastikov~\cite{KS05}, there is a competition between exchange and stray-field energies. The limiting energy, for maps~$m \in H^1(\Omega,\mathbb{S}^1)$ is composed of the exchange term and a boundary penalisation term (as in the previous regime). The limiting magnetizations are in-plane but no longer constant. Recently, L'Official~\cite{LO23} considered the model of Kohn-Slastikov with~DMI, assuming (after a rescaling) that the~DMI density is of the same order as the exchange and boundary energy, and derived a~$\Gamma$-limit for limiting in-plane magnetizations. We also mention the recent work of Di~Fratta et. al.~\cite{DiFratta2023} for the study of this regime with~DMI.

In the regime~$h \ll \eta^2 \ll h \lv \log h \rv$, Kurzke~\cite{Kurzke06},~\cite{Kurzke06-2},~\cite{Kurzke07} and Ignat-Kurzke~\cite{IK22} showed that the limiting magnetization develops boundary vortices due to topological arguments. In the narrower regime~$h \log \lv \log h \rv \ll \eta^2 \ll h \lv \log h \rv$, Ignat-Kurzke~\cite{IK22} derive a~$\Gamma$-expansion at the second order of the limiting energy. Our aim in this paper consists in taking into account the DMI in these thin-film regimes, and more precisely to study the influence of the DMI on the location of the boundary vortices.

Several other thin-film regimes have been investigated before. In the regime~$\eta^2=O(h)$, Moser~\cite{Moser03},~\cite{Moser04},~\cite{Moser05} showed the emergence of boundary vortices, again for topological reasons. In large thin-films corresponding to the regime~$\eta^2 \ll h$, studied by Ignat-Otto~\cite{IO12} and Ignat-Kn\"upfer~\cite{IKn10}, other types of topological singularities occur as interior vortices or N\'eel walls.

\subsection{Outline of the paper}
\label{INTROSECTION_outline}

The rest of the paper is organised as follows.
In Section~2, we prove Theorems~\ref{THM_1.2_IK21+DMI},~\ref{THM_1.4_IK21+DMI} and~\ref{THM_1.5_IK21+DMI}, that give the~$\Gamma$-convergence of the two-dimensional reduced energy~$E_{\epsilon,\eta}^\delta$. Moreover, we study minimisers of the renormalised energy~$W_\Omega^\delta$, and prove Theorems~\ref{THM_6_IK22+DMI} and~\ref{THM_minimisers_unit_disk}, as well as Corollaries~\ref{COR_7_IK22+DMI} and~\ref{COR_1.6_IK21+DMI}.
In Section~3, we reduce the~3D energy~$E_h$ to the~2D energy~$E_{\epsilon,\eta}^\delta$, and prove Theorem~\ref{THM_reduction3D2D}. We also prove the~$\Gamma$-convergence of the~3D energy~$E_h$, i.e. Theorems~\ref{THM_9_order1_IK22+DMI},~\ref{THM_9_order2_IK22+DMI} and~\ref{THM_10_IK22+DMI}, and deduce Corollary~\ref{COR_11_IK22+DMI}.

\bigskip

\noindent{\bf Acknowledgment.} R.I. is partially supported by the ANR projects ANR-21-CE40-0004 and ANR-22-CE40-0006-01.

\section{Two-dimensional reduced model for maps $v \colon \Omega \subset \R^2 \rightarrow \R^2$}
\label{SECTION_2D_ENERGY}

Since all quantities in this section are two-dimensional quantities, \textbf{we drop the primes~$'$ in the notations}, e.g.~$x=(x_1,x_2)$ instead of~$x'$,~$\nabla=(\dr_1,\dr_2)$ instead of~$\nabla'$, etc.

\subsection{Approximation by $\mathbb{S}^1$-valued maps}
\label{SUBSECTION_approximation_by_S1_valued_maps}

Given a map~$v \colon \Omega \rightarrow \R^2$ such that~$E_{\epsilon,\eta}^\delta(v) = O(\lv \log \epsilon \rv)$, we show in this section that~$v$ can be approximated by a~$\mathbb{S}^1$-valued map~$\m \colon \Omega \rightarrow \mathbb{S}^1$ in the regime~\eqref{DEF_regime2D}. The idea is to prove that in this context, we have~$E_{\epsilon,\eta}^0(v)=O(\lv \log \epsilon \rv)$, i.e., we reduce to the case~$\delta=0$ and then use the approximation result of Ignat-Kurzke~\cite[Theorem 3.1]{IK21}.

\begin{lemma}
\label{LEM_for_THM_3.1_IK21+DMI}
Let~$\delta \in \R^2$ and~$\Omega \subset \mathbb{R}^2$ be a bounded, simply connected~$C^{1,1}$ domain. Assume~$\epsilon \rightarrow 0$ and~$\eta=\eta(\epsilon) \rightarrow 0$ in the regime~\eqref{DEF_regime2D}.
For every~$v=v_\epsilon \colon \Omega \rightarrow \mathbb{R}^2$ satisfying~$E_{\epsilon,\eta}^\delta(v) = O(\left\vert \log \epsilon \right\vert)$, we have~$E_{\epsilon,\eta}^0(v)=O(\left\vert \log \epsilon \right\vert)$.
\end{lemma}

\begin{proof}
We have by Young's inequality
\begin{align*}
\left\vert E_{\epsilon,\eta}^\delta(v)
-E_{\epsilon,\eta}^0(v) \right\vert
= 2 \left\vert \int_\Omega \delta \cdot \nabla v \wedge v \ \d x \right\vert
& \leqslant \frac{1}{2} \int_\Omega \left( \left\vert \nabla v \right\vert^2 + 4 \left\vert \delta \right\vert^2 \vert v \vert^2 \right) \d x.
\end{align*}
Setting~$S=\lbrace x \in \Omega : \left\vert v (x) \right\vert^2 \geqslant 2 \rbrace$, it follows that for every~$x \in S$, $\left\vert v(x) \right\vert^2-1 
\geqslant \frac{\lv v(x) \rv^2}{2}
\geqslant \frac{1}{\sqrt{2}} \left\vert v(x) \right\vert$, thus
\begin{align*}
\left\vert E_{\epsilon,\eta}^\delta(v)
-E_{\epsilon,\eta}^0(v) \right\vert
& \leqslant \frac{1}{2} \int_\Omega \left\vert \nabla v \right\vert^2 \d x
+ 2 \lv \delta \right\vert^2 \int_S \vert v \vert^2 \d x
+ 2 \left\vert \delta \right\vert^2 \int_{\Omega \setminus S} \vert v \vert^2 \d x
\\
& \leqslant \frac{1}{2} \int_\Omega \left\vert \nabla v \right\vert^2 \d x
+ 4 \left\vert \delta \right\vert^2 \int_S \left( 1-\vert v \vert^2 \right)^2 \d x
+ 4 \left\vert \delta \right\vert^2 \left\vert \Omega \setminus S \right\vert
\\
& \leqslant \frac{1}{2} \int_\Omega \left\vert \nabla v \right\vert^2 \d x
+ \frac{1}{2\eta^2} \int_\Omega \left( 1-\vert v \vert^2 \right)^2 \d x
+ 4 \left\vert \delta \right\vert^2 \left\vert \Omega \right\vert
\\
& \leqslant \frac{1}{2} E_{\epsilon,\eta}^0(v)
+ 4 \left\vert \delta \right\vert^2 \left\vert \Omega \right\vert
\end{align*}
because~$\left\vert \delta \right\vert^2 \ll \frac{1}{\eta^2}$ if~$\epsilon>0$ is small in the regime~\eqref{DEF_regime2D}. It follows that
\begin{equation*}
E_{\epsilon,\eta}^0(v)
\leqslant 2 E_{\epsilon,\eta}^\delta(v)
+ 8 \lv \delta \right\vert^2 \left\vert \Omega \right\vert
= O(\left\vert \log \epsilon \right\vert).
\end{equation*}
\end{proof}

\begin{remark}
\label{REM_LEM_for_THM_3.1_IK21+DMI}
We assumed in the above lemma that~$\delta \in \R^2$ is constant. However, the above proof shows that Lemma~\ref{LEM_for_THM_3.1_IK21+DMI} holds under the weaker assumption~$\lv \delta \rv = O(\lv \log \epsilon \rv^{1/2})$.
\end{remark}

\begin{notation}
Let~$\delta \in \R^2$,~$\epsilon>0$ and~$\eta=\eta(\epsilon)>0$. For any open set~$G \subset \Omega$ and~$v \colon \Omega \rightarrow \R^2$, we define the localised functional
\begin{equation*}
\begin{split}
E_{\epsilon,\eta}^\delta(v;G)
& = \int_G \left\vert \nabla v \right\vert^2 \d x
+ 2 \int_G \delta \cdot \nabla v \wedge v \ \d x
\\
& \quad
+ \frac{1}{\eta^2} \int_G \left( 1-\vert v \vert^2 \right)^2 \d x
+ \frac{1}{2\pi\epsilon} \int_{\overline{G} \cap \partial \Omega} (v \cdot \nu)^2 \d \mathcal{H}^1.
\end{split}
\end{equation*}
\end{notation}

The next theorem gives the approximation of~$v \colon \Omega \rightarrow \R^2$ in~$L^2(\Omega)$ by a~$\mathbb{S}^1$-valued \linebreak map~$V \colon \Omega \rightarrow \mathbb{S}^1$ that keeps the energy and global Jacobian close to the ones of~$v$. It is reminiscent from~\cite[Theorem~3.1]{IK21} where the case~$\delta=0$ was treated.

\begin{theorem}
\label{THM_3.1_IK21+DMI}
Let~$\delta \in \R^2$,~$\beta \in (\frac{1}{2},1)$,~$C>0$ and~$\Omega \subset \mathbb{R}^2$ be a bounded, simply connected~$C^{1,1}$ domain. Assume~$\epsilon \rightarrow 0$ and~$\eta=\eta(\epsilon) \rightarrow 0$ in the regime~\eqref{DEF_regime2D}.
There exist~$\epsilon_0>0$ and~$\widetilde{C}>0$, depending only on~$\beta$,~$C$,~$\delta$,~$\Omega$, the function~$\epsilon \mapsto \eta(\epsilon)$ and~$\widetilde{\beta} \in \left( 0,\frac{1-\beta}{6} \right)$ such that for every~$\epsilon \in (0,\epsilon_0)$ \linebreak and every~$v=v_\epsilon \colon \Omega \rightarrow \mathbb{R}^2$ satisfying~$E_{\epsilon,\eta}^\delta(v) \leqslant C \left\vert \log \epsilon \right\vert$, we can construct a unit-length \linebreak map~$V=V_\epsilon \colon \Omega \rightarrow \mathbb{S}^1$ that satisfies the following relations:
\begin{equation}
\label{eq_thm_3.1_IK21+DMI_1}
\int_\Omega \left\vert \m-v \right\vert^2 \d x
\lesssim \eta^{2\beta} E_{\epsilon,\eta}^0(v),
\ \ \int_{\partial \Omega} \left\vert \m-v \right\vert^2 \d \mathcal{H}^1
\lesssim \eta^\beta E_{\epsilon,\eta}^0(v),
\end{equation}
and
\begin{equation}
\label{eq_thm_3.1_IK21+DMI_2}
E_{\epsilon,\eta}^\delta(\m)
\leqslant E_{\epsilon,\eta}^\delta(v)
+ \widetilde{C} \eta^{\widetilde{\beta}} \left(
E_{\epsilon,\eta}^0(v) + \sqrt{E_{\epsilon,\eta}^0(v)}
\right).
\end{equation}
As a consequence, for every~$p \in [1,+\infty)$,\footnote{Note that~$\mathrm{jac}(V)=0$ as~$V \in H^1(\Omega,\mathbb{S}^1)$. \\ For~$A \in (W^{1,\infty}(\Omega))^\ast$, we define~$\Vert A \Vert_{(\mathrm{Lip}(\Omega))^\ast}=\sup \lb \langle A,\zeta \rangle \colon \zeta \in W^{1,\infty}(\Omega), \vert \nabla \zeta \vert \leqslant 1 \rb$ \\ and~$\Vert A \Vert_{(W_0^{1,\infty}(\Omega))^\ast}=\sup \lb \langle A,\zeta \rangle \colon \zeta \in W^{1,\infty}(\Omega), \vert \nabla \zeta \vert \leqslant 1, \zeta=0 \text{ on } \partial \Omega \rb$.}
\begin{equation}
\label{eq_thm_3.1_IK21+DMI_3}
\begin{split}
& \lim\limits_{\epsilon \rightarrow 0} 
\left\Vert V-v \right\Vert_{L^p(\Omega)} = 0,
\ \ 
\lim\limits_{\epsilon \rightarrow 0} 
\left\Vert V-v \right\Vert_{L^p(\partial \Omega)} = 0,
\\
& \left\Vert \mathrm{jac}(v) \right\Vert_{(W_0^{1,\infty})^\ast(\Omega)} \lesssim \eta^\beta E_{\epsilon,\eta}^0(v),
\ \text{and} \ \ 
\left\Vert \mathcal{J}(V)-\mathcal{J}(v) \right\Vert_{(\mathrm{Lip}(\Omega))^\ast} \lesssim \sqrt{\eta^\beta E_{\epsilon,\eta}^0(v)}.
\end{split}
\end{equation}
The map~$V$ also satisfies the following local estimate: for any open set~$G \subset \Omega$ independent of~$\epsilon$, there exists a constant~$\widetilde{C}_G>0$ such that
\begin{equation}
\label{eq_thm_3.1_IK21+DMI_4}
E_{\epsilon,\eta}^0(V;G_\eta)
\leqslant E_{\epsilon,\eta}^0(v;G)
+ \widetilde{C}_G\eta^{\widetilde{\beta}} \left(
E_{\epsilon,\eta}^0(v;G) + \sqrt{E_{\epsilon,\eta}^0(v;G)} 
\right)
\end{equation}
where
\begin{equation*}
G_\eta = \left\lbrace x \in G : \mathrm{dist}(x,\Omega \cap \partial G)>c_0\eta^\beta \right\rbrace,
\end{equation*}
with~$c_0>0$ depending only on~$\partial\Omega$.
\end{theorem}

\begin{proof}
As~$E_{\epsilon,\eta}^0(v)=O(\lv \log \epsilon \rv)$ by Lemma~\ref{LEM_for_THM_3.1_IK21+DMI}, we apply~\cite[Theorem~3.1]{IK21} to prove the existence of~$V \colon \Omega \rightarrow \mathbb{S}^1$ that satisfies~\eqref{eq_thm_3.1_IK21+DMI_1},~\eqref{eq_thm_3.1_IK21+DMI_3} and
\begin{equation}
\label{eq_thm_3.1_IK21+DMI_5}
\int_\Omega \left(\vert\nabla V\vert^2+\vert\nabla v\vert^2\right) \d x \lesssim E_{\epsilon,\eta}^0(v).
\end{equation}
Moreover, by~\cite[Theorem~3.1, Equation~(31)]{IK21}, we know that
\begin{align*}
E_{\epsilon,\eta}^0(V)
& \leqslant E_{\epsilon,c_0\eta}^0(v)
+ \widetilde{C} \eta^{\widetilde{\beta}} 
\left(
E_{\epsilon,c_0\eta}^0(v) 
+ \sqrt{E_{\epsilon,c_0\eta}^0(v)} 
\right)
\end{align*}
for some~$c_0>0$. As~$\vert V \vert = 1$, note that~$E_{\epsilon,\eta}^0(V)=E_{\epsilon,c_0\eta}^0(V)$, so we can replace~$\eta$ by~$\widehat{\eta}=c_0\eta$. We still denote~$\eta$ instead of~$\widehat{\eta}$ in the following.
To prove~\eqref{eq_thm_3.1_IK21+DMI_2}, we compute
\begin{align*}
E_{\epsilon,\eta}^\delta(\m)
& = E_{\epsilon,\eta}^0(\m)
+ 2 \int_{\Omega} \delta \cdot \left( \nabla \m \wedge \m - \nabla v \wedge v \right) \d x + 2 \int_{\Omega} \delta \cdot \nabla v \wedge v \ \d x
\end{align*}
so that
\begin{align*}
E_{\epsilon,\eta}^\delta(V)
& \leqslant E_{\epsilon,\eta}^\delta(v)
+ \widetilde{C} \eta^{\widetilde{\beta}} 
\left(
E_{\epsilon,\eta}^0(v) 
+ \sqrt{E_{\epsilon,\eta}^0(v)} 
\right)
+ 2 \int_{\Omega} \delta \cdot \left( \nabla V \wedge V - \nabla v \wedge v \right) \d x.
\end{align*}
By integration by parts, we have
\begin{align*}
\int_\Omega \delta \cdot \left( \nabla V \wedge V - \nabla v \wedge v \right) \d x
& = \int_\Omega \delta \cdot \left( (\nabla V - \nabla v) \wedge \m - \nabla v \wedge (v-V) \right) \d x
\\
& = \int_{\dr\Omega} (\delta \cdot \nu)(V-v) \wedge V \ \d \mathcal{H}^1
- \int_\Omega \delta \cdot (V-v) \wedge (\nabla V + \nabla v) \d x.
\end{align*}
We deduce that
\begin{align*}
\lv \int_\Omega \delta \cdot \left( \nabla V \wedge V - \nabla v \wedge v \right) \d x \rv
& \leqslant
\int_{\dr\Omega} \lv \delta \cdot \nu \rv \vert V-v \vert \d\mathcal{H}^1
+ \lv \delta \rv \int_\Omega \vert V-v \vert \left( \vert \nabla V \vert + \lv \nabla v \rv \right) \d x
\\
& \lesssim \sqrt{\int_{\dr\Omega} \vert V-v \vert^2 \d\mathcal{H}^1}
+
\sqrt{\int_{\Omega} \vert V-v \vert^2 \d x}
\sqrt{\int_\Omega \left( \vert \nabla V \vert^2+\vert \nabla v \vert^2 \right) \d x},
\\
& \lesssim 
\eta^\beta E_{\epsilon,\eta}^0(v)
+ \eta^{\beta/2} \sqrt{E_{\epsilon,\eta}^0(v)}
\\
& \lesssim
\eta^{\widetilde{\beta}} 
\left(
E_{\epsilon,\eta}^0(v) 
+ \sqrt{E_{\epsilon,\eta}^0(v)} 
\right)
\end{align*}
where we used~\eqref{eq_thm_3.1_IK21+DMI_1},~\eqref{eq_thm_3.1_IK21+DMI_5} and~$\frac{\beta}{2} \geqslant \frac{1}{4} > \widetilde{\beta}$. We conclude to~\eqref{eq_thm_3.1_IK21+DMI_2}.
Finally,~\eqref{eq_thm_3.1_IK21+DMI_4} is reminiscent from~\cite[Theorem~3.1, Equation~(32)]{IK21} and the fact that~$E_{\epsilon,\eta}^0(V)=E_{\epsilon,c_0\eta}^0(V)$ as~$\vert V \vert =1$.
\end{proof}

\subsection{Lifting}
\label{SUBSECTION_lifting}

By Theorem~\ref{THM_3.1_IK21+DMI}, our study simplifies to the analysis of the energy functional~$E_{\epsilon,\eta}^\delta$ for~$\mathbb{S}^1$-valued maps~$V$. Such maps have a lifting (see Bethuel-Zheng~\cite{BethuelZheng}) on which we focus in the following:

\begin{lemma}
\label{LEM_4.1_IK21+DMI}
Let~$\delta \in \R^2$,~$\Omega \subset \mathbb{R}^2$ be a bounded, simply connected~$C^{1,1}$ domain and~$V \in H^1(\Omega,\mathbb{S}^1)$. There exists a lifting~$\phir \in H^1(\Omega,\mathbb{R})$ such that~$V=e^{i\phir}$ and~$\phir$ is unique up to an additive constant in~$2\pi\mathbb{Z}$. Furthermore, for every~$\epsilon>0$ and~$\eta>0$,
\begin{equation}
\label{DEF_Gepsdelta}
E_{\epsilon,\eta}^\delta(V)
= \int_\Omega \left( \left\vert \nabla \phir \right\vert^2 -2 \delta \cdot \nabla \phir \right) \d x
+ \frac{1}{2\pi\epsilon} \int_{\partial \Omega} \sin^2(\phir-g) \ \d \mathcal{H}^1
=: \mathcal{G}_{\epsilon}^\delta(\phir),
\end{equation}
where~$g$ is a lifting of the unit tangent vector field~$\tau$ on~$\partial \Omega$, i.e. as complex numbers
\begin{equation}
\label{DEF_lifting}
e^{ig}=\tau=i\nu \ \ \text{ on } \partial \Omega,
\end{equation}
and~$g$ is continuous except at one point of~$\partial \Omega$.
\end{lemma}

\begin{proof}
Existence and uniqueness of a lifting~$\phir$ of~$\m$ in~$\Omega$ come from a well-known theorem of Bethuel and Zheng~\cite{BethuelZheng}.
For the existence of~$g$, we note that~$\tau$ has winding number~1 on~$\dr\Omega$ as~$\Omega$ is simply connected, hence~$\tau$ cannot be lifted continuously on~$\dr\Omega$. However, if~$\dr\Omega$ is~$C^{1,1}$, we can choose a lifting~$g$ to be locally Lipschitz except at one point of~$\dr\Omega$ where it jumps by~$2\pi$. Clearly, the curvature~$\kappa$ of~$\dr\Omega$ is given by the absolutely continuous part of the derivative of~$g$ (as a~BV function), i.e.~$\kappa=(\dr_\tau g)_{ac}$ and~$\int_{\dr\Omega} \kappa \ \d \mathcal{H}^1 = 2\pi$, which is the Gauss-Bonnet formula on~$\dr\Omega$. Therefore,~$g \in BV(\dr\Omega,\R)$ with~$\dr_\tau g=\kappa \mathcal{H}^1 \llcorner \dr\Omega -2\pi\mathbf{\delta}_p$ for some point~$p \in \dr\Omega$.
\linebreak
As~$\lv \nabla V \rv = \lv \nabla \phir \rv$,~$\nabla V \wedge V = -\nabla \phir$ in~$\Omega$, and~$\m \cdot \nu = -\sin(\phir-g)$ on~$\dr\Omega$, we deduce~\eqref{DEF_Gepsdelta}.
\end{proof}

The functional~$\mathcal{G}_\epsilon^\delta$ in the above lemma has been studied by Kurzke~\cite{Kurzke06},~\cite{Kurzke06-2} for~$\delta=0$. In the following, we will prove~$\Gamma$-convergence for~$\mathcal{G}_\epsilon^\delta$ and use these result for proving~$\Gamma$-convergence for~$E_{\epsilon,\eta}^\delta$.

\subsection{Gamma-convergence in terms of liftings}
\label{SUBSECTION_GC_liftings}

We now present the~$\Gamma$-convergence for the functional~$\mathcal{G}_\epsilon^\delta$ defined in~\eqref{DEF_Gepsdelta} in terms of the lifting~$\phir_\epsilon \colon \Omega \rightarrow \R$. These results will be useful to deduce similar statements for~$E_{\epsilon,\eta}^\delta(v_\epsilon)$ in the next section. The first statement establishes the~$L^p(\dr\Omega)$-compactness of~$\phir_\epsilon$ and a lower bound for~$\mathcal{G}_\epsilon^\delta$ at the first order.

\begin{theorem}
\label{THM_4.2.1_IK21+DMI}
Let~$\delta \in \R^2$,~$\Omega \subset \mathbb{R}^2$ be a bounded, simply connected and~$C^{1,1}$ domain and~$\kappa$ be the curvature of~$\partial \Omega$. Let~$(\phir_\epsilon)_{\epsilon \downarrow 0}$ be a family in~$H^1(\Omega)$ such that
\begin{equation}
\label{hyp_limsup_G_order1}
\limsup\limits_{\epsilon \rightarrow 0} \frac{1}{\left\vert \log \epsilon \right\vert} \mathcal{G}_{\epsilon}^\delta(\phir_\epsilon) <+\infty.
\end{equation}
There exists a family~$(z_\epsilon)_\epsilon$ of integers such that~$(\phir_\epsilon-\pi z_\epsilon)_\epsilon$ is bounded in~$L^p(\partial \Omega)$, for \linebreak every~$p \in [1,+\infty)$.
Moreover, for a subsequence,~$(\phir_\epsilon-\pi z_\epsilon)_\epsilon$ converges strongly in~$L^p(\partial \Omega)$ to a limit~$\phir_0$ such that~$\phir_0-g \in BV(\partial \Omega,\pi \mathbb{Z})$, with~$g$ given in~\eqref{DEF_lifting}, and
\begin{equation*}
\partial_\tau\phir_0 = \kappa \mathcal{H}^1\llcorner\partial\Omega-\pi\sum_{j=1}^N d_j\Dirac_{a_j} \ \ \text{ as measure on } \partial \Omega
\end{equation*}
where~$N \geqslant 1$,~$\lb a_j \in \partial \Omega \rb_{1 \leqslant j \leqslant N}$ are distinct points, $d_j \in \mathbb{Z} \setminus \left\lbrace 0 \right\rbrace$,~$\sum_{j=1}^N d_j=2$, and~$(\partial_\tau \phir_\epsilon)_\epsilon$ converges to~$\partial_\tau \phir_0$ in~$W^{-1,p}(\partial \Omega)$ for every~$p \in [1,+\infty)$.
Furthermore, we have the following first order lower bound for the energy:
\begin{equation}
\label{eq_liminf_G_order1}
\liminf\limits_{\epsilon \rightarrow 0} \frac{1}{\left\vert \log \epsilon \right\vert} \mathcal{G}_{\epsilon}^\delta(\phir_\epsilon) \geqslant \left\vert \partial_\tau\phir_0-\kappa \mathcal{H}^1 \llcorner \partial\Omega \right\vert (\partial \Omega) = \pi\sum\limits_{j=1}^N \left\vert d_j \right\vert.
\end{equation}
\end{theorem}

\begin{proof}
Note that, for any~$\epsilon>0$ and any~$\sigma>0$, Young's inequality yields
\begin{align*}
\int_\Omega \left\vert \nabla \phir_\epsilon \right\vert^2 \d x
& = \int_\Omega \left\vert (\nabla \phir_\epsilon -\delta)+\delta \right\vert^2 \d x
 \leqslant (1+\sigma) \int_\Omega \left\vert \nabla \phir_\epsilon-\delta \right\vert^2 \d x
+ \left( 1+\frac{1}{\sigma} \right) \left\vert \delta \right\vert^2 \left\vert \Omega \right\vert,
\end{align*}
since~$(a+b)^2=a^2+b^2+2ab \leqslant a^2+b^2+\left( \sqrt{\sigma}a \right)^2+ \left( \frac{1}{\sqrt{\sigma}}b \right)^2$ for every~$a,b \in \R$. Hence,
\begin{align}
\label{est100}
\mathcal{G}_\epsilon^0(\phir_\epsilon)
& \leqslant (1+\sigma) \int_\Omega \lv \nabla \phir_\epsilon-\delta \rv^2 \d x
+ \frac{1}{2\pi\epsilon} \int_{\dr\Omega} \sin^2(\phir_\epsilon-g) \ \d \mathcal{H}^1
+ \left( 1+\frac{1}{\sigma} \right) \lv \delta \rv^2 \lv \Omega \rv
\nonumber
\\
& \leqslant (1+\sigma) \int_\Omega \left( \lv \nabla \phir_\epsilon \rv^2 -2\delta \cdot \nabla \phir_\epsilon \right) \d x
+ \frac{1+\sigma}{2\pi\epsilon} \int_{\dr\Omega} \sin^2(\phir_\epsilon-g) \ \d \mathcal{H}^1
+ \left( 2+\sigma+\frac{1}{\sigma} \right) \lv \delta \rv^2 \lv \Omega \rv
\nonumber
\\
& \leqslant (1+\sigma) \mathcal{G}_\epsilon^\delta(\phir_\epsilon)
+ \left( 2+\sigma+\frac{1}{\sigma} \right) \lv \delta \rv^2 \lv \Omega \rv.
\end{align}
By~\eqref{hyp_limsup_G_order1}, we deduce that~$\limsup\limits_{\epsilon \rightarrow 0} \frac{\mathcal{G}_\epsilon^0(\phir_\epsilon)}{\lv \log \epsilon \rv} <+\infty$. Hence, we can apply~\cite[Theorem~4.2]{IK21} and deduce the desired compactness results. Moreover, by~\eqref{est100} and~\cite[Theorem~4.2]{IK21}, we have
\begin{equation*}
\liminf\limits_{\epsilon \rightarrow 0} (1+\sigma) \frac{\mathcal{G}_\epsilon^\delta(\phir_\epsilon)}{\lv \log \epsilon \rv}
\geqslant \liminf\limits_{\epsilon \rightarrow 0} \frac{\mathcal{G}_\epsilon^0(\phir_\epsilon)}{\lv \log \epsilon \rv} \geqslant \pi \sum\limits_{j=1}^N \lv d_j \rv.
\end{equation*}
Letting $\sigma \rightarrow 0$, we get \eqref{eq_liminf_G_order1}.
\end{proof}

The following theorem gives the lower bound of~$\mathcal{G}_\epsilon^\delta$ of second order together with the multiplicities~$\pm 1$ of boundary vortices provided a more precise energy estimate than~\eqref{eq_liminf_G_order1}.

\begin{theorem}
\label{THM_4.2.2_IK21+DMI}
Let~$\delta \in \R^2$,~$\Omega \subset \mathbb{R}^2$ be a bounded, simply connected and~$C^{1,1}$ domain and~$\kappa$ be the curvature of~$\partial \Omega$. Let~$(\phir_\epsilon)_{\epsilon \downarrow 0}$ be a family in~$H^1(\Omega)$ satisfying the convergence in Theorem~\ref{THM_4.2.1_IK21+DMI} with the limit~$\phir_0$ on~$\partial \Omega$ as~$\epsilon \rightarrow 0$. Assume additionally that
\begin{equation}
\label{hyp_limsup_G_order2}
\limsup\limits_{\epsilon \rightarrow 0} \left( \mathcal{G}_{\epsilon}^\delta(\phir_\epsilon) -\pi \left\vert \log \epsilon \right\vert \sum_{j=1}^N \left\vert d_j \right\vert \right) <+\infty.
\end{equation}
Then~$d_j \in \left\lbrace \pm 1 \right\rbrace$ for every~$j \in \left\lbrace 1,...,N \right\rbrace$ and, for a subsequence,~$(\nabla \phir_\epsilon)_\epsilon$ converges weakly \linebreak in~$L^q(\Omega,\mathbb{R}^2)$ for any~$q \in [1,2)$ to~$\nabla \widehat{\phir}_0$, where~$\widehat{\phir}_0 \in W^{1,q}(\Omega)$ is an extension (not necessarily harmonic) of~$\phir_0$ to~$\Omega$.
Furthermore, we have the following second order lower bound of the energy:
\begin{equation}
\label{eq_liminf_G_order2}
\liminf\limits_{\epsilon \rightarrow 0} \left( \mathcal{G}_{\epsilon}^\delta(\phir_\epsilon)-N \pi \left\vert \log \epsilon \right\vert \right) \geqslant W_\Omega^\delta(\left\lbrace (a_j,d_j) \right\rbrace) + N\gamma_0,
\end{equation}
where  $\gamma_0=\pi \log\frac{e}{4\pi}$ and $W_\Omega^\delta(\lb(a_j,d_j)\rb)$ is the renormalised energy defined in~\eqref{DEF_renormalised_energy}.
\end{theorem}

\begin{proof}
We adapt the proof of~\cite[Theorem~4.2]{IK21} to the case of~$\delta \in \R^2$.
\vspace{.1cm}

\textit{Step 1:} We prove that~$d_j \in \lb \pm 1 \rb$ for every~$j \in \lb 1,...,N \rb$ and the weak convergence of~$(\nabla\phir_\epsilon)_{\epsilon \downarrow 0}$.
For that, note that by Theorem~\ref{THM_4.2.1_IK21+DMI}, there exists~$z_\epsilon \in \mathbb{Z}$ such that
\begin{equation*}
\int_\Omega \delta \cdot \nabla \phir_\epsilon \ \d x
= \int_\Omega \delta \cdot \nabla (\phir_\epsilon-\pi z_\epsilon) \ \d x
= \int_{\partial \Omega} (\phir_\epsilon - \pi z_\epsilon) \delta \cdot \nu \ \d \mathcal{H}^1
\leqslant \lv \delta \rv \lV \phir_\epsilon - \pi z_\epsilon \rV_{L^1(\dr\Omega)} \leqslant C.
\end{equation*}
As
\begin{equation*}
\mathcal{G}_\epsilon^\delta(\phir_\epsilon)=\mathcal{G}_\epsilon^0(\phir_\epsilon)-2\int_\Omega \delta \cdot \nabla\phir_\epsilon \ \d x \geqslant \mathcal{G}_\epsilon^0(\phir_\epsilon)-2C,
\end{equation*}
we deduce that~\eqref{hyp_limsup_G_order2} holds for~$\mathcal{G}_\epsilon^0(\phir_\epsilon)$ and by~\cite[Theorem~4.2, part~2]{IK21}, it yields the claim in Step~1. Moreover, by~\cite[Proof of Theorem~4.2]{IK21}, we have:
\begin{align*}
\int_{\Omega \setminus \bigcup_j B_r(a_j)} \lv \nabla \phir_\epsilon \rv^2 \d x \leqslant N\pi\log\frac{1}{r}+C,
\end{align*}
for every small~$r>0$ and~$\epsilon>0$.
\vspace{.1cm}

\textit{Step 2:} We prove the second order lower bound~\eqref{eq_liminf_G_order2}. To do so, we replace~$\phir_\epsilon$ by the harmonic extension of~$\phir_\epsilon \vert_{\partial \Omega}$ into~$\Omega$, denoted by~$\phir_\epsilon^\ast$. More precisely,~$\phir_\epsilon^\ast$ is the minimiser of the Dirichlet energy in~$\Omega$ under the Dirichlet boundary condition~$\phir_\epsilon \vert_{\partial \Omega}$. In particulier,~$\phir_\epsilon^\ast$ is harmonic in~$\Omega$. Since
\begin{equation*}
\int_\Omega \delta \cdot \nabla \phir_\epsilon \ \d x
= \int_{\partial \Omega} \phir_\epsilon \ \delta \cdot \nu \ \d \mathcal{H}^1
= \int_{\partial \Omega} \phir_\epsilon^\ast \ \delta \cdot \nu \ \d \mathcal{H}^1
= \int_\Omega \delta \cdot \nabla \phir_\epsilon^\ast \ \d x,
\end{equation*}
we deduce that~$\mathcal{G}_{\epsilon}^\delta(\phir_\epsilon) \geqslant \mathcal{G}_{\epsilon}^\delta(\phir_\epsilon^\ast)$, thus it suffices to prove~\eqref{eq_liminf_G_order2} for~$\phir_\epsilon^\ast$.
As~$\lb a_j,d_j \rb_{1 \leqslant j \leqslant N}$ are determined by~$\phir_\epsilon \vert_{\partial \Omega}$, the right-hand side in~\eqref{eq_liminf_G_order2} remains the same when replacing~$\phir_\epsilon$ by~$\phir_\epsilon^\ast$.
Using Step~1, we know that~$(\phir_\epsilon^\ast)$ converges weakly in~$W^{1,q}(\Omega)$, for every~$q \in [1,2)$, and weakly in~$H^1(\omega)$, for any open set~$\omega$ such that~$\overline{\omega} \subset \overline{\Omega} \setminus \left\lbrace a_1,...,a_N \right\rbrace$, to the harmonic extension~$\phir_\ast$ of~$\phir_0$ to~$\Omega$. Let~$r>0$ be small and~$\Omega^r:=\Omega \setminus \bigcup_{j=1}^N B_r(a_j)$. By weak lower semicontinuity of the Dirichlet integral, we have
\begin{equation*}
\int_{\Omega^r} \vert \nabla \phir_\ast \vert^2 \d x
\leqslant \liminf\limits_{\epsilon \rightarrow 0} \int_{\Omega^r} \vert \nabla \phir_\epsilon^\ast \vert^2 \d x.
\end{equation*}
Also, by weak convergence of~$(\nabla \phir_\epsilon^\ast)_\epsilon$ to~$\nabla \phir_\ast$ in~$L^1(\Omega,\mathbb{R}^2)$, we have
\begin{equation*}
\lim\limits_{\epsilon \rightarrow 0} \int_{\Omega^r} \delta \cdot \nabla \phir_\epsilon^\ast \ \d x
= \int_{\Omega^r} \delta \cdot \nabla \phir_\ast \ \d x.
\end{equation*}
By definition of~$W_\Omega^\delta$ in~\eqref{DEF_renormalised_energy}, we have
\begin{equation*}
\int_{\Omega^r} \left( \vert \nabla \phir_\ast \vert^2-2\delta \cdot \nabla\phir_\ast \right) \d x
\geqslant \pi N \log \frac{1}{r}+W_\Omega^\delta(\lb a_j,d_j \rb)+o_r(1), \text{ as } r \rightarrow 0.
\end{equation*}
In~$\Omega \setminus \Omega^r$, we have by H\"older's inequality:
\begin{equation*}
\int_{\Omega \setminus \Omega^r} \delta \cdot \nabla \phir_\epsilon^\ast \ \d x
= \sum_{j=1}^N \int_{\Omega \cap B_r(a_j)} \delta \cdot \nabla \phir_\epsilon^\ast \ \d x.
\leqslant C \vert \delta \vert r^{2/3} \left\Vert \nabla \phir_\epsilon^\ast \right\Vert_{L^{3/2}(\Omega)}
\leqslant C r^{2/3},
\end{equation*}
for some constant~$C>0$, because~$(\nabla \phir_\epsilon^\ast)_{\epsilon \downarrow 0}$ is bounded (independently of~$\epsilon$) in~$L^{3/2}(\Omega,\R^2)$. By~\cite[Equation~(85)]{IK21}, we have
\begin{equation*}
\begin{split}
\liminf\limits_{\epsilon \rightarrow 0}
\left[ \sum_{j=1}^N \left( \int_{\Omega \cap B_r(a_j)} \vert \nabla \phir_\epsilon^\ast \vert^2 \d x
+ \frac{1}{2\pi\epsilon} \int_{\partial \Omega \cap B_r(a_j)} \sin^2(\phir_\epsilon^\ast-g) \ \d \mathcal{H}^1
\right)
- N \pi \log \frac{r}{\epsilon} \right]
& 
\\
\geqslant -CNr^{1/2}+N\gamma_0.
& 
\end{split}
\end{equation*}
Therefore, we conclude
\begin{equation*}
\liminf\limits_{\epsilon \rightarrow 0} \left( \mathcal{G}_{\epsilon}^\delta(\phir_\epsilon^\ast) - N\pi \vert \log \epsilon \vert \right) \geqslant
W_\Omega^\delta(\left\lbrace (a_j,d_j) \right\rbrace)
-CNr^{1/2}
+ N\gamma_0
- O(r^{2/3}).
\end{equation*}
Taking the limit as~$r \rightarrow 0$, we get~\eqref{eq_liminf_G_order2} for~$\phir_\epsilon^\ast$.
\end{proof}

Now we prove the upper bound for the~$\Gamma$-convergence of~$\mathcal{G}_\epsilon^\delta$ at the first order, respectively at the second order.

\begin{theorem}
\label{THM_4.2.3_IK21+DMI}
Let~$\delta \in \R^2$,~$\Omega \subset \mathbb{R}^2$ be a bounded, simply connected~$C^{1,1}$ domain and~$\kappa$ be the curvature of~$\partial \Omega$. Let~$\phir_0 \colon \partial \Omega \rightarrow \mathbb{R}$ be such that
\begin{equation*}
\partial_\tau \phir_0 = \kappa \mathcal{H}^1 \llcorner \partial\Omega - \pi \sum\limits_{j=1}^N d_j\Dirac_{a_j} \ \ \text{ as measure on } \partial\Omega,
\end{equation*}
where~$d_j \in \mathbb{Z} \setminus \left\lbrace 0 \right\rbrace$ for every~$j \in \left\lbrace 1,...,N \right\rbrace$,~$\sum_{j=1}^N d_j=2$ and~$e^{i\phir_0} \cdot \nu = 0$ in~$\partial \Omega \setminus \left\lbrace a_1,...,a_N \right\rbrace$ for~$N \geqslant 1$ distinct points~$a_1,...,a_N \in \partial \Omega$.
There exists a family~$(\phir_\epsilon)_\epsilon$ in~$H^1(\Omega)$ such that~$(\phir_\epsilon)_\epsilon$ converges to~$\phir_0$ in~$L^p(\partial \Omega)$ and to~$\phir_\ast$ in~$L^p(\Omega)$, for every~$p \in [1,+\infty)$, where~$\phir_\ast$ is the harmonic extension of~$\phir_0$ to~$\Omega$, and we have
\begin{equation}
\label{eq_limsup_G_order1}
\lim\limits_{\epsilon \rightarrow 0} \frac{1}{\left\vert \log \epsilon \right\vert} \mathcal{G}_{\epsilon}^\delta(\phir_\epsilon)
\leqslant \pi \sum\limits_{j=1}^N \left\vert d_j \right\vert.
\end{equation}
Furthermore, if~$d_j \in \left\lbrace \pm 1 \right\rbrace$ for every~$j \in \left\lbrace 1,...,N \right\rbrace$, then
\begin{equation}
\label{eq_limsup_G_order2}
\lim\limits_{\epsilon \rightarrow 0} \left( \mathcal{G}_{\epsilon}^\delta(\phir_\epsilon)-N\pi \left\vert \log \epsilon \right\vert \right)
\leqslant W_\Omega^\delta(\left\lbrace (a_j,d_j) \right\rbrace)+N\gamma_0,
\end{equation}
where~$\gamma_0=\pi \log\frac{e}{4\pi}$ and~$W_\Omega^\delta(\lb(a_j,d_j)\rb)$ is the renormalised energy defined in~\eqref{DEF_renormalised_energy}.
\end{theorem}

\begin{proof}
Let~$\phir_\ast$ be the harmonic extension of~$\phir_0$ to~$\Omega$, that satisfies~\eqref{DEF_renormalised_energy}.
We consider the family~$(\widehat{\phir}_\epsilon)$ constructed in~\cite[Theorem~4.2, part~3]{IK21} that satisfies the required convergence to~$\phir_\ast$ and estimates~\eqref{eq_limsup_G_order1} and~\eqref{eq_limsup_G_order2} for~$\delta=0$. We will show that this family~$(\widehat{\phir}_\epsilon)$ satisfies~\eqref{eq_limsup_G_order1} and~\eqref{eq_limsup_G_order2} for general~$\delta \in \R^2$.

\textit{Case~1:} $d_j \in \mathbb{Z}$ are not necessarily equal to~$\pm 1$. Then for~$\delta \in \R^2$, recall that
\begin{equation*}
\mathcal{G}_{\epsilon}^\delta(\widehat{\phir}_\epsilon)
= \mathcal{G}_{\epsilon}^0(\widehat{\phir}_\epsilon) -2 \int_\Omega \delta \cdot \nabla \widehat{\phir}_\epsilon \ \d x
= \mathcal{G}_\epsilon^0(\widehat{\phir}_\epsilon)
+2\int_{\partial\Omega} \widehat{\phir}_\epsilon \ \delta \cdot \nu \ \d \mathcal{H}^1.
\end{equation*}
As~$\widehat{\phir}_\epsilon \rightarrow \phir_\ast$ in~$L^1(\partial \Omega)$, we deduce that~$\lv \int_{\partial\Omega} \widehat{\phir}_\epsilon \ \delta \cdot \nu \ \d \mathcal{H}^1 \rv \leqslant C$ as~$\epsilon \downarrow 0$, which yields~\eqref{eq_limsup_G_order1} for general~$\delta \in \R^2$.

\textit{Case~2:} $\lv d_j \rv =1$ for every~$j \in \lb 1,...,N \rb$. By the definition of~$W_\Omega^0(\lb(a_j,d_j)\rb)$ and the construction in~\cite[Theorem~4.2, part~3]{IK21} (satisfying~\eqref{eq_limsup_G_order2} above for~$\delta=0$), we have:
\begin{equation*}
\mathcal{G}_\epsilon^0(\widehat{\phir}_\epsilon)
\leqslant N\pi\lv \log \epsilon \rv
+\left( \int_{\Omega \setminus \bigcup_{j=1}^N B_r(a_j)}  \left\vert \nabla \phir_\ast \right\vert^2 \d x -N\pi\log \frac{1}{r} \right)
+N\gamma_0 +o_\epsilon(1)+o_r(1).
\end{equation*}
Moreover, we have
\begin{equation*}
\int_\Omega \delta \cdot \nabla \widehat{\phir}_\epsilon \ \d x
= \int_{\Omega \setminus \bigcup_j B_r(a_j)} \delta \cdot \nabla \widehat{\phir}_\epsilon \ \d x
+ \int_{\Omega \cap \bigcup_j B_r(a_j)} \delta \cdot \nabla \widehat{\phir}_\epsilon \ \d x.
\end{equation*}
Let~$r>0$ be sufficiently small. Then by H\"older's inequality,
\begin{equation*}
\int_{\Omega \cap \bigcup_j B_r(a_j)} \delta \cdot \nabla \widehat{\phir}_\epsilon \ \d x
= \sum_{j=1}^N \int_{\Omega \cap B_r(a_j)} \delta \cdot \nabla \widehat{\phir}_\epsilon \ \d x
\leqslant C \vert\delta\vert r^{2/3} \lV \nabla \widehat{\phir}_\epsilon \rV_{L^{3/2}(\Omega)} = O(r^{2/3})
\end{equation*}
because~$(\nabla\widehat{\phir}_\epsilon)_\epsilon$ is bounded in~$L^{3/2}(\Omega)$ by construction (see~\cite[Theorem~4.2, part~2]{IK21}).
Moreover, by~\cite[Theorem~4.2, part~2]{IK21} again,~$(\nabla \widehat{\phir}_\epsilon)$ converges weakly to~$\nabla \phir_\ast$ in~$L^1(\Omega,\R^2)$, thus
\begin{equation*}
\lim\limits_{\epsilon \rightarrow 0} \int_{\Omega \setminus \bigcup_j B_r(a_j)} \delta \cdot \nabla \widehat{\phir}_\epsilon \ \d x
= \int_{\Omega \setminus \bigcup_j B_r(a_j)} \delta \cdot \nabla \phir_\ast \ \d x
\end{equation*}
Combining all above, and taking the~$\limsup$ as~$\epsilon \rightarrow 0$, we deduce
\begin{align*}
\limsup\limits_{\epsilon \rightarrow 0} \left( \mathcal{G}_\epsilon^\delta(\widehat{\psi}_\epsilon)
- N\pi\lv \log \epsilon \rv \right)
& \leqslant \left( \int_{\Omega \setminus \bigcup_{j=1}^N B_r(a_j)} \left( \left\vert \nabla \phir_\ast \right\vert^2 - 2 \delta \cdot \nabla \phir_\ast \right) \d x -N\pi\log \frac{1}{r} \right)
\\
& \quad + N\gamma_0 + o_r(1).
\end{align*}
Taking the~$\liminf$ as~$r \rightarrow 0$, we get the desired upper bound.
\end{proof}

\subsection{Gamma-convergence for vector-valued maps}
\label{subsection_gc_vector_valued_maps}

We now prove Theorems~\ref{THM_1.2_IK21+DMI},~\ref{THM_1.4_IK21+DMI} and~\ref{THM_1.5_IK21+DMI}.

\begin{proof}[Proof of Theorem~\ref{THM_1.2_IK21+DMI}]
Let~$(v_\epsilon)$ be a family in~$H^1(\Omega,\mathbb{R}^2)$ such that~$E_{\epsilon,\eta}^\delta(v_\epsilon) \leqslant C \left\vert \log \epsilon \right\vert$, for some constant~$C>0$.
Using Theorem~\ref{THM_3.1_IK21+DMI} and~Lemma~\ref{LEM_for_THM_3.1_IK21+DMI}, we can construct a family~$(\m_\epsilon)$ in~$H^1(\Omega,\mathbb{S}^1)$ such that in the regime~\eqref{DEF_regime2D}:
\begin{equation}
\label{eqn1thm1.2}
E_{\epsilon,\eta}^\delta(\m_\epsilon) \leqslant E_{\epsilon,\eta}^\delta(v_\epsilon) + o_\epsilon(1),
\end{equation}
\begin{equation}
\label{eqn3thm1.2}
\lim\limits_{\epsilon \rightarrow 0}
\left\Vert \m_\epsilon - v_\epsilon \right\Vert_{L^p(\partial\Omega)} = 0,
\ \ \text{ for every } p \in [1,+\infty),
\end{equation}
and
\begin{equation}
\label{eqn4thm1.2}
\lim\limits_{\epsilon \rightarrow 0}
\left\Vert \mathcal{J}(\m_\epsilon) - \mathcal{J}(v_\epsilon) \right\Vert_{(\text{Lip}(\Omega))^\ast} = 0.
\end{equation}
Using Lemma~\ref{LEM_4.1_IK21+DMI}, for every~$\epsilon>0$, there exists a lifting~$\phir_\epsilon \in H^1(\Omega)$ such that~$\m_\epsilon = e^{i\phir_\epsilon}$ and~$E_{\epsilon,\eta}^\delta(\m_\epsilon)=\mathcal{G}_{\epsilon}^\delta(\phir_\epsilon)$. Moreover, for every~$\epsilon>0$, the global Jacobian of~ $\m_\epsilon$ is given by
\begin{equation*}
\mathcal{J}(\m_\epsilon)=\mathcal{J}_{\text{bd}}(\m_\epsilon)=-\partial_\tau\phir_\epsilon
\text{ as a distribution in } H^{-1/2}(\partial\Omega).
\end{equation*}
Using~\eqref{eqn1thm1.2} and~$\mathcal{G}_{\epsilon}^\delta(\phir_\epsilon) = E_{\epsilon,\eta}^\delta(\m_\epsilon)$, we deduce from Theorem~\ref{THM_4.2.1_IK21+DMI} that, for a subsequence, there exists a family of integers~$(z_\epsilon)$ -- that we can all assume to be either even or odd, up to take a further subsequence -- such that~$(\phir_\epsilon - \pi z_\epsilon)$ converges strongly in~$L^p(\partial\Omega)$, for every~$p \in [1,+\infty)$, to a limit~$\phid_0$ that satisfies $\phid_0-g \in BV(\partial\Omega,\pi \mathbb{Z})$ with~$g$ such that~$e^{ig}=\tau=i\nu$ on~$\partial \Omega$. Let~$\phir_0$ be such that~$\phir_0=\phid_0$ if the integers~$z_\epsilon$ are all even, and~$\phir_0=\phid_0-\pi$ if the integers~$z_\epsilon$ are all odd. By definition of~$\phid_0$ and~$g$,~$\phir_0$ is a~BV lifting of~$\pm\tau$.
Moreover, as~$\left\vert e^{is}-e^{it} \right\vert \leqslant \left\vert s-t \right\vert$ for every~$s,t \in \mathbb{R}$, we have, for every~$\epsilon>0$,
\begin{align*}
\left\vert \m_\epsilon-e^{i\phir_0} \right\vert
= \left\vert e^{i(\phir_\epsilon-\pi z_\epsilon)} - e^{i\phid_0} \right\vert
\leqslant \left\vert (\phir_\epsilon-\pi z_\epsilon)-\phid_0 \right\vert.
\end{align*}
It follows that~$(V_\epsilon)_{\epsilon \downarrow 0}$ converges strongly to~$e^{i\phir_0}$ in~$L^p(\partial\Omega)$, for every~$p \in [1,+\infty)$. Combining this with~\eqref{eqn3thm1.2}, we deduce that~$(v_\epsilon)$ converges strongly to~$e^{i\phir_0}$ in~$L^p(\partial\Omega)$, for every~$p \in [1,+\infty)$.
By Theorem~\ref{THM_4.2.1_IK21+DMI}, we also have
\begin{equation*}
\partial_\tau \phir_0 = \partial_\tau \phid_0 = \kappa \mathcal{H}^1\llcorner\partial\Omega - \pi \sum_{j=1}^N d_j \Dirac_{a_j}
=-J \ \ \text{ as measure on } \partial \Omega
\end{equation*}
where, for every~$j \in \left\lbrace 1,...,N \right\rbrace$,~$a_j \in \partial \Omega$ are distinct points,~$d_j \in \mathbb{Z} \setminus \left\lbrace 0 \right\rbrace$ and~$\sum_{j=1}^N d_j = 2$.
We also get, again by Theorem~\ref{THM_4.2.1_IK21+DMI}, the convergence of~$(\partial_\tau\phir_\epsilon)$ to~$\partial_\tau\phid_0$ in~$W^{-1,p}(\partial\Omega)$ for every~$p \in (1,+\infty)$. For every Lipschitz function~$\zeta \in W^{1,\infty}(\Omega)$, using that~$\m_\epsilon=e^{i\phir_\epsilon}$ and integrating by parts, we have
\begin{align*}
\langle \mathcal{J}(\m_\epsilon),\zeta \rangle
&= -\int_\Omega \m_\epsilon \wedge \nabla\m_\epsilon \cdot \nabla^\perp\zeta \ \d x
 = -\int_\Omega \nabla\phir_\epsilon \cdot \nabla^\perp\zeta \ \d x
\\
& = -\langle \partial_\tau\phir_\epsilon,\zeta \rangle_{H^{-1/2}(\partial\Omega), H^{1/2}(\partial\Omega)}
 = -\langle \partial_\tau\phir_\epsilon,\zeta \rangle_{W^{-1,2}(\partial\Omega), W^{1,2}(\partial \Omega)}.
\end{align*}
Thus,
\begin{align*}
\sup\limits_{\lv \nabla \zeta \rv \leqslant 1} \lv \langle \mathcal{J}(v_\epsilon)-J,\zeta \rangle \rv
\leqslant C \lV \partial_\tau\phir_\epsilon-\partial_\tau\phir_0 \rV_{W^{-1,2}(\partial\Omega)} \rightarrow 0 \text{ as } \epsilon \rightarrow 0.
\end{align*}
Combining this with~\eqref{eqn4thm1.2}, we deduce that~$(\mathcal{J}(v_\epsilon))$ converges to~$J$ in~$(\mathrm{Lip}(\Omega))^\ast$.
Finally, since
\begin{equation*}
E_{\epsilon,\eta}^\delta(v_\epsilon)
\geqslant E_{\epsilon,\eta}^\delta(\m_\epsilon)-o_\epsilon(1)
= \mathcal{G}_{\epsilon}^\delta(\phir_\epsilon)-o_\epsilon(1)
\end{equation*}
thanks to~\eqref{eqn1thm1.2}, the lower bound~\eqref{eq_liminf_G_order1} for $\mathcal{G}_{\epsilon}^\delta(\phir_\epsilon)$ given by Theorem~\ref{THM_4.2.1_IK21+DMI} gives the expected lower bound for~$E_{\epsilon,\eta}^\delta(v_\epsilon)$ at the first order.
\end{proof}

We will now prove Theorem~\ref{THM_1.4_IK21+DMI}. Within those assumptions, we additionally prove the following statements.

\begin{itemize}
\item[(a)] \textbf{Penalty bound.}
\\
The penalty terms in the energy are of order~$O(1)$, i.e.
\begin{equation}
\label{limsup_penalty_terms}
\limsup\limits_{\epsilon \rightarrow 0} \left(
\frac{1}{\eta^2} \int_\Omega \left( 1-\vert v_\epsilon \vert^2 \right)^2 \d x
+ \frac{1}{2\pi\epsilon} \int_{\partial \Omega} (v_\epsilon \cdot \nu)^2 \d \mathcal{H}^1
\right) < +\infty.
\end{equation}
\item[(b)] \textbf{Lower bound for the energy near boundary vortex cores.}
\\
There exist~$r_0>0$,~$\epsilon_0>0$ and~$C>0$ such that the Dirichlet energy of~$v_\epsilon$ near the singularities~$\left\lbrace a_j \right\rbrace_{j \in \left\lbrace 1,...,N \right\rbrace}$ satisfies, for all~$\epsilon \in (0,\epsilon_0)$ and~$r \in (0,r_0)$,
\begin{align}
\label{lower_bound_near_vortices}
\int_{\Omega \cap \bigcup_j B_r(a_j)} \left\vert \nabla v_\epsilon \right\vert^2 \d x
- N \pi \log \frac{r}{\epsilon}
\geqslant - C.
\end{align}
\item[(c)] \textbf{DMI bound.}
\\
The Dzyaloshinskii-Moriya interaction energy is of order~$O(1)$, i.e.
\begin{align}
\label{limsup_DMI_term}
\limsup\limits_{\epsilon \rightarrow 0} \int_\Omega \left\vert \delta \cdot \nabla v_\epsilon \wedge v_\epsilon \right\vert \d x < +\infty.
\end{align}
\end{itemize}

\begin{proof}[Proof of Theorem~\ref{THM_1.4_IK21+DMI}]
Continuing as in the proof of Theorem~\ref{THM_1.2_IK21+DMI} (with the same notations), we now assume the stronger condition~\eqref{EQ_limsup_order2_2D}. By definition of~$\m_\epsilon$, Theorem~\ref{THM_3.1_IK21+DMI} and Lemma~\ref{LEM_for_THM_3.1_IK21+DMI},
\begin{equation*}
\mathcal{G}_{\epsilon}^\delta(\phir_\epsilon)
= E_{\epsilon,\eta}^\delta(\m_\epsilon)
\leqslant E_{\epsilon,\eta}^\delta(v_\epsilon) + o_\epsilon(1),
\end{equation*}
hence by~\eqref{EQ_limsup_order2_2D},
\begin{equation}
\label{hyp_limsup_order2_Gepsilondelta}
\limsup\limits_{\epsilon \rightarrow 0} \left( \mathcal{G}_{\epsilon}^\delta(\phir_\epsilon) - \pi \left\vert \log \epsilon \right\vert \sum_{j=1}^N \left\vert d_j \right\vert \right) < +\infty.
\end{equation}

\textit{Step 1: Proof of~\textit{(i)}.}
We deduce immediately from the above assumption and Theorem~\ref{THM_4.2.2_IK21+DMI} that~\eqref{hyp_limsup_G_order2} holds true and~$d_j \in \left\lbrace \pm 1 \right\rbrace$ for every~$j \in \left\lbrace 1,...,N \right\rbrace$. Since $E_{\epsilon,\eta}^\delta(v_\epsilon) \geqslant \mathcal{G}_{\epsilon}^\delta(\phir_\epsilon) - o_\epsilon(1)$, it follows that the desired lower bound for~$E_{\epsilon,\eta}^\delta(v_\epsilon)$ holds true.
\vspace{.1cm}

\textit{Step 2: Estimating the DMI term in~\eqref{limsup_DMI_term}.}
By Theorem~\ref{THM_3.1_IK21+DMI} and Lemma~\ref{LEM_for_THM_3.1_IK21+DMI}, we have
\begin{equation}
\begin{split}
\left\vert \int_\Omega \delta \cdot \nabla v_\epsilon \wedge v_\epsilon \ \d x
\right\vert
& \leqslant
\left\vert \int_\Omega \delta \cdot \nabla \m_\epsilon \wedge \m_\epsilon \ \d x
\right\vert
+ \left\vert \int_\Omega \delta \cdot \left( \nabla \m_\epsilon \wedge \m_\epsilon - \nabla v_\epsilon \wedge v_\epsilon \right) \d x
\right\vert
\\
& \leqslant \left\vert \int_\Omega \delta \cdot \nabla \m_\epsilon \wedge \m_\epsilon \ \d x
\right\vert
+\vert \delta \vert \lV \mathcal{J}(V_\epsilon)-\mathcal{J}(v_\epsilon) \rV_{(\mathrm{Lip}(\Omega))^\ast}
\\
& \leqslant
\left\vert \int_\Omega \delta \cdot \nabla \phir_\epsilon \ \d x
\right\vert
+ o_\epsilon(1).
\end{split}
\label{eqn1thm1.4}
\end{equation}
By Theorem~\ref{THM_4.2.2_IK21+DMI},~$(\nabla \phir_\epsilon)$ is bounded in~$L^1(\Omega)$ and we conclude with~\eqref{limsup_DMI_term}.
\vspace{.1cm}

\textit{Step 3: Compactness and proofs of~\eqref{limsup_penalty_terms} and~\eqref{lower_bound_near_vortices}.}
By~\eqref{limsup_DMI_term}, one checks that
\begin{equation*}
E_{\epsilon,\eta}^\delta(v_\epsilon)
= E_{\epsilon,\eta}^0(v_\epsilon)+O_\epsilon(1).
\end{equation*}
In particular,~$E_{\epsilon,\eta}^0(v_\epsilon)$ satisfies~\eqref{EQ_limsup_order2_2D} for~$\delta=0$. Therefore, we apply Theorem~1.4 in~\cite{IK21} and conclude this step.
\end{proof}

\begin{proof}[Proof of Theorem~\ref{THM_1.5_IK21+DMI}]
Let~$\phir_0 \colon \partial \Omega \rightarrow \mathbb{R}$ be such that~$\partial_\tau \phir_0 = \kappa \mathcal{H}^1 \llcorner \partial \Omega - \pi \sum_{j=1}^N d_j\Dirac_{a_j}$ \linebreak and~$e^{i\phir_0} \cdot \nu = 0$ in~$\partial \Omega \setminus \left\lbrace a_1,...,a_N \right\rbrace$. Let~$\phir_\ast$ be the harmonic extension of~$\phir_0$ to~$\Omega$.
For any~$\epsilon>0$, we consider~$(\phir_\epsilon)_{\epsilon \downarrow 0}$ as in Theorem~\ref{THM_4.2.3_IK21+DMI} and we set~$v_\epsilon=e^{i \phir_\epsilon}$. Then, for every~$\epsilon>0$,~$v_\epsilon \in H^1(\Omega,\mathbb{S}^1)$ and~$\mathcal{J}(v_\epsilon) = -\partial_\tau \phir_\epsilon$ as measure on~$\partial\Omega$.
Since, for every~$\epsilon>0$,
\begin{equation*}
\vert v_\epsilon - e^{i\phir_\ast} \vert
= \vert e^{i\phir_\epsilon}-e^{i\phir_\ast} \vert
\leqslant \vert \phir_\epsilon-\phir_\ast \vert,
\end{equation*}
it follows from Theorem~\ref{THM_4.2.3_IK21+DMI} that~$(v_\epsilon)_{\epsilon \downarrow 0}$ converges strongly to~$\phir_\ast$ in~$L^p(\Omega)$ and in~$L^p(\partial\Omega)$ for every~$p \in [1,+\infty)$, and that~$(\mathcal{J}(v_\epsilon))_{\epsilon \downarrow 0}$ converges to
\begin{equation*}
-\partial_\tau\phir_0 = -\kappa \mathcal{H}^1 \llcorner \partial\Omega + \pi \sum_{j=1}^N d_j \Dirac_{a_j}
\end{equation*}
in~$(\mathrm{Lip}(\Omega))^\ast$. Finally, the expected upper bounds at first and at second order for~$E_{\epsilon,\eta}^\delta(v_\epsilon)$ follow from~\eqref{eq_limsup_G_order1} for the first order, from~\eqref{eq_limsup_G_order2} for the second order, combined with the equality~$E_{\epsilon,\eta}^\delta(v_\epsilon)~=~\mathcal{G}_{\epsilon}^\delta(\phir_\epsilon)$.
\end{proof}

\subsection{Minimisers of the renormalised energy}
\label{subsection_minimisers_renormalised_energy}

For proving Corollary~\ref{COR_1.6_IK21+DMI}, we first need to prove that~$E_{\epsilon,\eta}^\delta$ admits minimisers in~$H^1(\Omega,\R^2)$, and secondly that the renormalised energy in~\eqref{DEF_renormalised_energy} admits minimisers corresponding to two boundary vortices of multiplicities 1, i.e.~$N=2$ and~$d_1=d_2=1$ (see Corollary~\ref{COR_7_IK22+DMI}).

\begin{lemma}
\label{lem-minimizer-EepsetaD}
Let~$\delta \in \R^2$ and~$\Omega \subset \R^2$ be a bounded, simply connected~$C^{1,1}$ domain. Assume~$\epsilon \rightarrow 0$ and~$\eta=\eta(\epsilon) \rightarrow 0$ in the regime~\eqref{DEF_regime2D}. There exists a minimiser of~$E_{\epsilon,\eta}^\delta$ over the set~$H^1(\Omega,\R^2)$ for small~$\epsilon,\eta>0$.
\end{lemma}

\begin{proof}
We use the direct method in the calculus of variations. First, we show that~$E_{\epsilon,\eta}^\delta$ is coercive in~$H^1$. Indeed, if~$v \in H^1(\Omega,\R^2)$, then choosing~$\eta$ small such that~$4\lv \delta \rv^2 \leqslant \frac{1}{2\eta^2}$, we estimate as in the proof of Lemma~\ref{LEM_for_THM_3.1_IK21+DMI}:
\begin{equation*}
\lv 2 \int_\Omega \delta \cdot \nabla v \wedge v \ \d x \rv
\leqslant \frac{1}{2} \int_\Omega \lv \nabla v \rv^2 \d x
+ 4 \lv \delta \rv^2 \int_\Omega (1-\vert v \vert^2)^2 \d x
+ 4 \lv \delta \rv^2 \lv \Omega \rv,
\end{equation*}
and we conclude
\begin{align*}
E_{\epsilon,\eta}^\delta(v)
& \geqslant \frac{1}{2} \int_\Omega \left( \lv \nabla v \rv^2 + \frac{1}{\eta^2}(1-\vert v \vert^2)^2 \right) \d x
-C
\gtrsim \lV v \rV_{H^1}^2-1.
\end{align*}
Second, one easily checks that~$E_{\epsilon,\eta}^\delta$ is lower semicontinuous in the weak~$H^1$ topology for two-dimensional domains~$\Omega$. Therefore, we conclude to the existence of minimisers of~$E_{\epsilon,\eta}^\delta$ for small~$\epsilon$ and~$\eta$.
\end{proof}

The following statement extends~\cite[Proposition~20]{IK22} by giving a formula for the renormalised energy~$W_\Omega^\delta(\lb(a_j,d_j)\rb)$ defined in~\eqref{DEF_renormalised_energy}. In particular, it shows that the limit in~\eqref{DEF_renormalised_energy} exists. The formula for the renormalised energy is computed using the solution of a Neumann problem as in~\cite{BBH}.

\begin{proposition}
\label{prop-expr-renenergy-psi-R}
Let~$\delta \in \R^2$,~$\Omega \subset \R^2$ be a bounded, simply connected~$C^{1,1}$ domain, and~$\kappa$ be the curvature of~$\dr \Omega$. Let~$\lb a_j \in \partial \Omega \rb_{1 \leqslant j \leqslant N}$ be~$N \geqslant 2$ distinct points and~$d_j \in \lb \pm 1 \rb$ be the corresponding multiplicities, for~$j \in \lb 1,...,N \rb$, that satisfy~$\sum_{j=1}^N d_j = 2$. Then the limit in~\eqref{DEF_renormalised_energy} exists and the renormalised energy of~$\lb(a_j,d_j)\rb$ satisfies
\begin{equation}
\label{exprWOmegaD-psiR}
\begin{split}
W_\Omega^\delta(\lb (a_j,d_j) \rb)
& = - 2\pi \sum_{1 \leqslant j<k \leqslant N} d_jd_k \log \lv a_j-a_k \rv
\\
& \quad
- \int_{\dr \Omega} \psi(\kappa+2\delta^\perp \cdot \nu) \d \mathcal{H}^1
+ \pi \sum_{j=1}^N d_j R(a_j),
\end{split}
\end{equation}
where~$\nu$ is the outer unit normal vector on~$\partial\Omega$ and~$\psi$ denotes the unique solution (up to an additive constant) in~$W^{1,q}(\Omega)$, for every~$q \in [1,2)$, of the inhomogeneous Neumann problem
\begin{equation}
\label{syst-psi}
\lb \begin{array}{rcll}
\Delta \psi &=& 0 & \text{ in } \Omega,
\\ \frac{\dr \psi}{\dr \nu} &=& -\kappa + \pi \sum\limits_{j=1}^N d_j \Dirac_{a_j} & \text{ on } \dr \Omega,
\end{array} \right.
\end{equation}
and~$R$ is the harmonic function given by
\begin{equation}
\label{defR}
R(z) = \psi(z) + \sum_{j=1}^N d_j \log \lv z-a_j \rv,
\end{equation}
for every~$z \in \Omega$. Moreover, we have~$R \in C^{0,\alpha}(\overline{\Omega}) \cap W^{s,p}(\Omega)$ for every~$\alpha \in (0,1)$,~$p \in [1,+\infty)$ and~$s \in [1,1+\frac{1}{p})$.
\end{proposition}

\begin{proof}
By the definition of~$W_\Omega^\delta(\lb(a_j,d_j)\rb)$ in~\eqref{DEF_renormalised_energy}, we have
\begin{equation*}
W_\Omega^\delta(\lb(a_j,d_j) \rb)
= W_\Omega^0(\lb(a_j,d_j) \rb)
- \lim\limits_{r \rightarrow 0} \int_{\Omega^r} 2 \delta \cdot \nabla \phir_\ast \ \d x,
\end{equation*}
where~$\Omega^r=\Omega \setminus \bigcup_{j=1}^N B_r(a_j)$ for~$r>0$, and~$\phir_\ast$ is the harmonic extension of~$\phir_0$ (given in Definition~\ref{DEF_phi0}) to~$\Omega$. First, by~\cite[Proposition 20]{IK22}, we deduce all the stated properties for~$\psi$ and~$R$; moreover
\begin{equation}
\label{exprWOmega0-psiR}
\begin{split}
W_\Omega^0(\lb (a_j,d_j) \rb)
& = -2\pi \sum_{1 \leqslant j<k \leqslant N} d_jd_k \log \lv a_j-a_k \rv
- \int_{\dr \Omega} \psi\kappa \ \d \mathcal{H}^1
+ \pi \sum_{j=1}^N d_j R(a_j).
\end{split}
\end{equation}
Moreover, any solution~$\psi$ of~\eqref{syst-psi} is clearly a harmonic conjugate of~$\phir_\ast$, in particular~$\nabla \phir_\ast = - \nabla^\perp \psi$, and~$\nabla \phir_\ast \in L^q(\Omega)$, for every~$q \in [1,2)$. It follows by dominated convergence theorem that
\begin{align*}
\lim\limits_{r \rightarrow 0} \int_{\Omega^r} 2 \delta \cdot \nabla \phir_\ast \ \d x
= 2 \int_\Omega \delta \cdot \nabla \phir_\ast \ \d x
= -2 \int_{\Omega} \delta \cdot \nabla^\perp \psi \ \d x
= 2 \int_{\dr \Omega} \psi (\delta^\perp \cdot \nu) \d \mathcal{H}^1
\end{align*}
and~\eqref{exprWOmegaD-psiR} follows.
\end{proof}

We now prove Theorem~\ref{THM_6_IK22+DMI}:

\begin{proof}[Proof of Theorem~\ref{THM_6_IK22+DMI}]
By~Proposition~\ref{prop-expr-renenergy-psi-R}, we have
\begin{equation}
\label{devWOmegaD}
W_\Omega^\delta(\lb (a_j,d_j) \rb)
= W_\Omega^0(\lb (a_j,d_j) \rb)
- 2 \int_{\dr \Omega} \psi (\delta^\perp \cdot \nu) \d \mathcal{H}^1,
\end{equation}
for any bounded, simply connected~$C^{1,1}$ domain $\Omega \subset \R^2$, with boundary curvature~$\kappa$. 
\begin{itemize}
\item[\textit{(i)}] We assume that~$\Omega=B_1$. By~\cite[Theorem~6]{IK22}, we have
\begin{equation}
\label{exprWOmegaD-B1-1}
W_{B_1}^0(\lb (a_j,d_j) \rb)
= - 2\pi \sum_{1\leqslant j<k \leqslant N} d_jd_k \log \lv a_j-a_k \rv,
\end{equation}
and, for every~$z \in B_1$, the solution~$\psi$ in~\eqref{syst-psi} is given (up to an additive constant) by
\begin{equation}
\label{expr-psi(w)}
\psi(z) = - \sum_{j=1}^N d_j \log \lv z-a_j \rv.
\end{equation}
By Green's formula,
\begin{equation*}
\int_{\dr B_1} \psi \delta^\perp \cdot \nu \ \d \mathcal{H}^1
= \int_{B_1} \mathrm{div}(\psi(z)\delta^\perp) \d z
= -\sum_{j=1}^N d_j \int_{B_1} \mathrm{div} \left( \delta^\perp\log\lv z-a_j \rv \right) \d z.
\end{equation*}
For any~$j \in \lb 1,...,N \rb$ and~$z \in B_1$,
\begin{align*}
\mathrm{div} \left( \delta^\perp\log\lv z-a_j \rv \right)
& = \delta^\perp \cdot \frac{z-a_j}{\lv z-a_j \rv^2}
= \Re\left( \frac{\overline{z-a_j}}{\lv z-a_j \rv^2} \delta^\perp \right)
= \Re\left( \frac{1}{z-a_j} \delta^\perp \right)
\end{align*}
with the identification~$\delta^\perp=\binom{-\delta_2}{\delta_1}=-\delta_2+i\delta_1$ and~$\Re(z)$ is the real part of~$z \in \C$, so that
\begin{align*}
\int_{B_1} \mathrm{div} \left( \delta^\perp\log\lv z-a_j \rv \right) \d z
& = \Re\left( \delta^\perp \int_{B_1} \frac{1}{z-a_j}\ \d z \right)
= -\pi \Re\left( \delta^\perp \overline{a_j} \right)
= \pi \delta \cdot a_j^\perp;
\end{align*}
above we used that for $a\in \partial B_1$, the function $f:B_1\to \C$, $f(z)= \frac{1}{z-a}$ for every $z\in B_1$ is holomorphic and integrable in $B_1$ so that the mean value theorem yields $f(0)=\frac{1}{\pi} \int_{B_1} f(z)\, dz$. We conclude
\begin{equation}
\label{exprWOmegaD-B1-2}
\int_{\dr B_1} \psi \delta^\perp \cdot \nu \ \d \mathcal{H}^1
= -\pi\sum_{j=1}^N d_j \delta \cdot a_j^\perp.
\end{equation}
\item[\textit{(ii)}] By~\cite[Theorem 6]{IK22}, we have
\begin{equation*}
\begin{split}
& W_\Omega^0(\lb (a_j,d_j) \rb)
\\
& \qquad = - 2\pi \sum_{1 \leqslant j<k \leqslant N} d_jd_k \log \lv \Psi(a_j)-\Psi(a_k) \rv + \pi \sum_{j=1}^N (d_j-1)\log \lv \partial_z\Psi(a_j) \rv
\\
& \qquad \quad + \int_{\dr \Omega}\kappa(w) \left( \sum_{j=1}^N d_j \log \lv \Psi(w)-\Psi(a_j) \rv - \log \lv \partial_z\Psi(w) \rv \right) \d \mathcal{H}^1(w),
\end{split}
\end{equation*}
and, for every~$w \in \Omega$, the solution~$\psi$ in~\eqref{syst-psi} is given (up to an additive constant) by
\begin{equation*}
\psi(w) = - \sum_{j=1}^N d_j \log \lv \Psi(w)-\Psi(a_j) \rv + \log \lv \partial_z\Psi(w) \rv.
\end{equation*}
The expected identity~\eqref{EQ_WOmegadelta} is a direct consequence of~\eqref{devWOmegaD} and the above identities.
\end{itemize}
\end{proof}

We now prove Corollary~\ref{COR_7_IK22+DMI}, that gives the existence of a minimising pair~$(a_1^\ast,a_2^\ast)$ for the renormalised energy when~$N=2$ and~$d_1=d_2=1$.

\begin{proof}[Proof of Corollary~\ref{COR_7_IK22+DMI}]
Let~$\Phi \colon \overline{B_1} \rightarrow \overline{\Omega}$ be a~$C^1$ conformal diffeomorphism with inverse~$\Psi = \Phi^{-1}$. Let~$\mathfrak{D} = (\dr \Omega \times \dr \Omega) \setminus \lb (a,a) : a \in \dr \Omega \rb$. By Theorem~\ref{THM_6_IK22+DMI}, for every~$(a_1,a_2) \in \mathfrak{D}$,
\begin{equation*}
W_\Omega^\delta(\lb (a_1,1),(a_2,1) \rb) = -2\pi \log \lv \Psi(a_1)-\Psi(a_2) \rv + F(a_1,a_2)
\end{equation*}
where for every~$(a_1,a_2) \in \partial\Omega \times \partial\Omega$,
\begin{equation*}
\begin{split}
F(a_1,a_2)=
\int_{\dr \Omega} (\kappa(w) + 2 \delta^\perp \cdot \nu(w))& (\log \lv \Psi(w)-\Psi(a_1) \rv + \log \lv \Psi(w)-\Psi(a_2) \rv
\\ & \quad 
- \log \vert \partial_z\Psi(w) \vert ) \d \mathcal{H}^1(w).
\end{split}
\end{equation*}
Setting~$b_1=\Psi(a_1) \in \partial B_1$ and~$b_2=\Psi(a_2) \in \partial B_1$, we get, after changing variables,
\begin{equation*}
\begin{split}
F(a_1,a_2)=\int_{\dr B_1} (\kappa(\Phi(z))+2\delta^\perp \cdot \nu(\Phi(z))) & (\log \lv z-b_1 \rv + \log \lv z-b_2 \rv
\\ & \quad + \log \lv \partial_z\Phi(z) \rv ) \lv \partial_z\Phi(z) \rv \d \mathcal{H}^1(z).
\end{split}
\end{equation*}
Thus~$F$ is bounded on~$\partial\Omega \times \partial\Omega$, since~$\kappa + 2 \delta^\perp \cdot \nu \in L^\infty(\dr\Omega)$ and the functions~$z \mapsto \log \lv z-b_1 \rv$ and~$z \mapsto \log \lv z-b_2 \rv$ are in~$L^1(\dr B_1)$. Moreover, the function~$(a_1,a_2) \in \mathfrak{D} \mapsto W_\Omega^\delta(\lb (a_1,1),(a_2,1) \rb)$ is continuous on~$\mathfrak{D}$.
Let~$(a_1^{(n)},a_2^{(n)}) \subset \mathfrak{D}$ be a minimising sequence for~$W_\Omega^\delta(\lb(\cdot,1),(\cdot,1) \rb)$, i.e.
\begin{equation*}
\lim\limits_{n \rightarrow +\infty} W_\Omega^\delta(\{ (a_1^{(n)},1),(a_2^{(n)},1) \}) = \inf\limits_\mathfrak{D} W_\Omega^\delta(\lb(\cdot,1),(\cdot,1) \rb).
\end{equation*}
Note that such a sequence exists because~$F$ is bounded in~$\mathfrak{D}$. As~$\dr \Omega \times \dr \Omega$ is compact, we can \linebreak assume (up to a subsequence) that~$(a_1^{(n)},a_2^{(n)})$ converges to some~$(a_1^\ast,a_2^\ast) \in \dr \Omega \times \dr \Omega$. \linebreak Since~$(W_\Omega^\delta(\{ (a_1^{(n)},1),(a_2^{(n)},1) \}))$ is bounded, then using the boundedness of~$F$, we deduce that the sequence~$(\log \vert \Psi(a_1^{(n)})-\Psi(a_2^{(n)}) \vert)$ is bounded. By taking the limits as~$n \rightarrow +\infty$, we deduce that~$\Psi(a_1^\ast) \neq \Psi(a_2^\ast)$, and since~$\Psi$ is injective,~$a_1^\ast \neq a_2^\ast$, i.e.~$(a_1^\ast,a_2^\ast) \in \mathfrak{D}$. Finally, by continuity of~$W_\Omega^\delta(\lb (\cdot,1),(\cdot,1) \rb)$ over~$\mathfrak{D}$, we get
\begin{equation*}
W_\Omega^\delta(\{ (a_1^\ast,1),(a_2^\ast,1) \}) = \inf\limits_\mathfrak{D} W_\Omega^\delta(\lb(\cdot,1),(\cdot,1) \rb).
\end{equation*}
\end{proof}

We prove Theorem~\ref{THM_minimisers_unit_disk}, that gives the configuration of the vortices that minimises the renormalised energy in the unit disk $B_1$ in $\R^2$.

\begin{proof}[Proof of Theorem~\ref{THM_minimisers_unit_disk}]
By Theorem~\ref{THM_6_IK22+DMI} and Corollary~\ref{COR_7_IK22+DMI}, there exists a pair of distinct \linebreak points~$(a_1^\ast,a_2^\ast) \in \dr B_1 \times \dr B_1$ that minimises the renormalised energy
\begin{equation*}
W_{B_1}^\delta(\lb(a_1,1),(a_2,1)\rb)
= -2\pi\log \lv a_1-a_2 \rv + 2\pi\delta \cdot(a_1^\perp+a_2^\perp)
\end{equation*}
defined for~$(a_1,a_2) \in \dr B_1 \times \dr B_1$ such that~$a_1 \neq a_2$. We write~$a_1^\ast=e^{i\phir_1}$ and~$a_2^\ast=e^{i\phir_2}$ for some~$\phir_1$,~$\phir_2 \in \R$. Since~$a_1^\ast \neq a_2^\ast$, we have~$\phir_1-\phir_2 \notin 2\pi\Z$. Computing the renormalised energy above in terms of~$\phir_1$,~$\phir_2$ and~$\delta=(\delta_1,\delta_2)$, we deduce that
\begin{equation*}
W_{B_1}^\delta(\lb(a_1^\ast,1),(a_2^\ast,1)\rb)
= 2\pi f(\phir_1,\phir_2)
\end{equation*}
where
\begin{equation*}
f(\phir_1,\phir_2)
= -\frac{1}{2}\log 2 -\frac{1}{2}\log (1-\cos(\phir_1-\phir_2))-\delta_1(\sin\phir_1+\sin\phir_2)+\delta_2(\cos \phir_1+\cos\phir_2)
\end{equation*}
is~$C^\infty(\R^2 \setminus S)$ where~$S=\lb (\phir_1,\phir_2) \in \R^2 : \phir_1-\phir_2 \in 2\pi\Z \rb$. Hence, the problem of minimising the renormalised energy turns in minimising~$f$ over~$\R^2 \setminus S$.
\vspace{.1cm}

\textit{Step 1 :} We compute the critical points of~$f$.

\noindent
For every~$(\phir_1,\phir_2) \in \R^2 \setminus S$,
\begin{equation}
\label{sys_critical_points}
\nabla f(\phir_1,\phir_2)
= \left( \begin{matrix}
-\frac{\sin (\phir_1-\phir_2)}{2(1-\cos (\phir_1-\phir_2))}-a_1^\ast\cdot \delta
\\
\frac{\sin (\phir_1-\phir_2)}{2(1-\cos (\phir_1-\phir_2))}-a_2^\ast\cdot \delta
\end{matrix}
\right)
=0
\ \ \Longleftrightarrow \ \
\lb \begin{array}{rcl}
\frac{\sin(\phir_1-\phir_2)}{1-\cos(\phir_1-\phir_2)}&=&-2a_1^\ast\cdot\delta
\\
(a_1^\ast+a_2^\ast)\cdot\delta&=&0
\end{array} \right. .
\end{equation}
Moreover,
\begin{align*}
(a_1^\ast+a_2^\ast)\cdot\delta=0
& \ \ \Longleftrightarrow \ \ 
\delta_1(\cos\phir_1+\cos\phir_2)+\delta_2(\sin\phir_1+\sin\phir_2)=0
\\
& \ \ \Longleftrightarrow \ \ 
2\cos\left(\frac{\phir_1-\phir_2}{2}\right)\left(\delta_1\cos\frac{\phir_1+\phir_2}{2}+\delta_2\sin\frac{\phir_1+\phir_2}{2}\right)=0
\\
& \ \ \Longleftrightarrow \ \
\cos \frac{\phir_1-\phir_2}{2}=0 \ \ \text{ or } \ \ \delta \cdot b=0 \text{ where } b=e^{i(\phir_1+\phir_2)/2}.
\end{align*}

\noindent
\textit{Case 1:}~$\cos \frac{\phir_1-\phir_2}{2}=0$. Then we have~$\phir_1-\phir_2=\pi \ (\text{mod } 2\pi)$, thus~$a_1^\ast$ and~$a_2^\ast$ are diametrically opposed (i.e.~$a_2^\ast=-a_1^\ast$). By the first equation in~\eqref{sys_critical_points} and since~$\delta=\lv \delta \rv e^{i\theta}$, we deduce that
\begin{align*}
a_1^\ast\cdot\delta=0
& \ \ \Longleftrightarrow \ \ 
\cos(\phir_1-\theta)=0
\ \ \Longleftrightarrow \ \ 
\phir_1=\theta+\frac{\pi}{2} \ (\text{mod } \pi).
\end{align*}
Hence,~$a_1^\ast=e^{i\phir_1}
=\pm ie^{i\theta}
=\pm \frac{1}{\lv \delta \rv} \delta^\perp$.

\noindent
\textit{Case 2:}~$\delta \cdot b=0$ where~$b=e^{i(\phir_1+\phir_2)/2}$ and~$\phir_1-\theta \neq \frac{\pi}{2} \ (\text{mod } \pi)$. Since~$\delta=\lv \delta \rv e^{i\theta}$, we have
\begin{align*}
\delta \cdot b =0
& \ \ \Longleftrightarrow \ \
\delta \perp b
\ \ \Longleftrightarrow \ \
\frac{\phir_1+\phir_2}{2} = \theta+\frac{\pi}{2} \ (\text{mod } \pi)
\ \ \Longleftrightarrow \ \
\phir_1+\phir_2 = 2\theta+\pi \ (\text{mod } 2\pi),
\end{align*}
thus~$a_1^\ast$ and~$a_2^\ast$ are symmetric with respect to~$\delta^\perp$. By the first equation in \eqref{sys_critical_points}, we have
\begin{align*}
-2\lv \delta \rv a_1^\ast \cdot e^{i\theta}=\frac{\sin(\phir_1-(2\theta+\pi-\phir_1))}{1-\cos(\phir_1-(2\theta+\pi-\phir_1))}
& \ \ \Longleftrightarrow \ \
-2\lv \delta \rv \cos(\phir_1-\theta)
= \frac{-\sin(\phir_1-\theta)}{\cos(\phir_1-\theta)}
\\
& \ \ \Longleftrightarrow \ \
2\lv \delta \rv \sin^2(\phir_1-\theta)+\sin(\phir_1-\theta)-2\lv \delta \rv=0
\\
& \ \ \Longleftrightarrow \ \
\lb \begin{array}{l}
2\lv \delta \rv X^2+X-2\lv \delta \rv=0
\\ X=\sin(\phir_1-\theta)
\end{array} \right. .
\end{align*}
As~$X=\sin(\phir_1-\theta) \geqslant -1$, we deduce that
\begin{align*}
\sin(\phir_1-\theta)
& = \sqrt{1+\frac{1}{16\lv \delta \rv^2}}-\frac{1}{4\lv \delta \rv} \in [-1,1],
\end{align*}
thus~$\phir_1=\theta+\theta_\delta \ (\text{mod } 2\pi)$ or~$\phir_1=\theta+\pi - \theta_\delta \ (\text{mod } 2\pi)$ where~$\theta_\delta=\arcsin\left(\sqrt{1+\frac{1}{16\lv \delta \rv^2}}-\frac{1}{4\lv \delta \rv}\right)$. We obtain that~$\phir_2=2\theta+\pi-\phir_1=\theta+\pi-\theta_\delta \ (\text{mod } 2\pi)$ or~$\phir_2=\theta+\theta_\delta \ (\text{mod } 2\pi)$. Up to interchanging~$\phir_1$ and~$\phir_2$, we will assume in the following that~$\phir_1=\theta+\theta_\delta \ (\text{mod } 2\pi)$.
\vspace{.1cm}

\textit{Step 2 :} We study the nature of the critical points of~$f$.

\noindent
For every~$(\phir_1,\phir_2) \in \R^2 \setminus S$, the Hessian matrix of~$f$ is
\begin{equation*}
Hess(f)(\phir_1,\phir_2)=
\left( \begin{matrix}
\frac{1}{2(1-\cos(\phir_1-\phir_2))}+a_1^\ast\cdot\delta^\perp
& \frac{-1}{2(1-\cos(\phir_1-\phir_2))}
\\
\frac{-1}{2(1-\cos(\phir_1-\phir_2))}
& \frac{1}{2(1-\cos(\phir_1-\phir_2))}+a_2^\ast\cdot\delta^\perp
\end{matrix} \right)
\end{equation*}
and its determinant is
\begin{align*}
h(\phir_1,\phir_2)
& = \frac{1}{2(1-\cos(\phir_1-\phir_2))}(a_1^\ast+a_2^\ast) \cdot \delta^\perp + (a_1^\ast\cdot\delta^\perp)(a_2^\ast\cdot\delta^\perp).
\end{align*}

\noindent
\textit{Case~1:} $\cos\frac{\phir_1-\phir_2}{2}=0$. Here, we have diametrically opposed points~$a_2^\ast=-a_1^\ast$ with~$a_1^\ast=\pm \frac{1}{\lv \delta \rv}\delta^\perp$, and
\begin{equation*}
h(\phir_1,\phir_2)
= (a_1^\ast\cdot\delta^\perp)(a_2^\ast\cdot\delta^\perp)
= -(a_1^\ast\cdot\delta^\perp)^2
= -\lv \delta \rv^2<0.
\end{equation*}
Hence,~$(\phir_1,\phir_2)$ cannot be a minimiser for our renormalised energy.\footnote{Note that the diametrically opposed configuration~$(a_1^\ast, -a_1^\ast)$ with~$a_1^\ast=\pm \frac{1}{\lv \delta \rv} \delta^\perp\in \partial B_1$ is an unstable critical point of the renormalised energy when~$\delta \neq 0$.} 

\noindent
\textit{Case~2:} $\delta \cdot b=0$ and~$\cos\frac{\phir_1-\phir_2}{2} \neq 0$. Here,~$a_1^\ast$ and~$a_2^\ast$ are symmetric with respect to~$\delta^\perp$, with (by our convention)~$\phir_1=\theta+\theta_\delta \ (\text{mod } 2\pi)$ and~$\phir_2=\theta+\pi-\theta_\delta \ (\text{mod } 2\pi)$ where~$\theta_\delta$ is given above. Then~$h(\phir_1,\phir_2)>0$, because
\begin{align*}
a_1^\ast\cdot\delta^\perp
= e^{i\phir_1} \cdot \lv \delta \rv e^{i(\theta+\pi/2)}
& = \lv \delta \rv \cos\left(\phir_1-\theta-\frac{\pi}{2}\right)
\\
& = \lv \delta \rv \sin(\phir_1-\theta)
= \lv \delta \rv \left( \sqrt{1+\frac{1}{16\lv \delta \rv^2}}-\frac{1}{4\lv \delta \rv}\right)>0,
\end{align*}
and~$a_2^\ast\cdot\delta^\perp=a_1^\ast\cdot\delta^\perp>0$, as~$a_2^\ast$ and~$a_1^\ast$ are symmetric with respect to~$\delta^\perp$. Moreover, the trace of~$Hess(f)(\phir_1,\phir_2)$ is positive, hence~$(\phir_1,\phir_2)=(\theta+\theta_\delta,\theta+\pi-\theta_\delta)$ minimizes~$f$.
\end{proof}

To conclude this section, we prove Corollary~\ref{COR_1.6_IK21+DMI}.

\begin{proof}[Proof of Corollary~\ref{COR_1.6_IK21+DMI}]
By Lemma~\ref{lem-minimizer-EepsetaD}, there exists a minimiser~$v_\epsilon$ of~$E_{\epsilon,\eta}^\delta$ on~$H^1(\Omega,\R^2)$ for small~$\varepsilon>0$. By Corollary~\ref{COR_7_IK22+DMI}, there exist two points~$a_1^\ast \neq a_2^\ast \in \dr \Omega$ such that
\begin{equation}
\label{eqn0cor1.6}
W_\Omega^\delta(\lb (a_1^\ast,1),(a_2^\ast,1) \rb)
= \min \lb W_\Omega^\delta(\lb (\tilde{a}_1,1),(\tilde{a}_2,1) \rb) : \tilde{a}_1 \neq \tilde{a}_2 \in \dr \Omega \rb.
\end{equation}
By Theorem~\ref{THM_1.5_IK21+DMI} applied to~$\lb(a^\ast_1,1),(a^\ast_2,1)\rb$, the minimisers~$v_\epsilon$ must satisfy
\begin{equation}
\label{eqn1cor1.6}
E_{\epsilon,\eta}^\delta(v_\epsilon)
\leqslant 2\pi \lv \log \epsilon \rv + W_\Omega^\delta(\lb (a_1^\ast,1),(a_2^\ast,1) \rb) + 2\gamma_0 + o_\epsilon(1).
\end{equation}
By Theorem~\ref{THM_1.2_IK21+DMI}\textit{(i)}, for a subsequence,~$(\mathcal{J}(v_\epsilon))_{\epsilon>0}$ converges as in~\eqref{EQ_convergence_jacobian_2D} to
\begin{equation*}
J = -\kappa \mathcal{H}^1 \llcorner \dr \Omega + \pi\sum_{j=1}^N d_j \Dirac_{a_j}
\end{equation*}
for~$N \geqslant 2$ distinct boundary points~$a_1,...,a_N \in \dr \Omega$, with~$d_1,...,d_N \in \Z \setminus \lb 0 \rb$ such that~$\sum_{j=1}^N d_j = 2$. By Theorem~\ref{THM_1.2_IK21+DMI}\textit{(ii)}, we also have
\begin{equation}
\label{eqn2cor1.6}
\liminf\limits_{\epsilon \rightarrow 0} \frac{1}{\lv \log \epsilon \rv} E_{\epsilon,\eta}^\delta(v_\epsilon) \geqslant \pi \sum_{j=1}^N \lv d_j \rv.
\end{equation}
Combining~\eqref{eqn1cor1.6} and~\eqref{eqn2cor1.6}, we get~$\sum_{j=1}^N \lv d_j \rv \leqslant 2$, hence~$\sum_{j=1}^N \left( \lv d_j \rv - d_j \right) \leqslant 0$ so that~$d_j=|d_j|>0$ for $1 \leqslant j\leqslant N$. In particular,~\eqref{EQ_limsup_order2_2D} holds thanks to~\eqref{eqn1cor1.6} and applying Theorem~\ref{THM_1.4_IK21+DMI}\textit{(i)}, we deduce that~$N=2$,~$d_1=d_2=1$ and
\begin{equation}
\label{eqn3cor1.6}
E_{\epsilon,\eta}^\delta(v_\epsilon)
\geqslant 2\pi \lv \log \epsilon \rv + W_\Omega^\delta(\lb (a_1,1),(a_2,1) \rb) + 2\gamma_0 + o_\epsilon(1).
\end{equation}
Combining~\eqref{eqn1cor1.6} and~\eqref{eqn3cor1.6}, and letting~$\varepsilon \rightarrow 0$, we get
\begin{equation*}
W_\Omega^\delta(\lb (a_1,1),(a_2,1) \rb) \leqslant W_\Omega^\delta(\lb (a_1^\ast,1), (a_2^\ast,1) \rb),
\end{equation*}
so by definition of~$W_\Omega^\delta(\lb (a_1^\ast,1), (a_2^\ast,1) \rb)$, we deduce that~$(a_1,a_2)$ is also a minimiser in~\eqref{eqn0cor1.6}. It follows then from~\eqref{eqn1cor1.6} and~\eqref{eqn3cor1.6} that
\begin{equation}
\label{eqn4cor1.6}
E_{\epsilon,\eta}^\delta(v_\epsilon)
= 2\pi \lv \log \epsilon \rv + W_\Omega^\delta(\lb (a_1,1),(a_2,1) \rb) + 2\gamma_0 + o_\epsilon(1).
\end{equation}
By~\eqref{limsup_DMI_term}, for a subsequence,~$(v_\epsilon)$ converges weakly in~$W^{1,q}(\Omega,\R^2)$ for every~$q \in [1,2)$, and strongly in~$L^p(\Omega,\R^2)$ for every~$p \in [1,+\infty)$, to~$e^{i\widehat{\phir}_0}$, where~$\widehat{\phir}_0$ is an extension to~$\Omega$ of a \linebreak function~$\phir_0 \in BV(\dr \Omega,\pi \Z)$ that satisfies
\begin{equation*}
\dr_\tau \phir_0 = \kappa \mathcal{H}^1 \llcorner \partial\Omega - \pi(\Dirac_{a_1}+\Dirac_{a_2}) \ \ \text{ as measure on } \dr \Omega
\end{equation*}
and~$e^{i\phir_0} \cdot \nu = 0$ in~$\dr \Omega \setminus \lb a_1,a_2 \rb$.

It remains to prove that~$\widehat{\phir}_0$ is harmonic in~$\Omega$. Let~$r>0$ be small. As for a subsequence,~$(\nabla v_\epsilon)_\epsilon$ converges weakly to~$\nabla (e^{i\widehat{\phir}_0})$ in~$L^{3/2}(\Omega \setminus (B_r(a_1) \cup B_r(a_2))$, we have for every test \linebreak function~$\zeta\in C^\infty_c(\Omega \setminus (B_r(a_1) \cup B_r(a_2)), \R^2)$,
\begin{align*}
\lim_{\epsilon \rightarrow 0} \int_{\Omega \setminus (B_r(a_1) \cup B_r(a_2))} \nabla v_\epsilon\cdot \zeta \ \d x
=\int_{\Omega \setminus (B_r(a_1) \cup B_r(a_2))} \nabla(e^{i\widehat{\phir}_0})\cdot \zeta \ \d x
\end{align*}
yielding 
\begin{equation*}
\int_{\Omega \setminus (B_r(a_1) \cup B_r(a_2))} \vert \nabla \widehat{\phir}_0 \vert^2 \d x
\leqslant \liminf\limits_{\epsilon \rightarrow 0} \int_{\Omega \setminus (B_r(a_1) \cup B_r(a_2))} \vert \nabla v_\epsilon \vert^2 \d x,
\end{equation*}
since~$\vert \nabla (e^{i\widehat{\phir}_0}) \vert = \vert \nabla \widehat{\phir}_0 \vert$. Moreover, using that~$v_\epsilon \rightarrow e^{i\widehat{\phir}_0}$ in~$L^3(\Omega)$ for a subsequence~$\epsilon \rightarrow 0$, we get
\begin{equation*}
\int_{\Omega \setminus (B_r(a_1) \cup B_r(a_2))} -\delta \cdot \nabla \widehat{\phir}_0 \ \d x
= \lim\limits_{\epsilon \rightarrow 0} \int_{\Omega \setminus (B_r(a_1) \cup B_r(a_2))} \delta \cdot \nabla v_\epsilon \wedge v_\epsilon \ \d x,
\end{equation*}
since~$\nabla (e^{i\widehat{\phir}_0}) \wedge e^{i\widehat{\phir}_0} = - \nabla \widehat{\phir}_0$. From these observations, we deduce
\begin{equation}
\label{eqn5cor1.6}
\begin{split}
& \liminf\limits_{r \rightarrow 0}
\left( \int_{\Omega \setminus (B_r(a_1) \cup B_r(a_2))} \left( \vert \nabla \widehat{\phir}_0 \vert^2 - 2 \delta \cdot \nabla \widehat{\phir}_0 \right) \d x
- 2 \pi \log \frac{1}{r} \right)
\\
& \leqslant
\liminf\limits_{r \rightarrow 0}
\liminf\limits_{\epsilon \rightarrow 0}
\left( \int_{\Omega \setminus (B_r(a_1) \cup B_r(a_2))} \left( \vert \nabla v_\epsilon \vert^2 + 2 \delta \cdot \nabla v_\epsilon \wedge v_\epsilon \right) \d x
- 2 \pi \log \frac{1}{r} \right).
\end{split}
\end{equation}
Inside the disks~$\Omega \cap (B_r(a_1) \cup B_r(a_2))$, by the compactness of~$(v_\epsilon)_\epsilon$, we have as above
\begin{align*}
&\lim_{\epsilon \rightarrow 0}\lv \int_{\Omega \cap (B_r(a_1) \cup B_r(a_2))} 2 \delta \cdot \nabla v_\epsilon \wedge v_\epsilon \ \d x \rv
=\lv \int_{\Omega \cap (B_r(a_1) \cup B_r(a_2))} 2 \delta \cdot \nabla \widehat{\phir}_0\ \d x \rv \\
&\quad \quad \leqslant 2\lv \delta \rv |\Omega \cap (B_r(a_1) \cup B_r(a_2))|^{1/3} \lV \nabla \widehat{\phir}_0 \rV_{L^{3/2}(\Omega)}
\leqslant C r^{2/3}
\end{align*}
for some constant~$C>0$ independent of~$\epsilon$ and~$r$, and thus
\begin{equation}
\label{eqn6cor1.6}
\liminf\limits_{r \rightarrow 0}
\liminf\limits_{\epsilon \rightarrow 0}
\int_{\Omega \cap (B_r(a_1) \cup B_r(a_2))} 2 \delta \cdot \nabla v_\epsilon \wedge v_\epsilon \ \d x
= 0.
\end{equation}
Moreover, by~\cite[Equation~(86)]{IK21}, we have
\begin{equation}
\label{eqn7cor1.6}
\liminf\limits_{r \rightarrow 0}
\liminf\limits_{\epsilon \rightarrow 0}
\left( E_{\epsilon,\eta}^0(v_\epsilon; \Omega \cap (B_r(a_1) \cup B_r(a_2))) - 2\pi\log \frac{r}{\epsilon} - 2\gamma_0 \right)
\geqslant 0.
\end{equation}
Note that, for every small~$r>0$,
\begin{align*}
&E_{\epsilon,\eta}^\delta(v_\epsilon) - 2\pi \lv \log \epsilon \rv - 2\gamma_0
 = E_{\epsilon,\eta}^0(v_\epsilon;\Omega \cap (B_r(a_1) \cup B_r(a_2))) -2\pi \log \frac{r}{\epsilon} - 2\gamma_0
\\
& \quad + \int_{\Omega \cap (B_r(a_1) \cup B_r(a_2))} 2 \delta \cdot \nabla v_\epsilon \wedge v_\epsilon \ \d x
 + E_{\epsilon,\eta}^\delta(v_\epsilon;\Omega \setminus (B_r(a_1) \cup B_r(a_2))) -2\pi \log \frac{1}{r}
\\
& \geqslant E_{\epsilon,\eta}^0(v_\epsilon;\Omega \cap (B_r(a_1) \cup B_r(a_2))) -2\pi \log \frac{r}{\epsilon} - 2\gamma_0
\\
& \quad + \int_{\Omega \cap (B_r(a_1) \cup B_r(a_2))} 2 \delta \cdot \nabla v_\epsilon \wedge v_\epsilon \ \d x
 + \int_{\Omega \setminus (B_r(a_1) \cup B_r(a_2))} \left( \vert \nabla v_\epsilon \vert^2 + 2 \delta \cdot \nabla v_\epsilon \wedge v_\epsilon \right) \d x
- 2 \pi \log \frac{1}{r}.
\end{align*}
By~\eqref{eqn4cor1.6},~\eqref{eqn5cor1.6},~\eqref{eqn6cor1.6} and~\eqref{eqn7cor1.6}, it follows by passing to the limit as~$\epsilon \rightarrow 0$ and then~$r \rightarrow 0$:
\begin{equation}
\label{eqn8cor1.6}
\begin{split}
& W_\Omega^\delta(\lb(a_1,1),(a_2,1)\rb)=
\liminf\limits_{\epsilon \rightarrow 0}
\left( E_{\epsilon,\eta}^\delta(v_\epsilon) - 2\pi \lv \log \epsilon \rv - 2\gamma_0 \right)
\\
& \qquad \geqslant
\liminf\limits_{r \rightarrow 0}
\left( \int_{\Omega \setminus (B_r(a_1) \cup B_r(a_2))} \left( \vert \nabla \widehat{\phir}_0 \vert^2 - 2 \delta \cdot \nabla \widehat{\phir}_0 \right) \d x
- 2 \pi \log \frac{1}{r} \right).
\end{split}
\end{equation}
Let~$\phir_\ast$ be the harmonic extension of~$\phir_0$ to~$\Omega$. Then~$\widehat{\phir}_0 - \phir_\ast \in W_0^{1,q}(\Omega)$, for every~$q \in [1,2)$. For every small~$r>0$,
\begin{align*}
\int_{\Omega \setminus (B_r(a_1) \cup B_r(a_2))}
2 \delta \cdot (\nabla \widehat{\phir}_0 - \nabla \phir_\ast) \d x
& = \int_\Omega 2\delta \cdot \nabla(\widehat{\phir}_0-\phir_\ast) \d x
\\
& \quad -2\int_{\Omega \cap(B_r(a_1)\cup B_r(a_2))} 2\delta \cdot \nabla (\widehat{\phir}_0-\phir_\ast) \d x.
\end{align*}
As
\begin{align*}
\lv -2\int_{\Omega \cap(B_r(a_1)\cup B_r(a_2))} 2\delta \cdot \nabla (\widehat{\phir}_0-\phir_\ast)\d x \rv
& \leqslant Cr^{2/3} \lV \nabla (\widehat{\phir}_0-\phir_\ast) \rV_{L^{3/2}(\Omega)}
 \leqslant Cr^{2/3},
\end{align*}
for some~$C>0$ independent of~$r$, by Green's formula and using that~$\widehat{\phir}_0=\phir_\ast$ on~$\dr\Omega$, we get
\begin{equation*}
\liminf\limits_{r \rightarrow 0} \int_{\Omega \setminus (B_r(a_1) \cup B_r(a_2))}
2 \delta \cdot (\nabla \widehat{\phir}_0 - \nabla \phir_\ast) \d x = 0.
\end{equation*}
Combining this observation with~\eqref{eqn8cor1.6} and the definition of~$W_\Omega^\delta(\lb(a_1,1),(a_2,1)\rb)$, we deduce that
\begin{align*}
& W_\Omega^\delta(\lb (a_1,1),(a_2,1) \rb)
\geqslant
\liminf\limits_{r \rightarrow 0} \int_{\Omega \setminus (B_r(a_1) \cup B_r(a_2))} \left( \vert \nabla \widehat{\phir}_0 \vert^2 - \vert \nabla \phir_\ast \vert^2 \right)\d x
\\
& \qquad \quad
+ \liminf\limits_{r \rightarrow 0} \left( \int_{\Omega \setminus (B_r(a_1) \cup B_r(a_2))} \left( \vert \nabla \phir_\ast \vert^2 - 2 \delta \cdot \nabla \phir_\ast \right)\d x
- 2 \pi \log \frac{1}{r} \right)
\\
& \qquad =
\liminf\limits_{r \rightarrow 0} \int_{\Omega \setminus (B_r(a_1) \cup B_r(a_2))} \left( \vert \nabla \widehat{\phir}_0 \vert^2 - \vert \nabla \phir_\ast \vert^2 \right)\d x
+ W_\Omega^\delta(\lb (a_1,1),(a_2,1) \rb),
\end{align*}
i.e.
\begin{equation*}
\liminf\limits_{r \rightarrow 0} \int_{\Omega \setminus (B_r(a_1) \cup B_r(a_2))} \left( \vert \nabla \widehat{\phir}_0 \vert^2 - \vert \nabla \phir_\ast \vert^2 \right) \d x
\leqslant 0.
\end{equation*}
To deduce that~$\widehat{\phir}_0=\phir_\ast$, thus~$\widehat{\phir}_0$ is harmonic in~$\Omega$, we proceed as follows:
\begin{equation*}
\vert \nabla \widehat{\phir}_0 \vert^2 - \vert \nabla \phir_\ast \vert^2 = \vert \nabla (\widehat{\phir}_0-\phir_\ast) \vert^2 + 2 \nabla \phir_\ast \cdot \nabla (\widehat{\phir}_0-\phir_\ast).
\end{equation*}
Since~$\phir_\ast$ behaves as the sum of an angular function around~$a_1$ and~$a_2$ and a harmonic \linebreak function~$h \in W^{1,p}(\Omega)$ for every~$p \in [1,+\infty)$ (as in Step~3 of the proof of~\cite[Theorem~4.2]{IK21}), then integration by parts yields for small~$r>0$:
\begin{align*}
&\int_{\Omega \setminus (B_r(a_1) \cup B_r(a_2))} \nabla \phir_\ast \cdot \nabla(\widehat{\phir}_0-\phir_\ast) \ \d x
 = \int_{\dr(\Omega \setminus (B_r(a_1) \cup B_r(a_2)))} \dr_\nu \phir_\ast (\widehat{\phir}_0-\phir_\ast) \ \d \mathcal{H}^1
\\
& = \int_{\Omega \cap \dr(B_r(a_1) \cup B_r(a_2))} \dr_\nu h (\widehat{\phir}_0-\phir_\ast) \ \d \mathcal{H}^1
= \int_{\Omega \cap (B_r(a_1) \cup B_r(a_2))} \nabla h \cdot \nabla (\widehat{\phir}_0 -\phir_\ast) \ \d x.
\end{align*}
By H\"older's inequality, we deduce that
\begin{align*}
& \lv \int_{\Omega \setminus (B_r(a_1) \cup B_r(a_2))} \nabla \phir_\ast \cdot \nabla(\widehat{\phir}_0-\phir_\ast) \ \d x \rv
\\
& \qquad \leqslant \lV \nabla h \rV_{L^3(\Omega \cap (B_r(a_1) \cup B_r(a_2)))} \lV \nabla (\widehat{\phir}_0-\phir_\ast) \rV_{L^{3/2}(\Omega)} \rightarrow 0 \quad \textrm{as } r \rightarrow 0.
\end{align*}
As a consequence,
\begin{equation*}
0 \geqslant \liminf\limits_{r \rightarrow 0} \int_{\Omega \setminus (B_r(a_1) \cup B_r(a_2))} \left( \vert \nabla \widehat{\phir}_0 \vert^2 - \vert \nabla \phir_\ast \vert^2 \right) \d x = \liminf\limits_{r \rightarrow 0} \int_{\Omega \setminus (B_r(a_1) \cup B_r(a_2))} \vert \nabla (\widehat{\phir}_0-\phir_\ast) \vert^2 \d x.
\end{equation*}
We deduce that~$\widehat{\phir}_0 = \phir_\ast + c$ for some constant~$c \in \R$, but~$\widehat{\phir}_0 = \phir_\ast$ on~$\dr \Omega$, hence~$c=0$ and~$\widehat{\phir}_0 = \phir_\ast$ is harmonic in~$\Omega$.
\end{proof}

\section{Three-dimensional model for maps $m_h \colon \Omega_h \subset \R^3 \rightarrow \mathbb{S}^2$}
\label{SECTION_3D_ENERGY}

\subsection[Reduction from the 3D model to a reduced 2D model]{Reduction from the three-dimensional model to a reduced two-dimensional model}
\label{subsection2.2.2}

We begin by relating the three-dimensional energy $E_h(m_h)$ given at \eqref{DEF_Eh} with the corresponding energy in the absence of the Dzyaloshinskii-Moriya interaction:
\begin{align}
\label{Eh-zero}
E_h^0(m_h)
& = \frac{1}{\lv \log \epsilon \rv}
\left( \frac{1}{h} \int_{\Omega_h} \lv \nabla m_h \rv^2 \d x
+ \frac{1}{h \eta^2} \int_{\R^3} \lv \nabla u_h \rv^2 \d x \right)
\end{align}
defined for maps~$m_h \colon \Omega_h =\Omega \times (0,h) \rightarrow \mathbb{S}^2$ and~$u_h \colon \mathbb{R}^3 \rightarrow \mathbb{R}$ satisfying~\eqref{EQ_Maxwell_u_h}.
This latter energy has been studied by Ignat and Kurzke~\cite{IK22}, hence the next lemma will allow us to use in this section the statements they obtained on~$E_h^0(m_h)$.

\begin{lemma}
\label{lemma_Enh0_bounded}
Let~$\Omega_h=\Omega \times (0,h)$ with~$\Omega \subset \R^2$ a bounded, simply connected~$C^{1,1}$ domain. Assume that~$\frac{1}{\eta^2}\vert \widehat{D} \vert=O_h( \sqrt{\vert \log \epsilon \vert})$ and consider a family of magnetizations~$\{ m_h \colon \Omega_h \rightarrow \mathbb{S}^2 \}_{h \downarrow 0}$ that satisfies
\begin{equation*}
\limsup\limits_{h \rightarrow 0} E_h(m_h) < +\infty.
\end{equation*}
Then~$\limsup\limits_{h \rightarrow 0} E_h^0(m_h) < +\infty$.
\end{lemma}

\begin{proof}
Using~$\lv m_h \rv = 1$ and Young's inequality, we obtain
\begin{align*}
\lv \log \epsilon \rv \lv E_h(m_h) - E_h^0(m_h) \rv
& = \frac{1}{h\eta^2} \lv \int_{\Omega_h} \widehat{D} : \nabla m_h \wedge m_h \ \d x \rv
 \leqslant \frac{1}{h\eta^2} \int_{\Omega_h} \vert \widehat{D} \vert \lv \nabla m_h \rv \d x
\\
& \leqslant \frac{1}{2h} \int_{\Omega_h} \left( \lv \nabla m_h \rv^2 + \left( \frac{\vert \widehat{D} \vert}{\eta^2} \right)^2 \right) \d x
\leqslant \frac{|\log \epsilon|}{2} E_h^0(m_h)+\frac{\lv \Omega \rv}{2} \left(\frac{\vert \widehat{D} \vert}{\eta^2} \right)^2.
\end{align*}
Thus,
\begin{equation}
\label{345}
E_h^0(m_h) \leqslant 2E_h(m_h) + \frac{\lv \Omega \rv}{\lv \log \epsilon \rv} \left( \frac{\vert \widehat{D} \vert}{\eta^2} \right)^2 
\end{equation}
and the conclusion follows in the regime $\frac{\vert \widehat{D} \vert}{\eta^2} \lesssim \sqrt{|\log \epsilon|}$.
\end{proof}

As a consequence, we deduce the existence of minimisers for the energy~$E_h$:

\begin{corollary}
\label{cor:mini}
Let~$\Omega_h=\Omega \times (0,h)$ with~$\Omega \subset \R^2$ a bounded, simply connected~$C^{1,1}$ domain. Then~$E_h$ admits a minimiser over the set~$H^1(\Omega_h, \mathbb{S}^2)$ for every~$h>0$.
\end{corollary}

\begin{proof}
We use the direct method in the calculus of variations. Indeed, for fixed~$h>0$, consider a minimising sequence~$(m_{n})_{n \in \N}$ of~$E_h$ in~$H^1(\Omega_h,\mathbb{S}^2)$; by~\eqref{345},~$(m_{n})_{n \in \N}$ is bounded in~$H^1(\Omega_h, \R^3)$, in particular, for a subsequence,~$(m_{n})_n$ converges weakly in~$H^1(\Omega_h,\R^3)$, strongly in~$L^2(\Omega_h,\R^3)$ and a.e. in~$\Omega_h$. As the Helmholtz projection~$P \colon m\in L^2(\Omega_h, \R^3) \rightarrow \nabla u\in L^2(\R^3,\R^3)$ is a linear bounded operator (via~\eqref{EQ_Maxwell_u_h}), the relative compactness of~$(m_{n})_{n \in \N}$ and the lower semicontinuity of~$E_h$ in the weak~$H^1(\Omega_h,\R^3)$ topology yield the conclusion.
\end{proof}

In the next lemma, we estimate the~DMI in~$E_h(m_h)$ with respect to the~DMI corresponding to the~$x_3$-average of~$m_h$.

\begin{lemma}
\label{lemma_reduction_dmi}
Let~$\Omega_h=\Omega \times (0,h)$ with~$\Omega \subset \R^2$ a bounded, simply connected~$C^{1,1}$ domain. In the regime~$\frac{1}{\eta^2}\vert \widehat{D} \vert=O_h( \sqrt{\vert\log \epsilon\vert})$, consider a family of magnetizations~$\{ m_h \colon \Omega_h \rightarrow \mathbb{S}^2 \}_{h \downarrow 0}$ that satisfies~$\limsup\limits_{h \rightarrow 0} E_h(m_h) < +\infty$. If~$\overline{m}_h$ is the~$x_3$-average of~$m_h$ defined in~\eqref{DEF_mh_mean}, then
\begin{equation*}
\frac{1}{\lv \log \epsilon \rv} \frac{1}{h\eta^2}\int_{\Omega_h} {\widehat{D}} : \nabla m_h \wedge m_h \ \d x
= \frac{1}{\lv \log \epsilon \rv} \frac{1}{\eta^2} \int_{\Omega} {\widehat{D}}' : \nabla' \overline{m}_h \wedge \overline{m}_h \ \d x' +O(R(h)) \ \ \text{ as } h \rightarrow 0,
\end{equation*}
where~${\widehat{D}}'=(\widehat{D}_1, \widehat{D}_2) \in \R^{3\times 2}$ and 
\begin{equation}
\label{expression_R2}
R(h)=h\frac{\vert \widehat{D}_1 \vert + \vert \widehat{D}_2 \vert}{\eta^2}+\frac{\vert \widehat{D}_3 \vert}{\eta^2}\cdot \frac{1}{\sqrt{|\log \eps|}}.
\end{equation}
\end{lemma}

\begin{proof}
Writing
\begin{equation*}
{\widehat{D}} : \nabla m_h \wedge m_h
= {\widehat{D}}' : \nabla' m_h \wedge \overline{m}_h
+ {\widehat{D}}' : \nabla' m_h \wedge (m_h - \overline{m}_h)
+ {\widehat{D}}_3 \cdot \dr_3 m_h \wedge m_h,
\end{equation*}
we obtain the decomposition
\begin{align*}
\int_{\Omega_h} {\widehat{D}} : \nabla m_h \wedge m_h \ \d x
=I_1+I_2+I_3
\end{align*}
with
\begin{equation*}
I_1 = \int_{\Omega_h} {\widehat{D}}' : \nabla' m_h \wedge \overline{m}_h \ \d x
= \sum_{j=1}^2 \int_{\Omega_h} {\widehat{D}}_j \cdot \dr_j m_h \wedge \overline{m}_h \ \d x,
\end{equation*}
\begin{equation*}
I_2 = \int_{\Omega_h} {\widehat{D}}' : \nabla' m_h \wedge (m_h - \overline{m}_h) \ \d x
= \sum_{j=1}^2 \int_{\Omega_h} {\widehat{D}}_j \cdot \dr_j m_h \wedge (m_h - \overline{m}_h) \ \d x,
\end{equation*}
and
\begin{equation*}
I_3 = \int_{\Omega_h} {\widehat{D}}_3 \cdot \dr_3 m_h \wedge m_h \ \d x.
\end{equation*}
For calculating~$I_1$, by Fubini's theorem, we have 
\begin{align*}
I_1
 & = \sum_{j=1}^2 \int_\Omega \int_0^h \left( {\widehat{D}}_j \cdot \dr_j m_h(x',x_3) \wedge \overline{m}_h(x') \right) \d x_3 \d x'\\
&= \sum_{j=1}^2 \int_\Omega {\widehat{D}}_j \cdot \dr_j\left( \int_0^h m_h(x',x_3) \d x_3 \right) \wedge \overline{m}_h(x') \d x'
 = h \sum_{j=1}^2 \int_\Omega {\widehat{D}}_j \cdot \dr_j \overline{m}_h \wedge \overline{m}_h \ \d x'.
\end{align*}
For~$I_2$, by the Cauchy-Schwarz and Poincar\'e-Wirtinger inequalities, we have
\begin{align*}
\lv I_2 \rv
& \leqslant \sum_{j=1}^2 \int_{\Omega_h} \lv {\widehat{D}}_j \cdot \dr_j m_h \wedge (m_h - \overline{m}_h) \rv \d x
\\
& \leqslant \sum_{j=1}^2 \vert {\widehat{D}}_j \vert \left( \int_{\Omega_h} \lv \dr_j m_h \rv^2 \d x \right)^{1/2} \left( \int_{\Omega_h} \lv m_h - \overline{m}_h \rv^2 \d x \right)^{1/2}
\\
& \leqslant h \sum_{j=1}^2 \vert {\widehat{D}}_j \vert \left( \int_{\Omega_h} \lv \dr_j m_h \rv^2 \d x \right)^{1/2} \left( \int_{\Omega_h} \lv \dr_3 m_h \rv^2 \d x \right)^{1/2}\\
& \leqslant h \sum_{j=1}^2 \vert {\widehat{D}}_j \vert \int_{\Omega_h} \lv \nabla m_h \rv^2 \d x
\leqslant  h^2 \left( \vert {\widehat{D}}_1 \vert + \vert {\widehat{D}}_2 \vert \right) \lv \log \epsilon \rv E_h^0(m_h).
\end{align*}
Finally, for~$I_3$, since~$\lv m_h \rv = 1$, the Cauchy-Schwarz inequality yields
\begin{align*}
\lv I_3 \rv
& \leqslant \int_{\Omega_h} \vert {\widehat{D}}_3 \cdot \dr_3 m_h \wedge m_h \vert \d x
  \leqslant \sqrt{|\Omega|} \vert {\widehat{D}}_3 \vert \sqrt{h} \left( \int_{\Omega_h} \lv \dr_3 m_h \rv^2 \d x \right)^{1/2}
\leqslant \sqrt{|\Omega|} \vert {\widehat{D}}_3 \vert h \sqrt{\lv \log \epsilon \rv E_h^0(m_h)},
\end{align*}
where~$E_h^0$ is defined in~\eqref{Eh-zero}.
Combining the above estimates, we obtain
\begin{align*}
&
\frac{1}{\lv \log \epsilon \rv} \frac{1}{h\eta^2}\left|\int_{\Omega_h} {\widehat{D}} : \nabla m_h \wedge m_h \ \d x
-h \int_{\Omega} {\widehat{D}}' : \nabla' \overline{m}_h \wedge \overline{m}_h \ \d x'\right|
\\
& \qquad \leqslant h \frac{\vert {\widehat{D}}_1 \vert + \vert {\widehat{D}}_2 \vert}{\eta^2} E_h^0(m_h) +\sqrt{|\Omega|} \frac{\vert {\widehat{D}}_3 \vert}{\eta^2} \sqrt{\frac{E_h^0(m_h)}{\lv \log \epsilon \rv}}.
\end{align*}
Since~$\limsup\limits_{h \rightarrow 0} E_h^0(m_h)<+\infty$ by Lemma~\ref{lemma_Enh0_bounded}, the conclusion follows. 
\end{proof}

We have now all the ingredients to prove Theorem~\ref{THM_reduction3D2D}.

\begin{proof}[Proof of Theorem~\ref{THM_reduction3D2D}]
As $\frac{1}{\eta^2}\vert \widehat{D} \vert=O_h(\sqrt{\vert \log \epsilon \vert})$, by Lemma~\ref{lemma_reduction_dmi}, we have
\begin{align*}
E_h(m_h) &= E_h^0(m_h) + \frac{1}{\lv \log \epsilon \rv} \frac{1}{h\eta^2} \int_{\Omega_h} {\widehat{D}} : \nabla m_h \wedge m_h \ \d x\\
&= E_h^0(m_h) + \frac{1}{\lv \log \epsilon \rv} \frac{1}{\eta^2} \int_{\Omega} {\widehat{D}}' : \nabla' \overline{m}_h \wedge \overline{m}_h \ \d x'
-O(R(h))
\end{align*}
where~$E_h^0(m_h)$ is given in~\eqref{Eh-zero}. By~\eqref{expression_R2}, we have that~$O(R(h))=o_h(1)$ in the regime~\eqref{DEF_regime3D} and~$O(R(h))=o_h(1/\vert \log \epsilon \vert)$ in the regime~\eqref{DEF_regime3D_log} because~$h\ll \frac{1}{\vert \log h \vert}\ll \epsilon \ll \frac{1}{\vert \log \epsilon \vert}$ in the regime~\eqref{DEF_regime3D}.
\vspace{.1cm}

\textit{Step~1: Estimating~$E_h^0(m_h)$.} We denote by~$\overline{m}_h=(\overline{m}_h',\overline{m}_{h,3})$ the~$x_3$-average of~$m_h$ defined in~\eqref{DEF_mh_mean}. By~\cite[Theorem 1 and Inequality (1.11)]{IK22}, we have in the absence of DMI:
\begin{align*}
E_h^0(m_h) \geqslant \frac{1}{\vert \log \epsilon \vert} \left(E^0_{\eps, \eta}(\overline{m}_h')+\int_\Omega \vert\nabla' \overline{m}_{h,3}\vert^2\, dx'\right)-o_h(R_1(h))
\end{align*}
where~$R_1(h)=1$ in the regime~\eqref{DEF_regime3D} and~$R_1(h)=1/ \vert\log\epsilon\vert$ in the regime~\eqref{DEF_regime3D_log}. Moreover, if~$m_h$ is independent of~$x_3$ (i.e.,~$m_h=\overline{m}_h$) and~$m_{h,3}=0$ (i.e.,~$m_h=(m'_h,0) \colon \Omega \rightarrow \mathbb{S}^1\times \{0\}$), then the above inequality becomes equality.
\vspace{.1cm}

\textit{Step 2: Estimating the DMI term.} Since~$\limsup\limits_{h \rightarrow 0} E_h^0(m_h)<+\infty$ by Lemma~\ref{lemma_Enh0_bounded}, Step~1 implies that~$\displaystyle \int_\Omega \vert\nabla'\overline{m}_h\vert^2 \d x'=O(\vert\log\epsilon\vert)$. We claim that
\begin{equation}
\label{789}
\frac{1}{\vert\log\epsilon\vert} \int_\Omega \left( \frac{1}{\eta^2}\widehat{D}' : \nabla' \overline{m}_h \wedge \overline{m}_h - 2\delta \cdot \nabla' \overline{m}_h' \wedge \overline{m}_h' \right) \d x=o_h(R_1(h)).
\end{equation}
As~$\vert \overline{m}_h \vert \leqslant 1$, the term involving~$\widehat{D}_{11}$ is estimated using the Cauchy-Schwarz inequality:
\begin{equation*}
\left\vert \int_\Omega \frac{\widehat{D}_{11}}{\eta^2}  \left( \overline{m}_{h,3} \dr_1 \overline{m}_{h,2} - \overline{m}_{h,2} \dr_1 \overline{m}_{h,3} \right) \d x' \right\vert
\leqslant
\frac{\vert\widehat{D}_{11}\vert}{\eta^2} \int_\Omega \vert\nabla'\overline{m}_{h}\vert \d x'
=\frac{\vert\widehat{D}_{11}\vert}{\eta^2}O(\sqrt{\vert\log\epsilon\vert}).
\end{equation*}
Similar estimates hold for the~DMI terms involving~$\widehat{D}_{12}$,~$\widehat{D}_{21}$ and~$\widehat{D}_{22}$. For the term involving~$\widehat{D}_{k3}$ with~$k=1,2$, we have by the Cauchy-Schwarz inequality: 
\begin{equation*}
\left\vert \int_\Omega \left(\frac{\widehat{D}_{k3}}{\eta^2}-2\delta_k \right) \partial_k \overline{m}'_h \wedge \overline{m}'_h \ \d x' \right\vert
\leqslant \left\vert \frac{\widehat{D}_{k3}}{\eta^2}-2\delta_k \right\vert \int_\Omega \vert\nabla' \overline{m}_{h}\vert \d x'
= \left\vert \frac{\widehat{D}_{k3}}{\eta^2}-2\delta_k\right\vert O(\sqrt{\vert\log\epsilon\vert}).
\end{equation*}
The choice of our regimes~\eqref{DEF_regime3D} and~\eqref{DEF_regime3D_log} yields the claim~\eqref{789}.
\end{proof}

\subsection{Gamma-convergence of the three-dimensional energy}
\label{subsection2.2.3}

In this section, we prove the~$\Gamma$-convergence for~$E_h(m_h)$, i.e. Theorems~\ref{THM_9_order1_IK22+DMI},~\ref{THM_9_order2_IK22+DMI},~\ref{THM_10_IK22+DMI} and Corollary~\ref{COR_11_IK22+DMI}.
Recall that the regime~\eqref{DEF_regime3D} (and also~\eqref{DEF_regime3D_log}) implies regime~\eqref{DEF_regime2D}, see Footnote~\ref{foot}.

\begin{proof}[Proof of Theorem~\ref{THM_9_order1_IK22+DMI}]
Let~$\lb m_h \colon \Omega_h \rightarrow \mathbb{S}^2 \rb_{h \downarrow 0}$ be a family of magnetizations such \linebreak that~$\limsup_{h \rightarrow 0} E_h(m_h) <\infty$. Denoting~$\overline{m}_h=(\overline{m}'_h, \overline{m}_{h,3}) \colon \Omega \rightarrow \R^3$ the~$x_3$-average of~$m_h$ given in~\eqref{DEF_mh_mean}, we have by Theorem~\ref{THM_reduction3D2D} in the regime~\eqref{DEF_regime3D},
\begin{equation*}
\limsup\limits_{\epsilon \rightarrow 0} \frac{1}{\lv \log \epsilon \rv} E_{\epsilon,\eta}^\delta(\overline{m}'_h) \leqslant \limsup\limits_{h \rightarrow 0} E_h(m_h)<\infty.
\end{equation*}
As~$\lv \log \epsilon \rv \ll \lv \log \eta \rv$, we can apply Theorem~\ref{THM_1.2_IK21+DMI} to~$v_\epsilon:= \overline{m}'_h \colon \Omega \rightarrow \R^2$. More precisely, by Theorem~\ref{THM_1.2_IK21+DMI}\textit{(i)}, for a subsequence,~$(\mathcal{J}(\overline{m}'_h))$ converges (in the sense of~\eqref{EQ_convergence_jacobian_2D}) to the measure~$J$ given in~\eqref{EQ_limit_jacobian_2D}.
Moreover, for a subsequence,~$(\overline{m}'_h \vert_{\dr\Omega})$ converges to~$e^{i\phir_0} \in BV(\dr\Omega,\mathbb{S}^1)$ in~$L^p(\dr\Omega)$, for every~$p \in [1,+\infty)$, where~$\phir_0 \in BV(\dr\Omega,\pi\Z)$ is a lifting of the tangent field~$\pm\tau$ on~$\dr\Omega$ determined (up to a constant in~$\pi\Z$) by~$\dr_\tau\phir_0 = -J$ as measure on~$\dr\Omega$. 
Since~$\overline{m}_{h,3}^2 \leqslant 1 - \lv \overline{m}'_h \rv^2$ and~$(\overline{m}'_h)_h$ converges to~$e^{i\phir_0}$ in~$L^2(\partial \Omega,\R^2)$ with~$\vert e^{i\phir_0} \vert=1$, we deduce that~$(\overline{m}_{h,3})_h$ converges to zero in~$L^2(\dr\Omega)$ and almost everywhere on~$\dr\Omega$ (up to a subsequence). As~$\lv \overline{m}_{h,3} \rv \leqslant 1$, by dominated convergence theorem, we get that~$(\overline{m}_{h,3})_h$ converges to zero in~$L^p(\dr\Omega)$ for every~$p \in [1,+\infty)$.
For proving~\textit{(ii)}, we apply Theorem~\ref{THM_1.2_IK21+DMI}\textit{(ii)} and Theorem~\ref{THM_reduction3D2D} to get
\begin{align*}
\pi \sum_{j=1}^N \lv d_j \rv
& \leqslant \liminf\limits_{h \rightarrow 0} \frac{1}{\lv \log \epsilon \rv} E_{\epsilon,\eta}^\delta(\overline{m}'_h)
\leqslant \liminf\limits_{h \rightarrow 0} E_h(m_h).
\end{align*}
\end{proof}

\begin{proof}[Proof of Theorem~\ref{THM_9_order2_IK22+DMI}]
By~\eqref{EQ_limsup_order2_3D} and Theorem~\ref{THM_reduction3D2D} in the regime~\eqref{DEF_regime3D_log}, we have
\begin{align*}
\infty
& > \limsup\limits_{h \rightarrow 0} \lv \log \epsilon \rv \left( E_h(m_h) - \pi \sum_{j=1}^N \lv d_j \rv \right)
\geqslant \limsup\limits_{h \rightarrow 0} \left( E_{\epsilon,\eta}^\delta(\overline{m}'_h) - \pi \lv \log \epsilon \rv \sum_{j=1}^N \lv d_j \rv \right).
\end{align*}
As~$\lv \log \epsilon \rv \ll \lv \log \eta \rv$, we can apply Theorem~\ref{THM_1.4_IK21+DMI} to~$v_\epsilon:= \overline{m}'_h$.
More precisely, by Theorem~\ref{THM_1.4_IK21+DMI}\textit{(i)}, we have~$d_j \in \lb \pm 1 \rb$ for every~$j \in \lb 1,...,N \rb$, so that~$\sum_{j=1}^N \lv d_j \rv = N$, and
\begin{align*}
W_\Omega^\delta(\lb(a_j,d_j)\rb) + N\gamma_0
& \leqslant \liminf\limits_{h \rightarrow 0} \left( E_{\epsilon,\eta}^\delta(\overline{m}'_h) - N\pi \lv \log \epsilon \rv \right)
\leqslant \liminf\limits_{h \rightarrow 0} \lv \log \epsilon \rv (E_h(m_h)-N\pi).
\end{align*}
Let us prove~\textit{(ii)}. By the proof of Theorem~\ref{THM_reduction3D2D}, we have
\begin{align*}
E_h(m_h)
& = E_h^0(m_h) + \frac{1}{\lv \log \epsilon \rv} \frac{1}{\eta^2} \int_{\Omega} {\widehat{D}}' : \nabla' \overline{m}_h \wedge \overline{m}_h \ \d x'
-O(R(h))
\\
& = E_h^0(m_h) + \frac{1}{\lv \log \epsilon \rv} \int_{\Omega} 2\delta \cdot \nabla' \overline{m}_h' \wedge \overline{m}_h' \ \d x-o_h(R_1(h))-O(R(h))
\end{align*}
where~$o_h(R_1(h))+O(R(h))=o_h(\frac{1}{\lv \log \epsilon \rv})$ in the regime~\eqref{DEF_regime3D_log}. By~\eqref{limsup_DMI_term},~$\int_{\Omega} 2\delta \cdot \nabla' \overline{m}_h' \wedge \overline{m}_h' \ \d x=O(1)$. By~\eqref{EQ_limsup_order2_3D}, we deduce that
\begin{equation*}
\limsup\limits_{h \rightarrow 0} \lv \log \epsilon \rv  \left( E_h^0(m_h)-N\pi \right)
\leqslant \limsup\limits_{h \rightarrow 0} \lv \log \epsilon \rv \left( E_h(m_h)-N\pi \right)+ O(1)<\infty.
\end{equation*}
Hence, we can apply~\cite[Theorem~9\textit{(iv)}]{IK22} and we obtain the expected compactness of~$(m_h)_h$ and properties of their limit points.
\end{proof}

\begin{proof}[Proof of Theorem~\ref{THM_10_IK22+DMI}]
In the regime~\eqref{DEF_regime3D}, we have~$\vert\log\epsilon\vert \ll \vert\log\eta\vert$ and we construct the family~$\{v_\epsilon \colon \Omega \rightarrow \mathbb{S}^1 \}$ as in Theorem~\ref{THM_1.5_IK21+DMI}. Set~$m_h \colon (x',x_3)\in \Omega_h \mapsto (v_\epsilon(x'),0)\in \mathbb{S}^1 \times \{0\}$.
For every~$h>0$,~$m_h$ is clearly independent of~$x_3$, thus~$m_h = \overline{m}_h$ and by Theorem~\ref{THM_1.5_IK21+DMI}, it follows that the global Jacobian~$\mathcal{J}(m'_h) = \mathcal{J}(v_\epsilon)$ converges (in the sense~\eqref{EQ_convergence_jacobian_2D}) to the measure~$J$ given in \eqref{EQ_limit_jacobian_2D}. Also,~$(m'_h)_h$ converges strongly to~$e^{i\phir_\ast}$ in~$L^p(\Omega,\R^2)$ and in~$L^p(\partial \Omega,\R^2)$, for every~$p \in [1,+\infty)$, where~$\phir_\ast$ is the harmonic extension to~$\Omega$ of a boundary lifting~$\phir_0$ of the tangent field~$\pm\tau$ on~$\partial\Omega$ that satisfies~$\partial_\tau \phir_0 = -J$ as measure on~$\partial \Omega$. By Theorem~\ref{THM_reduction3D2D} in the regime~\eqref{DEF_regime3D} combined with Theorem~\ref{THM_1.5_IK21+DMI}, we have
\begin{equation*}
E_h(m_h)
= \frac{1}{\vert\log\epsilon\vert} E_{\epsilon,\eta}^\delta(v_\epsilon) - o_\epsilon(1)
= \pi \sum_{j=1}^N \vert d_j \vert - o_\epsilon(1).
\end{equation*}
Furthermore, if~$d_j \in \lb \pm 1 \rb$ for every~$j \in \lb 1,...,N \rb$, then by Theorem~\ref{THM_reduction3D2D} in the regime~\eqref{DEF_regime3D_log} and Theorem~\ref{THM_1.5_IK21+DMI}, 
\begin{align*}
\vert\log\epsilon\vert \left( E_h(m_h)-N\pi \right)
& = \vert\log\epsilon\vert \left(\frac{1}{\vert\log\epsilon\vert} E_{\epsilon,\eta}^\delta(v_\epsilon) - N \pi - o\left(\frac{1}{\vert\log\epsilon\vert}\right) \right)
\\
&  =W_\Omega^\delta(\lb(a_j,d_j)\rb)+N\gamma_0-o_h(1).
\end{align*}
\end{proof}

\begin{proof}[Proof of Corollary~\ref{COR_11_IK22+DMI}]
By Corollary~\ref{COR_7_IK22+DMI}, there exist two points~$a_1^\ast \neq a_2^\ast \in \dr\Omega$ such that
\begin{equation*}
W_\Omega^\delta(\lb(a_1^\ast,1),(a_2^\ast,1)\rb) = \min \lb W_\Omega^\delta(\lb(\tilde{a}_1,1),(\tilde{a}_2,1)\rb) : \tilde{a}_1 \neq \tilde{a}_2 \in \dr\Omega \rb.
\end{equation*}
Let~$(m_h)_h$ be a family of minimisers of~$E_h$ on~$H^1(\Omega_h,\R^3)$ (such a family exists by Corollary \ref{cor:mini}). In the regime~\eqref{DEF_regime3D}, by Theorem~\ref{THM_10_IK22+DMI} applied to~$\lb(a_1^\ast,1),(a_2^\ast,1)\rb$, the minimisers~$m_h$ must satisfy
\begin{equation}
\label{eq1_cor1.6_3D}
\limsup\limits_{h \rightarrow 0} E_h(m_h) \leqslant 2\pi.
\end{equation}
Hence, we can apply Theorem~\ref{THM_9_order1_IK22+DMI}\textit{(i)} and we deduce that, for a subsequence, the global Jacobian~$(\mathcal{J}(\overline{m}_h'))_h$ converges as in~\eqref{EQ_convergence_jacobian_3D} to
\begin{equation*}
J = -\kappa \mathcal{H}^1 \llcorner \dr\Omega + \sum_{j=1}^N d_j\Dirac_{a_j}
\end{equation*}
for~$N \geqslant 1$ distinct boundary points~$a_1,...,a_N \in \dr\Omega$, with~$d_1,...,d_N \in \Z \setminus \lb 0 \rb$ such that~$\sum_{j=1}^N d_j = 2$. Moreover, by Theorem~\ref{THM_9_order1_IK22+DMI}\textit{(ii)}, we also have
\begin{equation}
\label{eq2_cor1.6_3D}
\liminf\limits_{h \rightarrow 0} E_h(m_h) \geqslant \pi \sum_{j=1}^N \vert d_j \vert.
\end{equation}
Combining~\eqref{eq1_cor1.6_3D} and~\eqref{eq2_cor1.6_3D}, we get~$\sum_{j=1}^N \vert d_j \vert \leqslant 2=\sum_{j=1}^N d_j$, hence~$\sum_{j=1}^N (\vert d_j \vert - d_j ) \leqslant 0$ so that for every $j \in \lb 1,...,N \rb$, $\vert d_j \vert=d_j$. It follows that
\begin{equation*}
\lim\limits_{h \rightarrow 0} E_h(m_h) = 2\pi,
\end{equation*}
and two cases can occur: either~$N=1$ and~$d_1=2$, or~$N=2$ and~$d_1=d_2=1$. Hence, there are two boundary points~$a_1,a_2 \in \dr\Omega$ (that might a-priori  coincide) such that
\begin{equation*}
J=-\kappa \mathcal{H}^1 \llcorner \dr\Omega + \pi(\Dirac_{a_1}+\Dirac_{a_2}).
\end{equation*}
We now assume that the regime~\eqref{DEF_regime3D_log} holds. By Theorem~\ref{THM_10_IK22+DMI}, we necessarily have
\begin{equation}
\label{eq3_cor1.6_3D}
\vert \log \epsilon \vert \left( E_h(m_h)-2\pi \right)
\leqslant W_\Omega^\delta(\lb(a_1^\ast,1),(a_2^\ast,1) \rb) + 2\gamma_0 + o_\epsilon(1) \ \ \text{ as } h \rightarrow 0.
\end{equation}
Hence, we can apply Theorem~\ref{THM_9_order2_IK22+DMI}\textit{(i)} and we deduce that~$d_1=d_2=1$,~$N=2$ (thus~$a_1 \neq a_2$) and
\begin{equation}
\label{eq4_cor1.6_3D}
\vert \log \epsilon \vert \left( E_h(m_h)-2\pi \right)
\geqslant W_\Omega^\delta(\lb(a_1,1),(a_2,1)\rb) + 2\gamma_0 + o_\epsilon(1) \ \ \text{ as } h \rightarrow 0.
\end{equation}
Combining~\eqref{eq3_cor1.6_3D} and~\eqref{eq4_cor1.6_3D}, we deduce that
\begin{equation*}
W_\Omega^\delta(\lb(a_1,1),(a_2,1)\rb)
\leqslant W_\Omega^\delta(\lb(a_1^\ast,1),(a_2^\ast,1) \rb).
\end{equation*}
By definition of~$W_\Omega^\delta(\lb(a_1^\ast,1),(a_2^\ast,1) \rb)$, we have
\begin{equation*}
W_\Omega^\delta(\lb(a_1,1),(a_2,1)\rb)
= W_\Omega^\delta(\lb(a_1^\ast,1),(a_2^\ast,1) \rb),
\end{equation*}
and thus
\begin{equation*}
\lim\limits_{\epsilon \rightarrow 0} \vert \log \epsilon \vert \left( E_h(m_h) - 2\pi \right) = W_\Omega^\delta(\lb(a_1,1),(a_2,1)\rb) + 2\gamma_0.
\end{equation*}
\end{proof}

\begin{appendices}

\section{Energy estimates near boundary vortices}
\label{SUBSECTION_energy_estimates_near_boundary_vortices}

For the second order term in the asymptotic expansion of the energy by~$\Gamma$-convergence, we need to estimate the energy~$\mathcal{G}_\epsilon^\delta$ defined in~\eqref{DEF_Gepsdelta} in a neighborhood of the boundary vortices. Up to use a conformal map and make a blow-up near a boundary vortex, we can consider the following localised functional.

\begin{notation}
Let~$\delta \in \R^2$ and~$\epsilon>0$. For any open set~$G \subset \mathbb{R}_+^2=\mathbb{R}\times(0,+\infty)$ and~$\psi \colon G \rightarrow \mathbb{R}$, we define the localised functional
\begin{equation}
\label{DEF_Fepsdelta}
F_{\epsilon}^\delta(\psi;G)
= \int_{G} \left( \left\vert \nabla \psi \right\vert^2 - 2 \delta \cdot \nabla \psi \right) \d x
+ \frac{1}{2\pi\epsilon} \int_{\overline{G} \cap (\mathbb{R} \times \left\lbrace 0 \right\rbrace)} \sin^2 \psi(x_1,0) \ \d x_1.
\end{equation}
\end{notation}

\begin{notation}
\label{NOT_phistar_phiepsstar}
For every~$(x_1,x_2) \in \mathbb{R}_+^2=\mathbb{R}\times(0,+\infty)$, we set
\begin{equation*}
\phid^\ast(x_1,x_2)
=\arg(x_1+ix_2)=\frac{\pi}{2}-\arctan \left( \frac{x_1}{x_2} \right),
\end{equation*}
and
\begin{equation*}
\phid_\epsilon^\ast(x_1,x_2)
=\arg(x_1+i(x_2+2\pi\epsilon))=\frac{\pi}{2}-\arctan \left( \frac{x_1}{x_2+2\pi\epsilon} \right).
\end{equation*}
\end{notation}

\begin{lemma}
\label{LEM_energy_near_vortices_gamma2}
Let~$\delta = (\delta_1,\delta_2) \in \R^2$. For any~$r \in (0,1)$, set~$I_r=(-r,r)$ and~$B_r^+=B_r \cap \R_+^2$. Then the function
\begin{equation}
\label{DEF_function_gamma2}
r \in (0,1) \mapsto
\liminf\limits_{\epsilon \rightarrow 0}
\left( \inf\limits_{\psi=\phid_\epsilon^\ast \text{ on } \partial B_r^+ \setminus I_r} F_{\epsilon}^\delta(\psi;B_r^+)-\pi \log \frac{r}{\epsilon} \right)
\end{equation}
has a limit as~$r \rightarrow 0$ given by~$\pi\log\frac{e}{4\pi}$.
\end{lemma}

\begin{proof}
In the case~$\delta=0$, Cabr\'e and Sola-Moral\`es~\cite[Lemma~3.1]{CSM05} showed that, for $r \in (0,1)$,
\begin{equation}
\label{CSM05_lemma3.1}
\inf\limits_{\psi=\phid_\epsilon^\ast \text{ on } \partial B_r^+ \setminus I_r} F_{\epsilon}^0(\psi;B_r^+)
= F_\epsilon^0(\phid_\epsilon^\ast;B_r^+),
\end{equation}
i.e.,~$\phid_\epsilon^\ast$ is a minimiser of~\eqref{CSM05_lemma3.1}. 
Furthermore, Ignat and Kurzke~\cite[Lemma~4.14]{IK21} proved that
\begin{equation}
\label{IK21_lemma4.14}
\lim\limits_{r \rightarrow 0}
\liminf\limits_{\epsilon \rightarrow 0}
\left( F_{\epsilon}^0(\phid_\epsilon^\ast;B_r^+)-\pi \log \frac{r}{\epsilon} \right)
= \pi\log \frac{e}{4\pi}.
\end{equation}
We will use these statements to prove that~$\pi\log\frac{e}{4\pi}$ is also the limit as~$r \rightarrow 0$ of the function in~\eqref{DEF_function_gamma2}, for all~$\delta \in \R^2$. We define
\begin{equation}
\gamma_2
:=\limsup\limits_{r \rightarrow 0}
\liminf\limits_{\epsilon \rightarrow 0}
\left( \inf\limits_{\psi=\phid_\epsilon^\ast \text{ on } \partial B_r^+ \setminus I_r} F_{\epsilon}^\delta(\psi;B_r^+)-\pi \log \frac{r}{\epsilon} \right).
\end{equation}

\textit{Step 1 :} We first prove that~$\gamma_2 \leqslant \pi\log\frac{e}{4\pi}$.
\vspace{0.1cm}

\noindent
To do so, we observe that~$\inf \lb F_\epsilon^\delta(\psi;B_r^+) \colon \psi=\phid_\epsilon^\ast \text{ on } \dr B_r^+ \setminus I_r \rb \leqslant F_\epsilon^\delta(\phid_\epsilon^\ast;B_r^+)$, hence
\begin{equation}
\label{eqn_lemmegamma2_1}
\gamma_2 \leqslant \limsup\limits_{r \rightarrow 0} \liminf\limits_{\epsilon \rightarrow 0} \left( F_\epsilon^\delta(\phid_\epsilon^\ast;B_r^+) - \pi \log\frac{r}{\epsilon} \right).
\end{equation}
Let $r \in (0,1)$. By Green's formula, we have
\begin{align*}
F_{\epsilon}^\delta(\phid_\epsilon^\ast;B_r^+)
& = F_\epsilon^0(\phid_\epsilon^\ast;B_r^+) - 2 \int_{B_r^+} \delta \cdot \nabla \phid_\epsilon^\ast \ \d x
= F_\epsilon^0(\phid_\epsilon^\ast;B_r^+) - 2 \int_{\dr B_r^+} (\delta \cdot \nu) \phid_\epsilon^\ast \ \d \mathcal{H}^1,
\end{align*}
where $\nu=(\nu_1,\nu_2)$ is the outer unit normal vector on $\partial B_r^+$. Moreover,
\begin{equation*}
\left\vert \int_{\dr B_r^+} (\delta \cdot \nu) \phid_\epsilon^\ast \ \d \mathcal{H}^1 \right\vert
\leqslant \lv \delta \rv \lV \phid_\epsilon^\ast \rV_{L^\infty} \lv \dr B_r^+ \rv
\leqslant Cr
\end{equation*}
for some~$C>0$ independent of~$\epsilon$ and~$r$, since~$\lv \phid_\epsilon^\ast \rv \leqslant \pi$ and~$\left\vert \partial B_r^+ \right\vert = (\pi+2)r$. We deduce that
\begin{equation*}
F_{\epsilon}^\delta(\phid_\epsilon^\ast;B_r^+)-\pi \log \frac{r}{\epsilon}
= F_{\epsilon}^0(\phid_\epsilon^\ast;B_r^+)-\pi \log \frac{r}{\epsilon}+O(r).
\end{equation*}
Using~\eqref{IK21_lemma4.14} and~\eqref{eqn_lemmegamma2_1}, the conclusion follows.
\vspace{.1cm}

\textit{Step 2 :} We prove that~$\widetilde{\gamma}_2
:=\liminf\limits_{r \rightarrow 0}
\liminf\limits_{\epsilon \rightarrow 0}
\left( \inf\limits_{\psi=\phid_\epsilon^\ast \text{ on } \partial B_r^+ \setminus I_r} F_{\epsilon}^\delta(\psi;B_r^+)-\pi \log \frac{r}{\epsilon} \right)
\geqslant \pi\log\frac{e}{4\pi}$.
\vspace{0.1cm}

\noindent
Let~$r \in (0,1)$ and~$\psi \colon B_r^+ \rightarrow \R$ such that~$\psi=\phid_\epsilon^\ast$ on~$\dr B_r^+ \setminus I_r$. As~$\phid_\epsilon^\ast \in [0,\pi]$ in~$\R_+^2$, then
\begin{equation*}
-\lv \delta_2 \rv \leqslant \phid_\epsilon^\ast(x_1,x_2)-\delta_2x_2 \leqslant \pi + \lv \delta_2 \rv
\ \ \text{ on } \dr B_r^+ \setminus I_r.
\end{equation*}
Obviously, there exists an integer~$N \geqslant 1$ (independent of~$\epsilon$ and~$\psi$) such that
\begin{equation*}
-N\pi \leqslant \phid_\epsilon^\ast(x_1,x_2)-\delta_2x_2 \leqslant N\pi
\ \ \text{ on } \dr B_r^+ \setminus I_r, \text{ for every } r \in (0,1).
\end{equation*}
We will replace~$\psi$ by a new function~$\widetilde{\psi} \colon B_r^+ \rightarrow \R$ that is appropriately bounded and with less energy: more precisely, we define for~$x=(x_1,x_2) \in B_r^+$:
\begin{equation}
\label{psitildezeropi}
\widetilde{\psi}(x_1,x_2) = \delta_2x_2 +
\lb \begin{array}{ll}
\psi(x_1,x_2)-\delta_2x_2 & \text{ if } -N\pi \leqslant \psi(x_1,x_2)-\delta_2x_2 \leqslant N\pi,
\\
-N\pi & \text{ if } \psi(x_1,x_2)-\delta_2x_2 \leqslant -N\pi,
\\
N\pi & \text{ if } \psi(x_1,x_2)-\delta_2x_2 \geqslant N\pi.
\end{array}
\right.
\end{equation}
Note that~$\widetilde{\psi}(\cdot,0) \in [-N\pi,N\pi]$ on~$I_r$, and~$\widetilde{\psi}=\psi=\phid_\epsilon^\ast$ on~$\dr B_r^+ \setminus I_r$.

Let us show that~$F_\epsilon^\delta(\widetilde{\psi};B_r^+) \leqslant F_\epsilon^\delta(\psi;B_r^+)$. First, we note that on~$I_r$, $\widetilde{\psi}(\cdot,0)$ is equal either to~$\psi(\cdot,0)$, or to~$-N\pi$, or to~$N\pi$. Thus~$\sin^2 \widetilde{\psi}(\cdot,0) \leqslant \sin^2 \psi(\cdot,0)$, yielding
\begin{equation*}
\int_{I_r} \sin^2 \widetilde{\psi}(x_1,0) \ \d x_1 \leqslant \int_{I_r} \sin^2 \psi(x_1,0) \ \d x_1.
\end{equation*}
Inside~$B_r^+$, we note that 
\begin{align*}
\int_{B_r^+} \left( \vert \nabla \widetilde{\psi} \vert^2 - 2\delta \cdot \nabla \widetilde{\psi} \right) \d x
& = 
\int_{B_r^+} \left( \vert \dr_1\widetilde{\psi} \vert^2 - 2\delta_1 \dr_1\widetilde{\psi} \right) \d x
+ \int_{B_r^+} \vert \dr_2\widetilde{\psi}-\delta_2 \vert^2 \d x
- \delta_2^2 \lv B_r^+ \rv. 
\end{align*}
On the one hand, using~\eqref{psitildezeropi}, we have~$\vert \dr_1 \widetilde{\psi} \vert \leqslant \vert \dr_1 \psi \vert$ in~$B_r^+$, yielding~$\int_{B_r^+} \vert \dr_1\widetilde{\psi} \vert^2 \d x \leqslant \int_{B_r^+} \vert \dr_1\psi \vert^2 \d x$. Moreover, 
\begin{equation*}
\int_{B_r^+} \dr_1 \widetilde{\psi} \ \d x
= \int_{\dr B_r^+ \setminus I_r} \widetilde{\psi} \nu_1 \ \d \mathcal{H}^1
= \int_{\dr B_r^+ \setminus I_r} \phid_\epsilon^\ast \nu_1 \ \d \mathcal{H}^1
= \int_{B_r^+} \dr_1 \psi \ \d x.
\end{equation*}
On the other hand, using~\eqref{psitildezeropi},~$\vert \dr_2 \widetilde{\psi}-\delta_2 \vert \leqslant \vert \dr_2 \psi-\delta_2 \vert$ in~$B_r^+$, yielding
\begin{align*}
\int_{B_r^+} \vert \dr_2 \widetilde{\psi}-\delta_2 \vert^2 \d x \leqslant \int_{B_r^+} \vert \dr_2 \psi-\delta_2 \vert^2 \d x.
\end{align*}
Combining the above inequalities, we deduce that~$F_\epsilon^\delta(\widetilde{\psi};B_r^+) \leqslant F_\epsilon^\delta(\psi;B_r^+)$.
By definition of~$F_\epsilon^\delta(\widetilde{\psi};B_r^+)$ and using Green's formula,
\begin{align*}
F_\epsilon^\delta(\widetilde{\psi};B_r^+)
& = F_\epsilon^0(\widetilde{\psi};B_r^+)
-\int_{B_r^+} 2\delta \cdot \nabla \widetilde{\psi} \ \d x
\\
& = F_\epsilon^0(\widetilde{\psi};B_r^+)
- \int_{\dr B_r^+ \setminus I_r} 2(\delta \cdot \nu) \widetilde{\psi} \ \d \mathcal{H}^1
+ \int_{I_r} 2\delta_2 \widetilde{\psi} (x_1,0) \ \d x_1
\\
& \stackrel{\eqref{CSM05_lemma3.1}}{\geqslant} F_\epsilon^0(\phid_\epsilon^\ast;B_r^+)
- \int_{\dr B_r^+ \setminus I_r} 2(\delta \cdot \nu) \phid_\epsilon^\ast \ \d \mathcal{H}^1
+ \int_{I_r} 2\delta_2 \widetilde{\psi} (x_1,0) \ \d x_1.
\end{align*}
Using that~$\phid_\epsilon^\ast \in [0,\pi]$ and~$\widetilde{\psi}(\cdot,0) \in [-N\pi,N\pi]$ on~$I_r$, we get
\begin{align*}
F_\epsilon^\delta(\widetilde{\psi};B_r^+)
& \geqslant F_\epsilon^0(\phid_\epsilon^\ast;B_r^+)-Cr
\end{align*}
for some~$C>0$ independent of~$\epsilon$ and~$r$. Hence,
\begin{align*}
\widetilde{\gamma}_2
& \geqslant \liminf\limits_{r \rightarrow 0} \left( 
\liminf\limits_{\epsilon \rightarrow 0} \left( F_\epsilon^0(\phid_\epsilon^\ast;B_r^+)-\pi\log\frac{r}{\epsilon} \right)-Cr \right)
\stackrel{\eqref{IK21_lemma4.14}}{=} \pi\log\frac{e}{4\pi}.
\end{align*}
Combining Step~1 and Step~2, we deduce that the function in~\eqref{DEF_function_gamma2} has a limit at~$r \rightarrow 0$ given by~$\pi\log\frac{e}{4\pi}$.
\end{proof}

\begin{lemma}
\label{LEM_4.14_IK21+DMI}
Let~$\delta = (\delta_1,\delta_2) \in \R^2$. For any~$r \in (0,1)$, set~$I_r=(-r,r)$ and~$B_r^+=B_r \cap \R_+^2$. Then the function
\begin{equation}
\label{DEF_function_gamma1}
r \in (0,1) \mapsto
\liminf\limits_{\epsilon \rightarrow 0}
\left( \inf\limits_{\psi=\phid^\ast \text{ on } \partial B_r^+ \setminus I_r} F_{\epsilon}^\delta(\psi;B_r^+)-\pi \log \frac{r}{\epsilon} \right),
\end{equation}
has a limit as~$r \rightarrow 0$ given by~$\pi\log\frac{e}{4\pi}$.
\end{lemma}

\begin{proof}
We define
\begin{equation*}
\gamma_1
:=\limsup\limits_{r \rightarrow 0}
\liminf\limits_{\epsilon \rightarrow 0}
\left( \inf\limits_{\psi=\phid^\ast \text{ on } \partial B_r^+ \setminus I_r} F_{\epsilon}^\delta(\psi;B_r^+)-\pi \log \frac{r}{\epsilon} \right),
\end{equation*}
\begin{equation*}
\widetilde{\gamma}_1
:=\liminf\limits_{r \rightarrow 0}
\liminf\limits_{\epsilon \rightarrow 0}
\left( \inf\limits_{\psi=\phid^\ast \text{ on } \partial B_r^+ \setminus I_r} F_{\epsilon}^\delta(\psi;B_r^+)-\pi \log \frac{r}{\epsilon} \right),
\end{equation*}
and recall the quantities~$\gamma_2$ and~$\widetilde{\gamma}_2$ introduced in Lemma~\ref{LEM_energy_near_vortices_gamma2}.
\vspace{.1cm}

\textit{Step 1.} We show that~$\widetilde{\gamma}_2 \leqslant \widetilde{\gamma}_1$.
Let $r \in (0,1)$ and~$\psi \colon B_r^+ \rightarrow \R$ such that~$\psi=\phid^\ast$ on~$\partial B_r \setminus I_r$. Consider the family of functions~$(\phid_\epsilon)_{\epsilon>0}$ defined in~$B^+_{r(1+r)} \setminus B_r$ as
\begin{equation*}
\phid_\epsilon(x_1,x_2)
= \arg \left( x_1+i \left( x_2+2\pi\epsilon \frac{\sqrt{x_1^2+x_2^2}-r}{r^2} \right)\right) \in [0,\pi]
\end{equation*}
for every~$(x_1,x_2) \in B^+_{r(1+r)} \setminus B_r$, and that satisfies, for every~$\epsilon>0$,~$\phid_\epsilon=\phid_\epsilon^\ast$ on the half-circle~$\partial B^+_{r(1+r)} \setminus I_{r(1+r)}$ and~$\phid_\epsilon=\phid^\ast$ on the half-circle~$\partial B_r^+ \setminus I_r$.
We extend~$\psi=\phid_\epsilon$ on~$B_{r(1+r)}^+ \setminus B_r$ and compute
\begin{align}
\label{eq_lem_4.14_IK21+DMI_1}
F_{\epsilon}^\delta(\psi;B^+_{r(1+r)})-\pi\log \frac{r(1+r)}{\epsilon}
& = F_{\epsilon}^\delta(\psi;B^+_r)-\pi \log \frac{r}{\epsilon}-\pi\log(1+r)
\nonumber
\\
& \quad + \int_{B^+_{r(1+r)} \setminus B_r} \left\vert \nabla \phid_\epsilon \right\vert^2 \d x
+ \frac{1}{2\pi\epsilon} \int_{I_{r(1+r)} \setminus I_r} \sin^2\phid_\epsilon(x_1,0) \ \d x_1
\nonumber
\\
& \quad -2 \int_{B^+_{r(1+r)} \setminus B_r} \delta \cdot \nabla \phid_\epsilon \ \d x.
\end{align}
By~\cite[Lemma~4.14]{IK21} (see Step~1 in the proof of Lemma~4.14):
\begin{equation*}
\lim\limits_{r \rightarrow 0}
\lim\limits_{\epsilon \rightarrow 0}
\left(
\int_{B^+_{r(1+r)} \setminus B_r} \left\vert \nabla \phid_\epsilon \right\vert^2 \d x
+ \frac{1}{2\pi\epsilon} \int_{I_{r(1+r)} \setminus I_r} \sin^2\phid_\epsilon(x_1,0) \ \d x_1
\right)=0.
\end{equation*}
Moreover, by integration by parts,
\begin{equation*}
\lv \int_{B^+_{r(1+r)} \setminus B_r} \delta \cdot \nabla \phid_\epsilon \ \d x \rv
\leqslant \int_{\dr (B_{r(1+r)}^+ \setminus B_r)} \lv \delta \rv \lv \phid_\epsilon \rv \ \d \mathcal{H}^1 \leqslant C_\delta r
\end{equation*}
where~$C_\delta$ is independent of~$r$ and~$\epsilon$, so that~$\lim\limits_{r \rightarrow 0} \liminf\limits_{\epsilon \rightarrow 0} \int_{B^+_{r(1+r)} \setminus B_r} \delta \cdot \nabla \phid_\epsilon \ \d x =0$. Clearly, the left-hand side in~\eqref{eq_lem_4.14_IK21+DMI_1} is greater than
\begin{equation*}
\inf\limits_{\widetilde{\psi}=\phid_\epsilon^\ast \text{ on } \partial B_{r(1+r)}^+ \setminus I_{r(1+r)} }
F_{\epsilon}^\delta(\widetilde{\psi};B_{r(1+r)}^+)-\pi \log \frac{r(1+r)}{\epsilon},
\end{equation*}
so that, taking the infimum of~$F_{\epsilon}^\delta(\psi;B_r^+)$ over functions~$\psi$ such that~$\psi=\phid^\ast$ on~$\partial B_{r}^+ \setminus I_{r}$ in~\eqref{eq_lem_4.14_IK21+DMI_1}, and then taking~$\liminf$ as~$\epsilon \rightarrow 0$ and then~$\liminf$ as~$r \rightarrow 0$, we get~$\widetilde{\gamma}_2 \leqslant \widetilde{\gamma}_1$.
\vspace{.1cm}

\textit{Step 2.} Considering~$r(1-r)$ instead of~$r(1+r)$ and proceeding as in Step~1, we show similarly that~$\gamma_1 \leqslant \gamma_2$.
\vspace{.1cm}

\textit{Step 3.} Combining Steps~1 and~2 with Lemma~\ref{LEM_energy_near_vortices_gamma2}, the conclusion follows.
\end{proof}

\begin{lemma}
\label{LEM_4.15_IK21+DMI}
Let~$\delta = (\delta_1,\delta_2) \in \R^2$ and~$d \in \mathbb{N}^\ast$. For every~$(x_1,x_2) \in \mathbb{R}_+^2=\mathbb{R}\times(0,+\infty)$, set
\begin{equation*}
\phid_d^\ast(x_1,x_2)
= d\arg(x_1+ix_2).
\end{equation*}
Then for every~$r \in (0,1)$ and~$\epsilon \in (0,e^{-1/r^2})$, there exists a function~$\phid_{d,\epsilon} \colon B_r^+ \rightarrow \mathbb{R}$ such that~$\phid_{d,\epsilon}=\phid_d^\ast$ on~$\partial B_r \cap \mathbb{R}_+^2$ and
\begin{equation*}
F_{\epsilon}^\delta(\phid_{d,\epsilon};B_r^+)
\leqslant \pi d \log \frac{r}{\epsilon}+Cd^2 (1+\left\vert \log r \right\vert + \log \left\vert \log \epsilon \right\vert ),
\end{equation*}
where~$C>0$ is independent of~$r$ and~$\epsilon$.
\end{lemma}

\begin{proof}
We use the same arguments as in~\cite[Lemma~4.15]{IK21}.
Let~$r \in (0,1)$ and~$\epsilon \in (0,e^{-1/r^2})$. Set~$a_\epsilon=\frac{1}{\left\vert \log \epsilon \right\vert}$ and~$x_\epsilon^j=ja_\epsilon$ for~$j \in \left\lbrace 1,...,d \right\rbrace$. Consider the interpolation function
\begin{equation*}
f \colon (x_1,x_2) \in \mathbb{R}_+^2 \mapsto \left\lbrace \begin{array}{ll}
1 & \text{ if } \sqrt{x_1^2+x_2^2}<r(1-r),
\\ \frac{r-\sqrt{x_1^2+x_2^2}}{r^2} & \text{ if } r(1-r) \leqslant \sqrt{x_1^2+x_2^2} \leqslant r,
\\ 0 & \text{ if } \sqrt{x_1^2+x_2^2}>r,
\end{array} \right.
\end{equation*}
and set
\begin{equation*}
\phid_{d,\epsilon} \colon (x_1,x_2) \in \mathbb{R}_+^2 \mapsto \sum_{j=1}^d \arg \left( x_1-f(x_1,x_2)x_\epsilon^j + i(x_2+2\pi\epsilon f(x_1,x_2)) \right).
\end{equation*}
We have
\begin{align*}
F_{\epsilon}^\delta(\phid_{d,\epsilon};B_r^+)
& = F_\epsilon^0(\phid_{d,\epsilon};B_r^+)
- 2 \int_{B_r^+} (\delta_1 \partial_1\phid_{d,\epsilon} +\delta_2 \partial_2\phid_{d,\epsilon}) \ \d x.
\end{align*}
From~\cite[Lemma~4.15]{IK21}, we have
\begin{equation*}
F_\epsilon^0(\phid_{d,\epsilon};B_r^+)
\leqslant \pi d \log \frac{r}{\epsilon}+Cd^2 (1+\left\vert \log r \right\vert + \log \left\vert \log \epsilon \right\vert ),
\end{equation*}
where~$C>0$ is independent of~$\epsilon$ and~$r$. Moreover, using Green's formula and~$\lv \phid_{d,\epsilon} \rv \leqslant \pi d$,
\begin{equation*}
\left\vert \int_{B_r^+} (\delta_1 \partial_1\phid_{d,\epsilon} +\delta_2 \partial_2\phid_{d,\epsilon}) \ \d x \right\vert
= \left\vert \int_{\partial B_r^+} \phid_{d,\epsilon} (\delta_1 \nu_1 +\delta_2\nu_2) \d \mathcal{H}^1 \right\vert
\leqslant Cd \left( \lv \delta_1 \rv + \lv \delta_2 \rv \right),
\end{equation*}
for every~$r \in (0,1)$, where $\nu=(\nu_1,\nu_2)$ is the outer unit normal vector on $\partial B_r^+$. The conclusion follows.
\end{proof}

\end{appendices}

\bibliographystyle{plain}

\end{document}